\documentclass[a4paper,reqno]{amsart}

\calclayout

\usepackage{quiver}
\usepackage{csquotes}
\usepackage{mathtools}
\usepackage{scalerel}
\usepackage{amssymb}
\usepackage{wasysym}
\usepackage{stmaryrd}
\usepackage{mathrsfs}
\usepackage{amsthm}
\usepackage{enumitem}
\usepackage{graphicx}
\usepackage{tikz-cd}
\usetikzlibrary{calc}
\usepackage{wrapfig}
\usepackage{DotArrow}
\usepackage[%
backend=biber, %
style=alphabetic, %
maxnames=99, %
maxalphanames=4, %
natbib=false, %
isbn=false, 
doi=false, 
url=false, 
eprint=true, 
giveninits=true, 
backref=false, 
]{biblatex} 
\AtEveryBibitem{%
	\ifentrytype{misc}{\renewbibmacro*{date}{}}{}%
}
\AtEveryBibitem{%
	\ifentrytype{unpublished}{}{}%
} \renewbibmacro{in:}{} 
\usepackage{doi}%
\usepackage{hyperref}
\hypersetup
{
	colorlinks,
	linkcolor = {red},
	citecolor = {red},
	urlcolor  = {red}
}

\usepackage{adjustbox}

\usepackage[capitalise, noabbrev]{cleveref}

\DeclareMathOperator{\Hom}{Hom}
\DeclareMathOperator{\REnd}{REnd}

\DeclareMathOperator{\Ext}{Ext}
\DeclareMathOperator{\id}{id}
\DeclareMathOperator{\Mod}{mod}
\DeclareMathOperator{\End}{End}
\DeclareMathOperator{\Ker}{ker}

\DeclareMathOperator{\thick}{thick} 
\DeclareMathOperator{\cone}{cone}
\DeclareMathOperator{\cocone}{cocone}

\DeclareMathOperator{\hker}{hker}
\DeclareMathOperator{\hcoker}{hcoker}
\DeclareMathOperator{\Filt}{Filt}

\DeclareMathOperator{\add}{add}
\DeclareMathOperator{\smd}{smd}

\newcommand{\Dproj}{\D\mathrm{Proj}}

\newcommand{\Cat}[1]{\mathcal{#1}}
\newcommand{\D}{\Cat{D}}
\newcommand{\T}{\Cat{T}}
\newcommand{\A}{\Cat{A}}
\newcommand{\B}{\Cat{B}}
\newcommand{\C}{\Cat{C}}
\newcommand{\E}{\Cat{E}}
\newcommand{\F}{\Cat{F}}
\newcommand{\W}{\Cat{W}}
\newcommand{\Sc}{\Cat{S}}
\newcommand{\ob}{\mathrm{ob}}
\newcommand{\Hqe}{\mathrm{Hqe}}

\newcommand{\real}{\mathcal{R}}
\newcommand{\repr}[1]{{{#1}}^{\wedge}}

\newcommand{\Ht}{\mathcal{H}_{3t}}
\newcommand{\Hc}[2]{\Cat{H}_{{#1}}^{{#2}}}
\newcommand{\He}[2]{\mathbb{H}_{{#1}}^{{#2}}}
\newcommand{\Proj}{\Cat{P}}
\newcommand{\Extr}{\mathfrak{C}}

\newcommand{\extrs}{\mathfrak{s}}
\newcommand{\Double}[1]{\mathbb{#1}}
\newcommand{\NN}{\Double{N}}
\newcommand{\ZZ}{\Double{Z}}
\newcommand{\EE}{\Double{E}}

\newcommand{\DP}{\Double{P}}
\newcommand{\DL}{\Double{L}}

\newcommand{\Mat}[1]{\mathrm{Mat}(\ZZ{#1})}
\newcommand{\KoszulDual}[1]{{#1}^{!}}
\newcommand{\per}{\mathrm{per}} 
\newcommand{\fd}{\D^{\mathrm{fd}}}
\newcommand{\KST}{\D}
\newcommand{\definef}[1]{\emph{{#1}}}
\newcommand{\define}[1]{{#1}}
\newcommand{\heart}{\heartsuit}
\newcommand{\tw}{\mathrm{tw}}
\newcommand{\pretr}{\mathrm{tw}}
\newcommand{\tr}{\mathrm{tr}}
\newcommand{\trunc}{\mathrm{t}}
\newcommand{\fib}{\mathrm{fib}}
\newcommand{\cofib}{\mathrm{cofib}}

\newcommand{\SW}{\mathcal{SW}}
\newcommand{\colim}{\mathrm{colim}}
\newcommand{\St}{\mathcal{H}^{\mathrm{st}}_{\infty}}
\newcommand{\infl}{\rightarrowtail}
\newcommand{\defl}{\twoheadrightarrow}
\newcommand{\ext}{\dashrightarrow}
\newcommand{\dis}[1]{{#1}^{\mathrm{discr.}}}
\newcommand{\Loewy}{\ell\ell}
\newcommand{\BigLoewy}{\mathcal{LL}}
\newcommand{\height}{\mathrm{ht}}

\newcommand{\topd}{\mathrm{coh.deg}}
\newcommand{\Simples}{\mathfrak{S}}
\newcommand{\topf}{\mathrm{top}}

\newcommand{\rad}{\mathrm{rad}}
\newcommand{\leaveout}[1]{}
\newcommand{\Einfty}{\mathcal{E}}

\newcommand\scalemath[2]{\scalebox{#1}{\mbox{\ensuremath{\displaystyle #2}}}}

\theoremstyle{plain} 
 
\newtheorem{theorem}{Theorem}[section]
\newtheorem{prop}[theorem]{Proposition}
\newtheorem{lemma}[theorem]{Lemma}
\newtheorem{coro}[theorem]{Corollary} \theoremstyle{definition}

\newtheorem{defi}[theorem]{Definition}
 \theoremstyle{remark}
\newtheorem{rema}[theorem]{Remark}
\newtheorem{ex}[theorem]{Example}
\newtheorem{nota}[theorem]{Notation}

\setlist[enumerate,1]{label=(\arabic*)}
\setlist[enumerate,2]{label=(\alph*)}

\makeatletter
\renewcommand{\l@subsection}{\@tocline{2}{0pt}{2.5pc}{5pc}{}} 
\makeatother


\title{Complicial simple-minded collections}

\author[M.~Plogmann]{Marvin Plogmann}%
\address{%
	Mathematisches Institut, %
	Universität zu Köln, %
	Weyertal 86-90, %
	50931 Köln, %
	Germany%
}%
\email{plogmann@math.uni-koeln.de}%
\urladdr{https://sites.google.com/view/marvinplogmann}%



\begin{document}
	
	\begin{abstract}
		We consider the problem of characterizing derived endomorphism algebras of simple objects in length categories up to quasi-isomorphism. We give such a characterization for module categories, abelian categories, exact categories, as well as, for certain differential graded analogues of them. It turns out that the property of being $d$-complicial ($d\geq 1$), in the sense of Lurie, of the involved simple-minded collections plays a central role. We also explain how this characterization can be interpreted as a coherent generation property for any minimal $A_{\infty}$-model of the derived endomorphism algebra. Along the way, we propose a notion of length exact differential graded categories and explain how they relate to length abelian $d$-truncated differential graded categories, generalizing results of Enomoto. 
	\end{abstract}
	
	\maketitle
	
	\setcounter{tocdepth}{1}
	\tableofcontents
	
	\section{Introduction}
	
	Let $k$ be a field. It is well-known that a module category or, more generally an abelian (or exact) length category $\A$ can be recovered from the derived endomorphism algebra of the simple objects in the sense that
	\[\D^b(\A)\simeq \per(\REnd_{\A}(S)),\]
	where $S$ is the direct sum of all simple objects. Moreover, one can describe the image of $\A$ under the above functor as the twisted modules over $\REnd_{\A}(S)$. This motivates the question of characterizing differential graded (=dg) endomorphism algebras of simple objects in such length categories.
	
	The study of such correspondences is known as Koszul duality. Koszul duality originated in the study of Koszul algebras, which appear in representation theory of algebras, Lie theory, algebraic geometry, and other areas of mathematics, see for example \cite{BGS88, BGS96}. By considering derived endomorphism algebras (instead of only their cohomology) Koszul duality has been generalized to dg and $A_{\infty}$-algebras, see \cite{Kel94,LH03,LPWZ08,Av13,VdB15,F25}. Koszul duality is also central in the one-to-one correspondence between silting collections and simple-minded collections, see \cite{KY14,SHYD19,Z23,F24,Bon24}.
	
	For a locally finite (cohomologically) connective dg algebra $A$, that is $H^{i}(A)$ is finite-dimensional for all $i\in\ZZ$ and $H^{i}(A)=0$ for $i>0$, we define $\define{S_A}\coloneqq\topf(H^0(A))$. The \definef{Koszul dual of $A$} is defined as \[\define{\KoszulDual{A}}\coloneqq \REnd_A(S_A).\] 
	Dually, for a locally finite coconnective dg algebra $E$, that is $H^{i}(E)$ is finite-dimensional for all $i\in\ZZ$ and $H^{i}(E)=0$ for $i<0$, and such that $H^0(E)$ is semisimple, we define $S_E\coloneqq H^0(E)$. The \definef{Koszul dual of $E$} is defined as 
	\[\define{\KoszulDual{E}}\coloneqq \REnd_E(S_E).\] 
	In this article, we study Koszul duality for proper connective dg algebras $A$, that is the graded vector space $H^{\ast}(A)$ is finite-dimensional and concentrated in non-positive degrees. In particular, we are interested to characterize the property that the cohomology of $A$ is concentrated in degree $0$ or, more generally, in degrees $1-d$ to $0$ for some $d\geq 1$, in terms of $\KoszulDual{A}$.
	
	One central object of interest for this purpose are the extended hearts associated to a proper connective dg algebra or, more generally, a $t$-structure on a pretriangulated dg category. For $d\in\NN_{>0}$ and a pretriangulated dg category $\C$ equipped with a $t$-structure $(\C^{\leq 0},\C^{\geq 0})$, we define the \definef{$d$-extended heart} by
	\[\define{\Hc{\C}{d}}:=\define{\C^{(-d,0]}}\coloneqq \tau^{\leq 0}\left(\C^{\leq 0}\cap \C^{>-d}\right),\]
	where $\tau^{\leq 0}$ denotes the \definef{connective cover} (see section \ref{section:Terminology}). Observe that $\Hc{\C}{d}$ is an extension-closed subcategory of a pretriangulated dg category and hence an exact dg category in the sense of \cite{Chen23}. For a proper connective dg algebra $A$ we can consider $\C\coloneqq\fd_{dg}(A)$, the canonical dg enhancement of
	\[\define{\fd(A)}\coloneqq \left\{M\in\D(A): \dim_k H^*(M)<\infty \right\}\subseteq\D(A)\]
	equipped with the standard $t$-structure defined by 
	\[\fd(A)^{\leq 0}\coloneqq\left\{M:H^{i}(M)=0 \;\; \forall i\geq 0\right\} \text{ and } \fd(A)^{\geq 0}\coloneqq\left\{M:H^{i}(M)=0 \;\; \forall i\leq 0\right\}.\]
	In that case, we also use the notation $\define{\Hc{A}{d}}$ to refer to the $d$-extended heart. The starting point is the basic observation that $\Hc{A}{d}$ is $1$-generated in $\D^b(\Hc{A}{d})$ in the following sense; here $\D^b(\Hc{A}{d})$ denotes the bounded derived category of the exact dg category $\Hc{A}{d}$ in the sense of \cite[Thm.~6.1]{Chen23}.
	
	\begin{defi}
		Let $\T$ be a triangulated category and $\Sc\subseteq \T$ a subcategory. We say that $\Sc$ is \definef{$1$-generated} in $\T$ if for all $X,Y\in\Sc$ and morphisms $f:X\to Y[n]$ in $\T$ for an integer $n\geq 1$, there exist
		\begin{itemize}
			\item $X=X_0,X_1,\ldots,X_{n-1},X_n=Y\in\Sc$,
			\item and morphisms $f_i:X_i[i]\to X_{i+1}[i+1]$ in $\T$,
		\end{itemize}
		such that $f=f_{m-1}\cdots f_1f_0$.
	\end{defi}
	
	Using this observation, we prove the following characterization of dg algebras arising as $\REnd_{A}(S_A)$ for a finite-dimensional algebra $A$ or more generally a proper connective dg algebra $A$. Here, for a coconnective locally-finite dg algebra $E$ with $H^0(E)$ semisimple, we denote by $\per(E)^{(-d,0]}\subseteq \per(E)$ the $d$-extended heart of the $t$-structure associated to the simple-minded collection consisting of the direct summands of $E$.
	
	\begin{theorem}[\cref{theorem:KoszulDperfect}]\label{theorem:KoszulDperfectIntro}
		Let $k$ be a perfect field and $d\in\NN_{>0}$. Koszul duality yields a bijective correspondence between quasi-isomorphism classes of the following two classes of dg algebras:
		\begin{enumerate}
			\item Connective proper dg algebras $A$ such that $H^*(A)$ is concentrated in degrees $(-d,0]$.
			\item Coconnective locally-finite homologically smooth dg algebras $E$ such that $H^0(E)$ is semisimple and $\per(E)^{(-d,0]}\subseteq\per(E)$ is $1$-generated.
		\end{enumerate}
	\end{theorem}
	
	Consider now a finite-dimensional algebra $\Lambda$. Using the equivalence $\D^b(\Lambda)\simeq \per(\KoszulDual{\Lambda})$ and the observation that $\Mod(\Lambda)$ is $1$-generated in $\D^b(\Lambda)$, it was proven in \cite[Sec.~2.2, Prop.~1(b)]{Kel02} that the minimal $A_{\infty}$-model of the dg algebra $\KoszulDual{\Lambda}$ is generated in cohomological degrees $0$ and $1$ as an $A_{\infty}$-algebra. In \cref{subsection:CohGen} we recall the argument to explain how the condition that $\per(E)^{(-d,0]}\subseteq \per(E)$ is $1$-generated can be rephrased in terms of a condition on the minimal $A_{\infty}$-model $\E$ of the dg algebra $E$. In short, we require the $A_{\infty}$-algebra $\E$ to be coherently generated in cohomological degrees $0,1,\ldots,d$ in a specific way that is explained in \cref{subsection:CohGen}.
	
	The main technical tool to prove \cref{theorem:KoszulDperfectIntro} are dg realization functors. We prove the following generalization of conditions for realization functors to be fully-faithful (see e.g. \cite[Thm.~3.11]{PV18}).
	
	\begin{theorem}[\cref{GeneralisedRealFunctor}]
		Let $\C$ be a pretriangulated dg category and $\Cat{B}\subseteq \C$ a full extension-closed dg subcategory. Set $\A\coloneqq\tau^{\leq 0}\Cat{B}$. Then, the universal property of $\D^b_{dg}(\A)$ yields an exact dg functor $\real:\D^b_{dg}(\A)\to \C$. The following statements hold:
		\begin{enumerate}
			\item The realization functor $\real$ induces isomorphisms $\Hom_{D^b(\A)}(X,Y[n])\overset{\cong}{\to} \Hom_{H^0(C)}(X,Y[n])$ for $X,Y\in H^0(\A)$ and $n\leq 1$.
			\item The following are equivalent:
			\begin{enumerate}
				\item $\real$ is quasi fully-faithful.
				\item $\real$ induces isomorphisms $\Hom_{D^b(\A)}(X,Y[n])\overset{\cong}{\to} \Hom_{H^0(\C)}(X,Y[n])$ for all $n\geq 2$ and $X,Y\in H^0(\A)$.
				\item $\real$ induces epimorphisms $\Hom_{D^b(\A)}(X,Y[n])\twoheadrightarrow \Hom_{H^0(\C)}(X,Y[n])$ for all $n\geq 2$ and $X,Y\in H^0(\A)$.
				\item (Ef) $\Hom_{H^0(\C)}(-,Y[n])$ is weakly effaceable for every $Y \in H^0(\A)$ and $n\geq 2$, that is for all $X \in H^0(\A)$ and $f:X\to Y[n]$ in $H^0(\C)$ there exists a deflation $g:Z\to X$ in $H^0(\A)$ such that $fg=0$.
				\item (CoEf) $\Hom_{H^0(\C)}(X,-[n])$ is weakly effaceable for every $X \in H^0(\A)$ and $n\geq 2$, that is for all $Y \in H^0(\A)$ and $f:X\to Y[n]$ in $H^0(\C)$ there exists an inflation $g:Y\to Z$ in $H^0(\A)$ such that $g[n]\circ f=0$.
				\item $H^0(\A)$ is 1-generated in $H^0(\C)$.
			\end{enumerate}
			If the above equivalent conditions are satisfied, the essential image of $\real$ is $\mathrm{tria}(H^0(\A))$, the smallest triangulated subcategory of $H^0(\C)$ containing the objects of $H^0(\A)=H^0(\B)$.
		\end{enumerate}
	\end{theorem}
	
	As a consequence of the above we get the following relation to $(d-1)$-complicial $t$-structures.
	
	\begin{defi}
		Let $\T$ be a triangulated category. Following \cite[Def.~C.5.3.1.]{Lurie18}, a $t$-structure $t$ on $\T$ is \definef{$n$-complicial} if for every object $X\in \T^{\leq 0}$ there exists $X'\in \T^{[-n,0]}$ and a morphism $f\in \Hom_{H^0(\T)}(X',X)$, such that $H^0(f):H^0(X')\to H^0(X)$ is an epimorphism in $\heart_t$. We call $t$ \definef{strictly $n$-complicial} if $n$ is minimal with this property.
	\end{defi}
	
	The following theorem is a variant of \cite[Thm.~6.3.2]{Ste23} for dg categories.
	
	\begin{theorem}[\cref{compliciality}]
		Let $\C$ be a pretriangulated dg category with a bounded $t$-structure. Let $d\geq 1$. The following are equivalent:
		\begin{enumerate}
			\item The $t$-structure on $\C$ is $(d-1)$-complicial.
			\item The realization functor $\real:\D^b_{dg}(\Hc{\C}{d})\to \C$ is a quasi-equivalence.
		\end{enumerate}
		Similarly, the following are equivalent:
		\begin{enumerate}
			\item The $t$-structure on $\C$ is strictly $(d-1)$-complicial.
			\item The realization functor $\D^b_{dg}(\Hc{\C}{d})\to \C$ is a quasi-equivalence, but the realization functor $\D^b_{dg}(\Hc{\C}{d-1})\to \C$ is not.
		\end{enumerate}
	\end{theorem}
	
	As a further application we obtain the following recognition theorem for $\fd(A)$ with $A$ a proper connective dg algebra such that $H^{*}A$ is concentrated in degrees between $1-d$ and $0$ which is inspired by \cite[Prop.~1.3.3.7]{Lurie17}. 
	
	\begin{theorem}[\cref{theorem:RecognitionTheorem}]
		Let $\T=H^0(\C)$ be an algebraic $\Hom$-finite triangulated category, with $\C$ a locally-finite pretriangulated dg category such that the following conditions are satisfied:
		\begin{enumerate}
			\item There exists a collection $\DL\coloneqq \{L_1,\ldots,L_n\}\subseteq \T$ which is a simple-minded collection in $\thick(\DL)$. In particular, $\DL$ yields a bounded $t$-structure on $\thick(\DL)$ which we denote by $(\thick(\DL)^{\leq 0},\thick(\DL)^{\geq 0})$.
			\item There exists $d\in\NN_{\geq 1}$ and indecomposable objects $\{P_1,\ldots,P_n\}\subseteq\T$
			such that for all $1\leq i\leq n$ there exists a triangle
			\[Y_i\longrightarrow P_i \longrightarrow L_i \longrightarrow Y_i[1]\]
			in $\T$ such that 
			\[Y_i\in\thick(\DL)^{(-d,0]}=\Filt(\DL)[d-1]\ast\ldots\ast\Filt(\DL)[1]\ast\Filt(\DL).\]
			\item For $P\coloneqq \oplus_{i=1}^n P_i$ and $L\coloneqq\oplus_{i=1}^n L_i$ it holds $\Hom_{\T}(P,L[k])=0$ for all $k\geq 1$.
		\end{enumerate} 
		Then, the realization functor 
		\[\real:\D^b_{dg}(\thick(\DL)^{(-d,0]})\stackrel{\simeq}{\longrightarrow} \thick(\DL)\subseteq \C\]
		is a $t$-exact quasi-equivalence. Moreover, for $A\coloneqq\C(P,P)$ there is a $t$-exact quasi-equivalence
		\[\fd_{dg}(A)\stackrel{\simeq}{\longrightarrow}\thick(\DL)\subseteq \C\]
		which sends $A$ to $P$. In particular, $A\in\fd(A)^{(-d,0]}$ is proper connective, the equivalence restricts to $\per(A)\simeq\thick(P)$ and if $\DL\subseteq \T$ is full, that is $\thick(\DL)=\T$, it holds $\fd(A)\simeq \T$.
	\end{theorem}
	
	We remark that for $d=1$ this yields a recognition theorem for $\D^b(\Lambda)$ for $\Lambda$ a finite-dimensional algebra. Indeed, in that case $A$ is in the heart, hence $H^*(A)$ is concentrated in degree $0$. We also remark that for $d=1$ the above theorem can be deduced from \cite{AMY19} (see \cref{rema:ExplanationAMY} for details).
	
	In \cref{sec:LengthExact} we propose a notion of length exact $d$-categories, $n$-semibricks and wide $n$-subcategories of such exact $d$-categories (for $1\leq n\leq d$). We prove the following statement, which is a generalization of \cite[Thm.~1.2]{Rin76} and \cite[Thm.~2.5]{Eno21}.
	
	\begin{theorem}[\cref{theorem:Enomoto}]
		Let $\E$ be a length exact $d$-category and $1\leq n\leq d$. There are mutually inverse bijections between the following two classes.
		\begin{enumerate}
			\item The class of $n$-semibricks in $\E$.
			\item The class of length wide $n$-subcategories of $\E$.
		\end{enumerate}
	\end{theorem}
	
	We obtain the following corollary.
	
	\begin{coro}[{\cref{coro:AbelianSchur}}]
		Let $\mathcal{E}$ be a length exact $d$-category. Then the following are equivalent:
		\begin{enumerate}
			\item $\mathcal{E}$ is an abelian $d$-truncated dg category in the sense of \cite[Def.~3.12]{Moc25},
			\item The collection of representatives of isomorphism classes of simple objects in $\E$ forms a $d$-semibrick.
		\end{enumerate}
	\end{coro}
	
	We consider the property of having enough projectives for length abelian $d$-truncated dg categories. The following result summarizes our findings.
	
	\begin{theorem}[\cref{theorem:SummaryAbelian}]
		Let $k$ be a perfect field. The following are in bijection:
		\begin{enumerate}
			\item\label{item:SummaryAbelian1Intro} Smooth locally-finite pretriangulated dg categories $\C$ equipped with a strictly $(d-1)$-complicial bounded $t$-structure $t$ such that $\heart_t$ has enough projectives, up to $t$-exact quasi-equivalence.  
			\item\label{item:SummaryAbelian2Intro} Length abelian $d$-truncated dg categories $\A$ with finitely many simples which are locally finite and $\Ext$-finite such that ${\height(\A)<\infty}$, up to exact quasi-equivalence.  
			\item\label{item:SummaryAbelian3Intro} Proper basic dg algebras, that is $H^0(A)$ is basic, with cohomology strictly concentrated in degrees $(-d,0]$, up to quasi-isomorphism. 
			\item\label{item:SummaryAbelian4Intro} Smooth locally-finite coconnective basic dg algebras, that is $H^0(E)$ is basic, such that $H^0(E)$ is semisimple and $\per(E)^{(-d,0]}\subseteq \per(E)$ is $1$-generated, up to quasi-isomorphism.
		\end{enumerate}
		Denoting by $L$ the direct sum of simple objects in $\C^{(-d,0]}$ or $\A$ and by $P$ the minimal right approximation of $L$ by projective objects in $\C^{(-d,0]}$ or $\A$, respectively, the bijections are given as follows (notice that the diagram commutes):
		
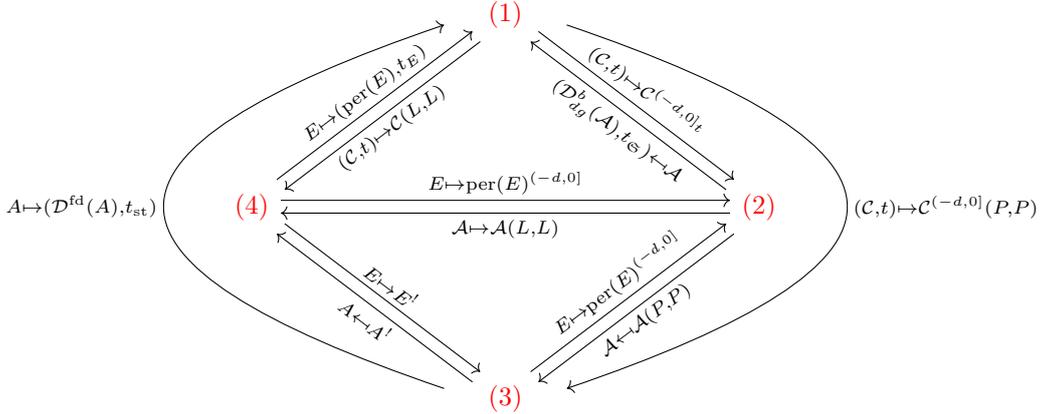
\[\begin{tikzcd}
			&&& {\ref{item:SummaryAbelian1Intro}} \\
			\\
			\\
			{\ref{item:SummaryAbelian4Intro}} &&&&&& {\ref{item:SummaryAbelian2Intro}} \\
			\\
			\\
			&&& {\ref{item:SummaryAbelian3Intro}}
			\arrow["{(\C,t)\mapsto \C(L,L)}"', sloped, shift left, from=1-4, to=4-1]
			\arrow["{(\C,t)\mapsto \C^{(-d,0]_t}}", sloped, shift left, from=1-4, to=4-7]
			\arrow["{(\C,t)\mapsto\C^{(-d,0]}(P,P)}", shift left=5, curve={height=-140pt}, from=1-4, to=7-4]
			\arrow["{E\mapsto(\per(E),t_E)}", sloped, shift left, from=4-1, to=1-4]
			\arrow["{E \mapsto \per(E)^{(-d,0]}}", shift left, from=4-1, to=4-7]
			\arrow["{E \mapsto \KoszulDual{E}}", sloped, shift left=1, from=4-1, to=7-4]
			\arrow["{\small{(\D^b_{dg}(\A),t_{\Simples})\mapsfrom \A}}"', sloped, shift left=1, from=4-7, to=1-4]
			\arrow["{\A\mapsto \A(L,L)}", shift left, from=4-7, to=4-1]
			\arrow["{\A\mapsfrom\A(P,P)}"', sloped, shift left, from=4-7, to=7-4]
			\arrow["{A\mapsto (\fd(A),t_{\mathrm{st}})}", shift left=5, curve={height=-140pt}, from=7-4, to=1-4]
			\arrow["{A\mapsfrom \KoszulDual{A}}"', sloped, shift left, from=7-4, to=4-1]
			\arrow["{E\mapsto\per(E)^{(-d,0]}}", sloped, shift left=1, from=7-4, to=4-7]
		\end{tikzcd}\]
	\end{theorem}
	
	Finally, we prove that the above results are special cases of a Koszul duality statement between certain length exact $d$-categories and dg categories which axiomatize dg endomorphism categories of the simple objects in exact $d$-categories (see \cref{sec:KDExact} for the definitions). 
	
	\begin{theorem}[{\cref{KoszulExact}}]
			There are mutually inverse bijections between the following: 
		\begin{enumerate}
			\item\label{item:KoszulExact1Intro} Coconnective dg categories $E$ such that $\E\coloneqq\tw(E)^{(-d,0]}\subseteq \tw(E)$ is $1$-generated,  $\dis{\E}=\Filt(\ob(E))$ and $E$ is admissible Schur, up to quasi-equivalence.
			\item\label{item:KoszulExact2Intro} Length exact $d$-categories, up to exact quasi-equivalence.
		\end{enumerate}
		They restrict to bijections between the following:
		\begin{enumerate}
			\item[(1')] Coconnective dg categories $E$ such that $\tw(E)^{(-d,0]}\subseteq \tw(E)$ is $1$-generated, and $H^0(E)(x,y)$ is a division algebra for all $x,y\in\ob(E)$, up to quasi-equivalence.
			\item[(2')] Length abelian $d$-truncated dg categories, up to exact quasi-equivalence.
		\end{enumerate}
		If $k$ is a perfect field, they restrict further to bijections between the following:
		\begin{enumerate}
			\item[(1'')] Coconnective locally-finite smooth dg categories $E$ with $|\ob(E)|<\infty$ such that $\tw(E)^{(-d,0]}\subseteq \tw(E)$ is $1$-generated, and $H^0(E)(x,y)$ is a division algebra for all $x,y\in\ob(E)$, up to quasi-equivalence.
			\item[(2'')] Locally-finite, $\Ext$-finite length abelian $d$-truncated dg categories with finitely many simples and enough projectives, up to exact quasi-equivalence.
		\end{enumerate}
		
		The bijections are given as follows:
		\begin{align*}
			F:\ref{item:KoszulExact1Intro} &\longrightarrow \ref{item:KoszulExact2Intro} & G:\ref{item:KoszulExact2Intro} &\longrightarrow \ref{item:KoszulExact1Intro}\\
			E &\longmapsto \tw(E)^{(-d,0]} & \E &\longmapsto \D^b_{dg}(\E)(\DL_{\E},\DL_{\E})
		\end{align*}
		Here, $\DL_{\E}$ denotes the collection of simple objects in $\E$. Moreover, if $\E$ and $E$ correspond to each other under the above bijections, then there is a quasi-equivalence \[\tw(E)\stackrel{\simeq}{\longrightarrow}\D^b_{dg}(\E).\]
	\end{theorem}
	
	\subsection{Context and relation to other work}
	Recently, there has been growing interest in extended module categories of finite-dimensional algebras or, more generally, extended hearts of $t$-structures in (algebraic) triangulated categories.
	
	In the case of a finite-dimensional algebra $\Lambda$, the tilting theory of the extended module category $\Hc{\Lambda}{d}$ was studied in \cite{Gu24,Z25,GZ25}. In \cite[Thm.~4.1]{Gu24}, a bijection between $(d+1)$-term silting complexes, certain torsion pairs in $\Hc{\Lambda}{d}$ and certain cotorsion pairs in $\mathrm{K}^{[-d,0]}(\mathrm{proj}(\Lambda))$ is proven. In \cite[Thm.~3.12]{Z25} it is proven that $\Hc{\Lambda}{d}$ has almost-split conflations and the former bijections are extended to include basic $\tau$-tilting pairs in $\Hc{\Lambda}{d}$ in \cite[Thm.~4.7]{Z25}. Finally, in \cite[Thm.~7.3, Thm~7.11]{GZ25}, the bijections have been further extended to include $(d+1)$-term simple-minded collections in $\D^b(\Mod\Lambda)$ as well as left and right-finite semibricks and wide subcategories in $\Hc{\Lambda}{d}$. In particular, \cref{theorem:Enomoto} can be seen as a generalization of the bijections between semibricks and wide subcategories in $\Hc{\Lambda}{d}$, which corresponds to the case $n=1$. 
	
	The extended hearts of proper connective dg algebras $A$ were first studied in \cite{Moc25}, where it is proven that $\Hc{A}{d}$ is an abelian $d$-truncated dg category (see \cite[Thm.~3.44]{Moc25}). In \cref{prop:EnoughProj} we prove that any locally-finite, $\Ext$-finite, length abelian $d$-truncated dg category $\A$ with finitely many simples and enough projectives is of this form. In \cite{MP25} it was proven that $\Hc{A}{d}$ has almost split conflations and satisfies the first Brauer--Thrall Conjecture. Finally, in \cite{Moc26} a higher-dimensional version of the Auslander correspondence for proper connective dg algebras is proven. All these results indicate that $\Hc{A}{d}$ seems to behave very similarly to the module category of a finite-dimensional algebra.
	 
	In \cite[Thm.~3.44]{Moc25} it is shown that extended hearts of $t$-structures on pretriangulated dg categories are $d$-truncated dg categories. In \cref{prop:CharAb} we show that, analogous to the case $d=1$, all abelian $d$-truncated dg categories are of this form. We also remark that, in \cite{WZ25}, part of the tilting theory for $\Hc{\Lambda}{d}$ is generalized to extended hearts of silting $t$-structures on (not necessarily algebraic) triangulated categories.
	 
	This article continues this line of research but focuses on the foundational aspects of the relationship between length exact $d$-categories, length abelian $d$-truncated dg categories and extended hearts of proper connective dg algebras as well as their relationship to Koszul duality.

	\section{Preliminaries}
	
	\subsection{Terminology and notation}\label{section:Terminology}
	We work over an arbitrary field $k$. Unless mentioned otherwise, all categories are assumed to be $k$-linear. Modules over will always be right modules. We use freely the language of dg categories and their derived (dg) categories; our standard references are \cite{Kel94,Kel06}. In particular, for a dg category $\A$ we denote by \define{$\D_{dg}(\A)$} its \definef{derived dg category} and by 
	\[\define{\D(\A)}\coloneqq H^0(\D_{dg}(\A))\]
	its \definef{derived category}. For an object $A\in\A$ we denote by 
	\[\define{\repr{A}}\coloneqq \A(-,A)\in\D(\A)\] 
	the corresponding representable dg $\A$-module. We write 
	\[\define{\per(\A)}\coloneqq \thick(\{\repr{A}\colon A\in\A\})\subseteq \D(\A)\] 
	for the full subcategory of \definef{perfect dg $\A$-modules} and 
	\[\define{\fd(\A)}\coloneqq \left\{M\in\D(\A): \dim_k H^*(M)<\infty \right\}\subseteq\D(\A)\]
	for the \definef{perfectly valued derived category of $\A$}. We denote by $\define{\per_{dg}(\A)}$ and $\define{\fd_{dg}(\A)}$ the canonical dg enhancements of $\per(\A)$ and $\fd(\A)$, respectively. For a dg category $\A$, we denote by $\define{\tau^{\leq 0}\A}$ its \definef{connective cover}, that is the dg category with $\ob(\tau^{\leq 0}\A)\coloneqq \ob(\A)$ and, for $x,y\in\tau^{\leq 0}\A$ the dg vector space of morphisms 
	\[\tau^{\leq 0}\A(x,y)\coloneqq \left( \cdots \stackrel{d^{-3}_{\A}}{\longrightarrow} \A(x,y)^{-2} \stackrel{d^{-2}_{\A}}{\longrightarrow} \A(x,y)^{-1} \stackrel{d^{-1}_{\A}}{\longrightarrow} \Ker(d^0_{\A})\to 0\to \cdots\right),\]
	where $\Ker(d^0_{\A})$ is in degree $0$. We call $\A$ \definef{connective} if $H^*(\A)$ is concentrated in non-positive degrees and \definef{coconnective} if $H^*(\A)$ is concentrated in non-negative degrees. Observe that $\A$ is connective if and only if the canonical functor $\tau^{\leq 0}\A\to \A$ is a quasi-isomorphism. We say that $\A$ is \definef{locally-finite} if $H^{i}(\A)(x,y)$ is finite dimensional for all $x,y\in\A$ and all $i\in\ZZ$. If the graded vector space $H^{*}(\A)(x,y)$ is finite dimensional for all $x,y\in\A$, we call $\A$ \definef{proper}.

	\subsection{$t$-Structures, simple-minded collections and silting collections}\label{t-str-ch}
	We fix a Krull--Schmidt triangulated category $\KST$. In this section, we recall several different descriptions of the bounded $t$-structures associated to simple-minded and silting collections in $\KST$. We first introduce some notation.
	
	Let $\A,\Cat{B}\subseteq \C$ be subcategories of a category $\C$. We write \define{$\A\perp\Cat{B}$} if $\Hom_{\C}(a,b)=0$ for all $a\in\A,b\in\Cat{B}$. 
	
	For a collection of objects $\C$ in $\D$, we write 
	\[\define{\C^{\perp_{>0}}}=\{D\in\D:\Hom_{\D}(C,D[m])\; \forall m>0, C\in \C\}\] 
	and define \define{$\C^{\perp_{<0}}$}, \define{$\C^{\perp_{\geq 0}}$}, \define{$\C^{\perp_{\leq 0}}$} analogously.
	
	The following notion was introduced in \cite{BBD82}.
	
	\begin{defi}
		A \definef{$t$-structure} on a triangulated category $\D$ is a pair $t=\left(\D^{\leq0},\D^{\geq 0}\right)$ called \definef{aisle} and \definef{coaisle} of (strict) full subcategories such that 
		\begin{itemize}
			\item $\D^{\leq 0}[1]\subseteq \D^{\leq 0}$ and $\D^{\geq 0}[-1]\subseteq \D^{\geq 0}$,
			\item $\D^{\leq 0}\perp (\D^{\geq 0}[-1])$,
			\item for all $X\in\D$ there exists a triangle (called \definef{$t$-decomposition} of $X$) 
			\[\trunc^{<0}X\to X \to \trunc^{\geq 0}X\to \trunc^{<0}X[1]\]
			with $\trunc^{<0}X \in \D^{\leq 0}[1]$ and $\trunc^{\geq 0}X\in\D^{\geq 0}$.
		\end{itemize}
		The full subcategory 
		\[\heart\coloneqq\heart_t\coloneqq\D^{\leq 0}\cap \D^{\geq 0}\] 
		is called the \definef{heart} of $t$. For $n\in\ZZ$ we write 
		\[\D^{\leq n}:=\D^{\leq 0}[-n], \quad \D^{\geq n}:=\D^{\geq 0}[-n], \quad \D^{<n}:=\D^{\leq 0}[-(n-1)], \quad \D^{>n}:=\D^{\geq 0}[-(n+1)].\]  
		We call $t$ \definef{bounded} if 
		\[\D=(\cup_{n\in\ZZ}\D^{\leq n})\cap (\cup_{n\in\ZZ}\D^{\geq n}).\]
	\end{defi} 
	
	Recall that $\heart_t$ is an abelian category. Moreover, $\trunc^{<0}$ and $\trunc^{\geq 0}$ or more generally 
	\[\define{\trunc^{\leq n}}:\D\to\D^{\leq n}, \quad \define{\trunc^{\geq n}}:\D\to\D^{\geq n}\] 
	can be defined as the adjoints $\iota^{\leq n}\dashv \trunc^{\leq n}$ and $\trunc^{\geq n}\dashv \iota^{\geq n}$ to the inclusions 
	\[\iota^{\leq n}:\D^{\leq n}\to \D \text{ and } \iota^{\geq n}:\D^{\geq n}\to \D \text{ respectively.}\] 
	
	In this paper we are interested in bounded $t$-structures arising from simple-minded and silting collections. We recall the necessary definitions. The following definition goes back to \cite{Ric02}, see also \cite[Def.~3.2]{KY14}.
	
	\begin{defi}
		A \definef{simple-minded collection} is a collection of objects $\DL$ in $\D$ such that  
		\begin{itemize}
			\item no object of $\DL$ is isomorphic to zero,
			\item any non-zero morphism between two objects of $\DL$ is an isomorphism,
			\item $\Hom_{\D}(L,L'[n])=0$ for all $L,L'\in\DL$, $n<0$ and
			\item $\D=\thick(\DL)$.
		\end{itemize}
	\end{defi}
	
	In \cite[Thm.~4.4]{Sch20} it is shown that every simple-minded collection $\DL$ in $\D$ gives rise to a bounded $t$-structure $t_{\DL}=(\D^{\leq 0},\D^{\geq 0})$, where $\D^{\leq 0}$ (resp. $D^{\geq 0}$) denotes the full subcategory of $\D$ consisting of objects that have a \definef{$(d_1,\ldots,d_m)$-$\DL$-filtration} for some $m\in\mathbb{N}$ and $d_1\geq\dots\geq d_m\geq 0$ (resp. $0\geq d_1\geq\dots\geq d_m$), that is there exist $L_1,\ldots,L_m\in\DL$ and a sequence
	\[0=M_0\overset{f_0}{\longrightarrow}M_1\overset{f_1}{\longrightarrow}\cdots \overset{f_{m-2}}{\longrightarrow}M_{m-1}\overset{f_{m-1}}{\longrightarrow}M_m=M\]
	in $\D$ such that there exist triangles 
	\[M_{i-1}\overset{f_{i-1}}{\longrightarrow}M_{i}\overset{g_i}{\longrightarrow}L_i[d_i]\longrightarrow M_{i-1}[1] \quad \text{for }1\leq i\leq m.\]
	The heart $\heart_{\DL}$ of this $t$-structure is a \definef{length abelian category}, that is every object admits a finite composition series and the simple objects of $\heart_{\DL}$ are precisely the objects of $\DL$. Moreover, every bounded $t$-structure on $\D$ with length heart arises in this way, by taking $\DL$ to be the collection of simple objects in the heart (see for example \cite[Thm.~4.6]{Sch20}).
	 
	\begin{defi}[{\cite[Def.~4.1]{PV18}}]
		A \definef{silting collection} is a collection of objects $\DP$ in $\D$ such that  
		\begin{itemize}
			\item $\DP=\add(\DP)$, where $\add(\DP)\subseteq \D$ is the smallest subcategory closed under direct sums and direct summands,
			\item $t_{\DP}$, defined by $(\DP^{\perp_{>0}},\DP^{\perp_{<0}})$ is a $t$-structure on $\D$, called the \definef{silting $t$-structure} associated to $\DP$,
			\item and $\DP\subseteq\DP^{\perp_{>0}}$.
		\end{itemize}
		An object $P\in\D$ is called a \definef{silting object} if $\DP\coloneqq \add(P)$ is a silting collection and $\D=\thick(\DP)$.
	\end{defi}
	
	\begin{ex}
		Let $A$ be a proper connective dg algebra. The \definef{standard $t$-structure $t_{\mathrm{st}}$} on $\fd(A)$ is defined by
		\[\fd(A)^{\leq 0}\coloneqq\left\{M:H^{i}(M)=0 \;\; \forall i\geq 0\right\} \text{ and } \fd(A)^{\geq 0}\coloneqq\left\{M:H^{i}(M)=0 \;\; \forall i\leq 0\right\}.\]
		Then $(\fd(A)^{\leq 0},\fd(A)^{\geq 0})$ is a silting $t$-structure for the silting collection $\add(A)$. Moreover, $A$ is a silting object in $\per(A)$.
	\end{ex}

	\subsection{$A_{\infty}$-algebras}
	
	In this section we recall some basic definitions and facts on $A_{\infty}$-algebras and $A_{\infty}$-categories. We refer the reader to \cite{Kel01,Kel02,LH03} for details. In the following, we avoid the use of explicit signs appearing in the theory of $A_{\infty}$-categories. Although very important for the rigorous foundation of the theory, they are immaterial for our purposes. We remark, however, that there are several different sign conventions for the $A_{\infty}$-relations below. For an overview of the different conventions and their relations among each other we refer the reader to \cite{Pol}. The following definition is originally due to Stasheff \cite{Sta63}. 
	
	\begin{defi}
		An \definef{$A_{\infty}$-category} $\A$ over $k$ is the datum of 
		\begin{itemize}
			\item a class of objects $\mathrm{obj}(\A)$,
			\item for all $A,B\in\A$, a $\ZZ$-graded vector space ${\A}(A,B)$,
			\item for all $n\geq 1$ and all $a_0,\ldots,a_n\in\A$ a graded $k$-linear map
			\[m_n:{\A}(a_{n-1},a_{n})\otimes {\A}(a_{n-2},a_{n-1})\otimes \cdots \otimes {\A}(a_{0},a_{1})\longrightarrow {\A}(a_{0},a_{n})\]
			of degree $2-n$
		\end{itemize}
		such that for all $n\geq 1$ and all $a_0,\ldots,a_n\in\A$, we have the identity
		\[\sum_{r+s+t=n} \pm m_{r+1+t}(\id^{\otimes r}\otimes m_s\otimes \id^{\otimes t})=0,\]
		of maps 
		\[{\A}(a_{n-1},a_{n})\otimes {\A}(a_{n-2},a_{n-1})\otimes \cdots \otimes {\A}(a_{0},a_{1})\longrightarrow {\A}(a_{0},a_{n}).\]
		An $A_{\infty}$-category $\A$ is called \definef{minimal} if $m_1=0$. An $A_{\infty}$-category $\A$ is \definef{strictly unital} if for every $a\in\A$ there exists an element $1_{a}$ of degree $0$ such that for all $f,g$ it holds 
		\[m_1(1_a)=0, \quad m_2(1_a,f)=f, \quad m_2(g,1_a)=g\] 
		as well as for $f_1,\ldots,f_n$, the product $m_n(f_n,\ldots,f_1)$ vanishes if one of the $f_i$ equals $1_a$.
	\end{defi}
		
	\begin{rema}
		Observe that there are different notions of unitality of $A_{\infty}$-categories in the literature. However, as discussed in \cite[Ch.~\uppercase{i}.2]{Sei08}, it turns out that they are equivalent in a suitable sense.
	\end{rema}	
		
	\begin{rema}
		We identify a \definef{dg category} with an $A_{\infty}$-category with $m_n=0$ for $n\geq 3$. An \definef{$A_{\infty}$-algebra} is given by the datum of $A\coloneqq{\A}(\star,\star)$, where $\A$ is an $A_{\infty}$-category with one object $\star$. Explicitly, an $A_{\infty}$-algebra over $k$ is a $\ZZ$-graded vector space $A$ endowed with a family of graded $k$-linear maps 
		\[m_n:A^{\otimes n}\longrightarrow A, \quad n\geq 1\]
		of degree $2-n$, such that the following relations hold: For $n\geq 1$, 
		\[\sum_{r+s+t=n} \pm m_{r+1+t}(\id_A^{\otimes r}\otimes m_s\otimes \id_A^{\otimes t})=0,\]
		where $r,t\geq 0, s\geq 1$. 
	\end{rema}
		
	\begin{defi}
		Let $\A,\B$ be a pair of $A_{\infty}$-categories. An \definef{$A_{\infty}$-functor} $\F:\A\to\B$ is the datum of
		\begin{itemize}
			\item a map $\F_0:\mathrm{obj}(\A)\to\mathrm{obj}(\B)$,
			\item for all $n\geq 1$ and all $a_0,\ldots a_n\in\A$ a graded $k$-linear map
			\[\F_n:\A(a_{n-1},a_{n})\otimes \A(a_{n-2},a_{n-1})\otimes \cdots \otimes \A(a_{0},a_{1})\longrightarrow \B(\F_0(a_0),\F_0(a_n))\]
			of degree $1-n$
		\end{itemize}
		such that for all $n\geq 1$ and all $a_0,\ldots,a_n\in\A$, we have the identity
		\[\sum_{r+s+t=n} \pm \F_{r+1+t}(\id^{\otimes r}\otimes m^{\A}_s\otimes \id^{\otimes t}) = \sum_{u=1}^n \sum_{i_1+\ldots+i_r=n} \pm m^{\B}_u(\F_{i_1}\otimes \cdots \otimes \F_{i_r}) \]
		of maps 
		\[\A(a_{n-1},a_{n})\otimes \A(a_{n-2},a_{n-1})\otimes \cdots \otimes \A(a_{0},a_{1})\longrightarrow \B(\F_0(a_0),\F_0(a_n)).\]
		The $A_{\infty}$-functor $\F$ is a \definef{quasi-equivalence} if for all $a_0,a_1\in\A$ the chain map
		\[\F_1\colon \A(a_0,a_1)\longrightarrow \B(\F_0(a_0),\F_0(a_1))\] 
		is a quasi-isomorphism. We say that $\F$ is \definef{strict} if $\F_n=0$ for all $n\geq 2$. By specializing to the case where $\A$ and $\B$ have a single object, we obtain the notion of a \definef{(strict) morphism of $A_{\infty}$-algebras} and of \definef{quasi-isomorphism of $A_{\infty}$-algebras}.
			
		A \definef{minimal model} of an $A_{\infty}$-category $\A$ is given by the datum of the structure of a minimal $A_{\infty}$-category $(H^*(\A),m_2,m_3,\ldots)$ on the cohomology $H^*(\A)$ of $\A$ together with a quasi-equivalence of $A_{\infty}$-categories 
		\[\F:(H^*(\A),m_2,m_3,\ldots)\longrightarrow \A \quad \text{such that}\quad [\F_1(f)]=f \text{ for all } a,b\in \A, f\in H^{*}(\A)(a,b).\]
	\end{defi}

	\begin{theorem}[Homotopy Transfer Theorem \cite{Ka82}]\label{theorem:HTT}
		Every dg category admits a minimal model. Moreover, this model is unique up to (non unique) $A_{\infty}$-isomorphism.
	\end{theorem}
	
	\begin{rema}
		Every minimal $A_{\infty}$-category arises in this way. In fact, the canonical functor from the $(\infty,1)$-category of dg categories to the $(\infty,1)$-category of $A_{\infty}$-categories is an equivalence of $(\infty,1)$-categories after $(\infty,1)$-localising at the classes of quasi-equivalences, see \cite[Cor.~5.2]{P24}. 
	\end{rema}
	
	For an $A_{\infty}$-category one can define the dg category \define{$\C_{\infty}(\A)$} of $A_{\infty}$-modules as well as the \definef{derived category of $\A$}, denoted by \define{$\D(\A)$}, as well as the \definef{Yoneda $A_{\infty}$-functor} 
	\[Y:\A\to \C_{\infty}(\A), \quad a\mapsto \A(-,a)\]
	in analogy with dg categories (see \cite[\S~4.2]{Kel01} for the definitions).
	
	\begin{defi}[{\cite[Def.~7.2.0.3]{LH03}}]\label{Def:Tw}
		Let $\A$ be an $A_{\infty}$-category. We define the $A_{\infty}$-category of \definef{(one-sided) twisted complexes} \define{$\tw(\A)$} as follows:
		\begin{itemize}
			\item The objects of $\tw(\A)$ are pairs
			\[\left(\bigoplus_{i=1}^n A_i[r_i],\delta=(\delta_{ij})\right), \quad n\in\NN, r_i\in\ZZ,\delta_{ij}\in\A^{r_i-r_j+1}(A_{j},A_{i}),\delta_{ij}=0 \text{ for }i\geq j\]
			such that 
			\begin{equation}\label{eq:Maurer-Cartan}
				\sum_{t=1}^{\infty}\pm m^{\A}_t(\delta,\ldots,\delta)=0.
			\end{equation} 
			Here, by abuse of notation, $m^{\A}_t$ denotes the extension of $m^{\A}_t$ to matrices with entries in $\A$.
			Pictorially, objects of $\tw(\A)$ can be visualized as follows:
			\[\begin{tikzcd}
				{A_n[r_n]} & {A_{n-1}[r_{n-1}]} & \cdots & {A_2[r_2]} & {A_1[r_1]}
				\arrow["{\delta_{n-1,n}}", from=1-1, to=1-2]
				\arrow["{\delta_{2,n}}", curve={height=-36pt}, from=1-1, to=1-4]
				\arrow["{\delta_{1,n}}", curve={height=32pt}, from=1-1, to=1-5]
				\arrow["{\delta_{n-2,n-1}}", from=1-2, to=1-3]
				\arrow["{\delta_{2,n-1}}", curve={height=16pt}, from=1-2, to=1-4]
				\arrow["{\delta_{1,n-1}}", curve={height=-36pt}, from=1-2, to=1-5]
				\arrow["{\delta_{2,3}}", from=1-3, to=1-4]
				\arrow["{\delta_{1,2}}", from=1-4, to=1-5]
			\end{tikzcd}\]
			\item Given a pair of objects of $\tw(\A)$,
			\[T=\left(\bigoplus_{i=1}^n A_i[r_i],\delta=(\delta_{ij})\right), \quad \text{and} \quad T'=\left(\bigoplus_{i=1}^{n'} A'_i[r'_i],\delta'=(\delta'_{ij})\right),\]
			the $\ZZ$-graded vector space of morphisms $\tw(\A)(T,T')$ has as degree $m$ component the vector-space of matrices
			\[f=(f_{ij}), \quad f_{ij}\in\A^{m+r'_i-r_j}(A_j,A'_i).\]
			Pictorially, we can visualize a morphism as follows; to increase readability, we did not include the labels of $\delta_{ij}$ and $f_{ij}$:
			\[\begin{tikzcd}
				{A_n[r_n]} && {A'_{n'}[r'_{n'}]} \\
				{A_{n-1}[r_{n-1}]} && {A'_{n'-1}[r'_{n'-1}]} \\
				\vdots && \vdots \\
				{A_1[r_1]} && {A'_{1}[r'_{1}]}
				\arrow[from=1-1, to=1-3]
				\arrow[from=1-1, to=2-1]
				\arrow[from=1-1, to=2-3]
				\arrow[curve={height=40pt}, from=1-1, to=4-1]
				\arrow[from=1-1, to=4-3]
				\arrow[from=1-3, to=2-3]
				\arrow[curve={height=-45pt}, from=1-3, to=4-3]
				\arrow[from=2-1, to=1-3]
				\arrow[from=2-1, to=2-3]
				\arrow[from=2-1, to=3-1]
				\arrow[curve={height=12pt}, from=2-1, to=4-1]
				\arrow[from=2-1, to=4-3]
				\arrow[from=2-3, to=3-3]
				\arrow[curve={height=-12pt}, from=2-3, to=4-3]
				\arrow[from=3-1, to=4-1]
				\arrow[from=3-3, to=4-3]
				\arrow[from=4-1, to=1-3]
				\arrow[from=4-1, to=2-3]
				\arrow[from=4-1, to=4-3]
			\end{tikzcd}\]
			\item The multiplications $m^{\tw(\A)}_n$ are defined by 
			\[m^{\tw(\A)}_n=\sum_{t=0}^\infty \pm m^{\A}_{n+t}(\id^{\otimes i_1}\otimes\delta^{\otimes j_1} \otimes \id^{\otimes i_2}\otimes\delta^{\otimes j_2}\otimes \ldots \otimes \id^{\otimes i_r}\otimes\delta^{\otimes j_r}).\]
			Here the terms of the second sum are given by decompositions
			\[\sum_{k=1}^r i_k=n, \quad \sum_{k=1}^r j_k=t, \quad i_k,j_k\geq 0\]
			In particular, the higher multiplications $m^{\tw(\A)}_n$ are completely determined by the $A_{\infty}$-structure of $\A$. 
		\end{itemize}
		
		For a closed morphism $f\in Z^0(\tw(\A))(T,T')$ with $T,T'\in\tw(\A)$ we use the notation 
		\[\define{\cocone(f)}\coloneqq \left(T'[-1]\oplus T,\left(\begin{matrix}-\delta_{T'} & f \\ 0 & \delta_{T} \end{matrix}\right)\right),\quad \define{\cone(f)}\coloneqq\left(T'\oplus T[1],\left(\begin{matrix}\delta_{T'} & f \\ 0 & -\delta_{T} \end{matrix}\right)\right).\]
	\end{defi}

	\begin{defi}
		Let $\A$ be a strictly unital $A_{\infty}$-category. We say that $\A$ is a \definef{pretriangulated $A_{\infty}$-category} if the canonical inclusion $H^0(u_{\tw})\colon H^0(\A)\stackrel{\simeq}{\longrightarrow} H^0(\tw(\A))$ is an equivalence.
	\end{defi}
	
	\begin{rema}
		By \cite[\S~4.3]{Kel94} every algebraic triangulated category is equivalent to $H^0(\A)$ of a (non-unique) pretriangulated dg category.
	\end{rema}
	
	The $A_{\infty}$-category $\tw(\A)$ is characterized by the following universal property.
	
	\begin{theorem}[Universal Property of $\tw(\A)$ {\cite[\S~7.2]{LH03}}]\label{theorem:UnivPropTw}
		The canonical inclusion
		\[u_{\tw}:\A\to \tw(\A)\]
		is universal among $A_{\infty}$-functors $\A\to \C$, where $\C$ a pretriangulated $A_{\infty}$-category.
	\end{theorem}

	\subsection{Exact dg categories}
	
	In this section, we collect the necessary statements on exact dg categories from \cite{Chen23}; for a more comprehensive presentation, the reader is referred to the original text.
	
	Given a dg category $\A$ we denote by \define{$\Ht(\A)$} the \definef{homotopy category of $3$-term complexes}. It is constructed as follows. First we define $\mathcal{B}$ to be the dg $k$-path category of the graded quiver
	\[\begin{tikzcd}
		0 & 1 & 2
		\arrow["a", from=1-1, to=1-2]
		\arrow["b", from=1-2, to=1-3]
	\end{tikzcd}\]
	with $|a|=0=|b|$, $d(a)=0=d(b)$ modulo the relation $ba=0$. Form the dg category $\mathrm{Fun}_{A_{\infty}}(\mathcal{B},\A)$ of $A_{\infty}$-functors $F:\mathcal{B}\to \A$. We define $\Ht(\A)$ by
	\[\Ht(\A)\coloneqq H^0(\mathrm{Fun}_{A_{\infty}}(\mathcal{B},\A)).\]
	Explicitly, the objects of $\Ht(\A)$ are given by diagrams $X$ in $\A$ of the form
	\begin{equation}\label{eq:3Term}
		\begin{tikzcd}
		{A_0} & {A_1} & {A_2}
		\arrow["f", from=1-1, to=1-2]
		\arrow["h", curve={height=18pt}, from=1-1, to=1-3]
		\arrow["g", from=1-2, to=1-3]
	\end{tikzcd}\end{equation}
	where $|f|=0=|g|$, $|h|=-1$ and $d(f)=0=d(g)$ as well as $d(h)=-gf$, or equivalently the objects are $3$-term twisted complexes over $\A$. For an explicit description of the morphisms in $\Ht(\A)$ we refer the reader to \cite[Def.~3.39]{Chen23}.
	
	\begin{defi}[{\cite[Def.~3.14]{Chen23}}]
		Consider a diagram in $\A$ of the form 
		\begin{equation}\label{eq:Square}
			\begin{tikzcd}
			{X_{00}} & {X_{01}} \\
			{X_{10}} & {X_{11}}
			\arrow["{f_2}", from=1-1, to=1-2]
			\arrow["{f_1}"', from=1-1, to=2-1]
			\arrow["h", from=1-1, to=2-2]
			\arrow["{g_2}", from=1-2, to=2-2]
			\arrow["{g_1}", from=2-1, to=2-2]
		\end{tikzcd}\end{equation}
		with $|f_i|=0=|g_i|$, $|h|=-1$ and $d(h)=g_1f_1-g_2f_2$. The morphism $(-g_1,g_2):X_{10}\oplus X_{01}\to X_{11}$ yields a canonical morphism 
		\[m:X_{00}\xrightarrow{((f_1,f_2),h)} \cocone((-g_1,g_2))\]
		in $\pretr(\A)$. 
		\begin{itemize}
			\item We say that the above diagram is \definef{homotopy cartesian} (or a \definef{homotopy pullback}) if for all $A\in\A$ the induced morphism
			\[\tau^{\leq 0}\Hom_{\D_{dg}(\A)}(\repr{A},\repr{X_{00}})\xrightarrow{(\repr{m})_{*}}\tau^{\leq 0}\Hom_{\D_{dg}(\A)}(\repr{A},\cocone((-g_1,g_2)))\]
			is an isomorphism in $\D(k)$. Equivalently, 
			\[H^{i}(\repr{m}):H^{i}(\repr{X_{00}})\longrightarrow H^{i}(\cocone((-g_1,g_2)))\] 
			is an isomorphism for all $i\leq 0$.
			\item Dually, the diagram \eqref{eq:Square} is \definef{homotopy cocartesian} (or \definef{homotopy pushout}) if the corresponding diagram in $\A^{op}$ is homotopy cartesian or equivalently the canonical morphism 
			\[\cone((-f_1,f_2))\longrightarrow \repr{X_{11}}\] 
			in $\D_{dg}(\A)$ induces as isomorphism in all non-positive cohomological degrees.
			\item The diagram \eqref{eq:Square} is \definef{homotopy bicartesian} if it is homotopy cartesian and homotopy cocartesian. 
			\item An object $X\in \Ht(\A)$ as in \eqref{eq:3Term} is \definef{homotopy left exact}, if the diagram 
			\begin{equation}\label{eq:LeftExact}\begin{tikzcd}
					{A_0} & {A_1} \\
					0 & {A_2}
					\arrow["f", from=1-1, to=1-2]
					\arrow["0"', from=1-1, to=2-1]
					\arrow["h", from=1-1, to=2-2]
					\arrow["g", from=1-2, to=2-2]
					\arrow["0", from=2-1, to=2-2]
			\end{tikzcd}\end{equation}
			is homotopy cartesian. In that case we call the diagram \eqref{eq:LeftExact} the \definef{homotopy kernel} of $g$. Sometimes we say that $f:A_0\to A_1$ or just $A_0$ is the homotopy kernel of $g$. We will use the notation \define{$\hker(g)$} for the homotopy kernel of $g$.
			\item Dually, $X$ is \definef{homotopy right exact} if \eqref{eq:LeftExact} is homotopy cocartesian and in that case we call \eqref{eq:LeftExact} the \definef{homotopy cokernel} of $f$. Sometimes we say that $g:A_1\to A_2$ or just $A_2$ is the homotopy cokernel of $f$. We will use the notation \define{$\hcoker(f)$} for the homotopy cokernel of $f$.
			\item If an object $X\in\Ht(\A)$ is both homotopy left and right exact, that is the diagram \eqref{eq:LeftExact} is homotopy bicartesian, we call $X$ a \definef{homotopy short exact sequence}.
		\end{itemize}
	\end{defi}
	
	\begin{rema}
		Observe that homotopy (co)kernels, if they exist, are unique up to homotopy equivalence.
	\end{rema}
	
	Recall from \cite{Tab05}, that the category of small dg categories admits the (cofibrantly generated) Dwyer--Kan model category structure, whose weak equivalences are the quasi-equivalences. Its homotopy category is denoted by \define{$\Hqe$}. 
		
	\begin{defi}[{\cite[Def.~5.1]{Chen23}}]
		Let $\A$ be an additive dg category. An \definef{exact structure} on $\A$ is a class $\Sc\subseteq\Ht(\A)$ stable under isomorphisms, consisting of homotopy short exact sequences (called \definef{conflations}) 
		\[\begin{tikzcd}
			{A_0} & {A_1} & {A_2}
			\arrow["f", tail, from=1-1, to=1-2]
			\arrow["h", curve={height=18pt}, from=1-1, to=1-3]
			\arrow["g", two heads, from=1-2, to=1-3]
		\end{tikzcd}\]
		where $f$ is called an \definef{inflation} and $g$ is called a \definef{deflation}, such that the following axioms are satisfied.
		\begin{enumerate}
			\item[Ex0] For every $A\in\A$, $\id_{A}$ is a deflation.
			\item[Ex1] Compositions of deflations are deflations.
			\item[Ex2] Given a deflation $p:B\defl C$ and a morphism $c:C'\to C$ in $Z^0(\A)$ there exists a homotopy pullback
			\[\begin{tikzcd}
				{B'} & {C'} \\
				B & C
				\arrow["{p'}", two heads, from=1-1, to=1-2]
				\arrow["b"', from=1-1, to=2-1]
				\arrow["{h'}", from=1-1, to=2-2]
				\arrow["c", from=1-2, to=2-2]
				\arrow["p", two heads, from=2-1, to=2-2]
			\end{tikzcd}\]
			such that $p'$ is a deflation as well. 
			\item[$\mathrm{Ex2}^{op}$] Given an inflation $i:A\infl B$ and a morphism $a:A\to A'$ in $Z^0(\A)$ there exists a homotopy pushout
			\[\begin{tikzcd}
				{A} & {B} \\
				{A'} & {B'}
				\arrow["{i}", tail, from=1-1, to=1-2]
				\arrow["a"', from=1-1, to=2-1]
				\arrow["{h''}", from=1-1, to=2-2]
				\arrow["{b'}", from=1-2, to=2-2]
				\arrow["{i'}", tail, from=2-1, to=2-2]
			\end{tikzcd}\]
			such that $i'$ is an inflation as well. 
		\end{enumerate}
		We call $(\A,\Sc)$ or simply $\A$, in case there is no risk of confusion, an \definef{exact dg category}.
		
		Let $(\A,\Sc), (\A',\Sc')$ be two exact dg categories. A morphism $F:\A\to \A'$ in $\Hqe$ is called \definef{exact} if the induced functor $\Ht(\A)\to\Ht(\A)$ sends the objects of $\Sc$ to $\Sc'$
	\end{defi}
	
	\begin{rema}
		Observe that by \cite[Ex.~5.6]{Chen23} a pretriangulated dg category equipped with the maximal exact structure of all homotopy short exact sequences is an exact dg category. Hence also every extension-closed subcategory is naturally an exact dg category with the induced exact structure (see \cite[Def.~5.7]{Chen23} for the definitions and the proof).
	\end{rema}
	
	We briefly want to recall the canonical extriangulated structure in the sense of \cite[Def.~2.12]{NP19} on the homotopy category $H^0(\A)$ of an exact dg category $(\A,\Sc)$. We freely use the language of extriangulated categories from \cite{NP19}.
	
	\begin{defi}[{\cite[Def.~5.12]{Chen23}}]
		Consider two objects $A,C\in H^0(\A)$, $i=1,2$ and two conflations $X_i$ of the form
		\begin{equation}\label{eq:Conflation}
			\begin{tikzcd}
			{A} & {B} & {C}
			\arrow["f_i", tail, from=1-1, to=1-2]
			\arrow["h_i", curve={height=18pt}, from=1-1, to=1-3]
			\arrow["g_i", two heads, from=1-2, to=1-3]
		\end{tikzcd}.\end{equation}
		A morphism $\theta:X_1\to X_2$ in $\Ht$ is called an \definef{equivalence} if it restricts to $\id_{A}$ and $\id_C$. This notion of equivalence gives rise to an equivalence relation on the set of conflations. For a conflation $X$ we denote by \define{$[X]$} its equivalence class. We define \define{$\EE_{\Sc}(C,A)$} to be the set of equivalence classes of conflations \eqref{eq:Conflation} with fixed ends $A$ and $C$.
	\end{defi}
	
	By \cite[Prop.~5.17]{Chen23}, $\EE\coloneqq\EE_{\Sc}$ defines a bifunctor. We define the realisation $\extrs$ of $\EE$ as follows: For $\delta=[X]\in\EE(C,A)$ with $X$ a conflation of the form 
	\begin{equation*}\begin{tikzcd}
			{A} & {B} & {C}
			\arrow["f", tail, from=1-1, to=1-2]
			\arrow["h", curve={height=18pt}, from=1-1, to=1-3]
			\arrow["g", two heads, from=1-2, to=1-3]
	\end{tikzcd}.\end{equation*}
	Finally, we define $\extrs(\delta)\coloneqq [A\stackrel{\overline{f}}{\to} B\stackrel{\overline{g}}{\to} C]$, where $\overline{f},\overline{g}$ denote the images of $f$ and $g$ in $H^0(\A)$, respectively.
 	
	\begin{theorem}[{\cite[Thm.~6.19]{Chen23}}]\label{theorem:CanonExt}
		The triple $(H^0(\A),\EE,\extrs)$ defines an extriangulated structure, the \definef{canonical extriangulated structure}, on $H^0(\A)$.
	\end{theorem}

	\begin{theorem}[{\cite[Thm.~6.1.]{Chen23}}]\label{EmbeddingTheorem}
		Let $(\A,\Sc)$ be an exact dg category. There exists a universal exact morphism $F:\A\to \D^b_{dg}(\A)$ in Hqe from $\A$ to a pretriangulated dg category $\D^b_{dg}(\A)$.
		
		If $\A$ is connective, the morphism $F$
		\begin{enumerate}
			\item induces a quasi-equivalence from $\A$ to $\tau^{\leq 0}\D'$, where $\D'$ is an extension-closed dg subcategory of $\D^b_{dg}(\A)$, and
			\item a natural bijection $\EE(C,A)\overset{\sim}{\to} \Hom_{\D^b(\A)}(FC,FA[1])$ for each pair of objects $C,A$ in $H^0(\A)$ where $\define{\D^b(\A)}:=H^0(\D^b_{dg}(\A))$.
		\end{enumerate}
		We call $\D^b_{dg}(\A)$ the \definef{bounded derived dg category} of $(\A,\Sc)$.
	\end{theorem}

	\begin{defi}[{\cite[Def.~1.1.1]{Bon10}},{\cite[Def.~1.4]{Pau08}}]
		Let $\T$ be a triangulated category. A pair $(\T_{\geq 0},\T_{\leq 0})$ of full subcategories in $\T$ that are closed under isomorphisms, direct sums and direct summands is called a \definef{co-t-structure} on $\T$ if the following conditions are satisfied:
		\begin{enumerate}
			\item $\T_{\geq 0}[-1]\subseteq \T_{\geq 0}$ and $\T_{\leq 0}[1]\subseteq \T_{\leq 0}$.
			\item\label{item:orthogonality} $\T_{\geq 0}[-1] \perp \T_{\leq 0}$.
			\item For all $X\in\T$ there exists a triangle 
			\[w_{>0}X\to X \to w_{\leq 0}X \to w_{>0} X[1]\]
			with $w_{>0}X\in \T_{\geq 0}[-1]$ and $w_{\leq 0}X\in\T_{\leq 0}$.
		\end{enumerate}
		The full subcategory $\T_{\geq 0}\cap \T_{\leq 0}$ is called the \definef{coheart} of the co-t-structure. As for $t$-structures we use the notation $\define{\T_{\geq n}}\coloneqq\T_{\geq 0}[n]$, $\define{\T_{> n}}\coloneqq\T_{\geq 0}[n+1]$, etc.
	\end{defi}
	
	\begin{theorem}[{\cite[Prop.~6.2.1, Rem.~6.2.2.3)]{Bon10}}]\label{theorem:Co-T-Structure}
		Let $\A$ be a connective dg category. Then the following full subcategories define a bounded co-t-structure, the canonical co-t-structure, on $\tr(\A)\coloneqq H^0(\tw(\A))$.
		\[\define{\tr(\A)_{\geq 0}}\coloneqq \left\{\left(\bigoplus_{i=1}^n A_i[r_i],q\right) : n\geq 0,r_{i}\leq 0 \right\},\;  \define{\tr(\A)_{\leq 0}}\coloneqq \left\{\left(\bigoplus_{i=1}^n A_i[r_i],q\right) : n\geq 0,r_{i}\geq 0 \right\}\]
	\end{theorem}
	
	\begin{rema}
		Although $\D^b(\A)$ is constructed as a certain Verdier quotient of $\tr(\A)$ in \cite[Lem.~6.5]{Chen23}, this does not imply that the above induces a co-t-structure on $\D^b(\A)$. However, in \cite[Def.~2.2]{ChenXW25} and \cite[Def.~0.1]{Sau23} the notion of a pre-weight or heart structure was introduced, which is essentially a co-t-structure in which condition \ref{item:orthogonality} does not necessarily hold. It was proven in \cite[Thm.~1.6]{SW25} that, if $\A$ is connective, $\D^b(\A)$ always admits a bounded heart structure, whose heart is the weak-idempotent completion of $\A$. We, however, make no use of this fact. 
	\end{rema}
	
	\subsection{$\delta$-functors}
	
	In this subsection we give a small summary of the notion of a $\delta$-functor that will be relevant for the proof of \cref{GeneralisedRealFunctor}. For a more comprehensive overview, see \cite{GNP21}. In the following $(\Extr,\EE,\extrs)$ will be an extriangulated category. We denote by 
	\[\Mod_\Extr:=\mathrm{Fun}_{k}(\Extr^{op},\Mod(k)) \text{ and } _\Extr\Mod:=\mathrm{Fun}_{k}(\Extr,\Mod(k))\]
	the right and the left $\Extr$-modules, respectively.
	
	\begin{defi}[{\cite[Def.~4.1., Def.~4.5]{GNP21}, \cite[Sec.~7.1]{GNP23}}]
		A \definef{contravariant connected sequence of functors} is a pair $(F,\epsilon)$ where $F=(F^n)_{n\geq 0}$ is a sequence of right $\Extr$-modules and a collection of morphisms $\epsilon=\epsilon^{n}_{\delta}:F^n(A)\to F^{n+1}(C)$ for any $\EE$-extension $\delta\in\EE(C,A)$ and $n\geq 0$, which is natural with respect to morphisms of $\EE$-extensions. Such a pair is a \definef{right $\delta$-functor} if for any $\EE$-extension $\delta$, with $\extrs(\delta)=[A\overset{x}{\to} B\overset{y}{\to} C]$, the sequence
		\[\cdots\overset{\epsilon^{n-1}_{\delta}}{\longrightarrow}F^{n}(C)\overset{y_*}{\longrightarrow}F^{n}(B)\overset{x_*}{\longrightarrow}F^{n}(A)\overset{\epsilon^{n}_{\delta}}{\longrightarrow}F^{n+1}(C)\longrightarrow\cdots,\]
		is exact. A \definef{morphism of right $\delta$-functors} $(F,\epsilon)\to (\tilde{F},\tilde{\epsilon})$ is a family of morphisms $\theta^n:F^n\to\tilde{F}^n$ in $\Mod_{\Extr}$ which is compatible with the connecting homomorphism, i.e. for each $\delta\in\EE(C,A)$ and $n\geq 0$, the following diagram commutes:
		\[\begin{tikzcd}
			{F^n(A)} & {F^{n+1}(C)} \\
			{\tilde{F}^{n}(A)} & {\tilde{F}^{n+1}(C)}
			\arrow["{\epsilon^n_{\delta}}", from=1-1, to=1-2]
			\arrow["{\theta^n(A)}"', from=1-1, to=2-1]
			\arrow["{\theta^{n+1}(C)}", from=1-2, to=2-2]
			\arrow["{\tilde{\epsilon}^n_{\delta}}"', from=2-1, to=2-2]
		\end{tikzcd}\]
		Composition of morphisms of right $\delta$-functors is defined in the obvious way. \definef{Covariant connected sequences of functors} $(F,\eta)$ and \definef{left $\delta$-functors} are defined dually.
		
		A \definef{bivariant connected sequence of functors} $(F,\epsilon,\eta)$ is a triplet, where $F^*$ is a sequence of $\Extr$-$\Extr$-bimodules and $\epsilon$ and $\eta$ are collections of morphisms $\epsilon^{n}_{\delta}:F^n(A,-)\to F^{n+1}(C,-)$ and $\eta^{n}_{\delta}:F^n(-,C)\to F^{n+1}(-,A)$, for any $\delta\in\EE(C,A)$ and $n\geq 0$, which are natural with respect to morphisms of $\EE$-extensions. It is called a \definef{$\delta$-functor}, if for any $\EE$-extension $\delta$, with $\extrs(\delta)=[A\overset{x}{\to} B\overset{y}{\to} C]$ the sequences
		\[\cdots\overset{\epsilon^{n-1}_{\delta}}{\longrightarrow}F^{n}(C,-)\overset{F^n(y,-)}{\longrightarrow}F^{n}(B,-)\overset{F^n(x,-)}{\longrightarrow}F^{n}(A,-)\overset{\epsilon^{n}_{\delta}}{\longrightarrow}F^{n+1}(C,-)\to\cdots,\]
		\[\cdots\overset{\eta^{n-1}_{\delta}}{\longrightarrow}F^{n}(-,A)\overset{F^n(-,x)}{\longrightarrow}F^{n}(-,B)\overset{F^n(-,y)}{\longrightarrow}F^{n}(-,C)\overset{\eta^{n}_{\delta}}{\longrightarrow}F^{n+1}(-,A)\to\cdots,\]
		are exact in $\Mod_\Extr$ and $_\Extr\Mod$, respectively. A \definef{morphism of $\delta$-functors} $(F,\epsilon)\to (\tilde{F},\tilde{\epsilon})$ is a family of $\Extr$-$\Extr$-bimodule morphisms $\theta^n:F^n\to\tilde{F}^n$, which is compatible with the connecting morphisms. Again, composition of morphisms of $\delta$-functors is defined in the obvious way.
	\end{defi}

	\begin{ex}[{\cite[\S~6.23]{Chen23}}]
		Let $\A$ be a connective exact dg category. Then $H^0(\A)$ is canonically an extriangulated category. Consider $(\Hom_{\D^b(\A)}(?,-[n]))_{n\geq 0}$ together with the connecting homomorphisms given by triangles in $\D^b(\A)$. This is a $\delta$-functor on the extriangulated category $H^0(\A)$.
	\end{ex}
	
	\begin{ex}[{\cite[Def.~3.1.]{GNP21}}]\label{ExtriangYonedaExt}
		Define $\define{\EE^0}=\Hom_{\Extr}$, $\define{\EE^1}:=\EE$ and inductively
		\[\define{\EE^n(X,Y)}=(\EE^{n-1}\diamond\EE)(X,Y):=\int^{C\in\Extr}\EE^{n-1}(C,Y)\otimes_{k}\EE(X,C).\]
		More explicitly, by \cite[Prop.~3.9]{GNP21} the above can be identified with
		\[\coprod_{C\in\Extr}\left(\EE^{n-1}(C,Y)\times \EE(X,C) \right)/\sim,\]
		where $\sim$ is the equivalence relation generated by 
		\[(\kappa,\EE(X,f)(\lambda))\sim(\EE^{n-1}(f,Y)(\kappa),\lambda)\] 
		with $f\in\Extr(C,C')$, $\kappa\in\EE^{n-1}(C',Y)$, $\lambda\in\EE(X,C)$ arbitrary. For $\rho\in\EE^{n-1}(C',Y)$ and $\delta\in\EE(X,C)$ we denote by $\overline{(\rho,\delta)}$ the equivalence class of $\rho\otimes\delta$ in $\EE^n(X,Y)$. Then one defines $\epsilon^{n}_{\delta}:\EE^n(A,-)\to \EE^{n+1}(C,-)$ and $\eta^{n}_{\delta}:\EE^n(-,C)\to \EE^{n+1}(-,A)$  for any $\EE$-extension $\delta\in\EE(C,A)$ and $n\geq 0$ by 
		\[\epsilon^{n}_{\delta}(\rho)=\overline{(\rho,\delta)} \text{ for } \rho\in\EE^n(A,-), \quad \eta^n_{\delta}(\gamma)=\overline{(\delta,\gamma)} \text{ for }\gamma\in\EE^n(-,C) \]
		where we use the fact that we have a natural isomorphism $\EE^{\diamond(s+t)}\cong \EE^s\diamond\EE^t$ ($s,t\geq 0$) for the definition of $\eta$. It is shown in \cite[Ch.~3.4]{GNP21} that with these definitions $(\EE^*,\epsilon,\eta)$ defines a $\delta$-functor on $\Extr$.
	\end{ex}
	
	The following proposition connects the above two notions and allows us to interpret morphisms of degree $n$ in $\D^b(\A)$ between objects of $\A$ as long exact sequences analogous to the abelian or exact case. 
	
	\begin{prop}[{\cite[Prop.~6.24]{Chen23}}]\label{HigherExt}
		Let $\A$ be a connective exact dg category and $F:\A\to\D^b_{dg}(\A)$ the universal embedding into a pretriangulated dg category from \cref{EmbeddingTheorem}. Then, we have a canonical isomorphism of $\delta$-functors $\alpha:(\EE^n(?,-))_{n\geq 0}\to(\Hom_{\D^b(\A)}(?,-[n]))_{n\geq 0}$.
	\end{prop}
	
	The main ingredient for the proof is the following lemma, that is also used in our proof of \cref{GeneralisedRealFunctor}.
	
	\begin{prop}[{\cite[Prop.~A.15]{Chen23}}]\label{WeaklyEffaceable}
		Let $(F,\epsilon)$ and $(G,\eta)$ be right $\delta$-functors such that $G^n$ is \definef{weakly effaceable} for each $n>0$, that is for all $X\in \Extr$ and $f\in G^n(X)$ there exists a deflation $g:Z\to X$ in $\Extr$ such that $G^n(g)(f)=0$. Then, each morphism $F^0\to G^0$ in $\Mod_{\Extr}$ extends uniquely to a morphism of right $\delta$-functors $(F,\epsilon)\to (G,\eta)$.
	\end{prop}
	
	\begin{coro}\label{DeltaCorollary}
		Let $(F,\epsilon)$ and $(G,\eta)$ be right $\delta$-functors such that $F^n,G^n$ are weakly effaceable for each $n>0$. Assume $F^0=\Hom_{\Extr}(-,Y)=G^0$ for some $Y\in\Extr$. Then, $(F,\epsilon)$ and $(G,\eta)$ are isomorphic right $\delta$-functors.
	\end{coro}
	
	\begin{proof}
		Using \cref{WeaklyEffaceable} we can extend the identity in degree $0$ uniquely to morphisms $\theta:F\to G$, $\rho:G\to F$. But then, by \cref{WeaklyEffaceable} the compositions $\rho\circ \theta$ and $\theta\circ \rho$ must be $\id:(F,\epsilon)\to (F,\epsilon)$ and $\id:(G,\eta)\to (G,\eta)$ respectively.
	\end{proof}
	
	\section{Dg realization functors}\label{sec:Realisation}
	
	Let $\C$ be a pretriangulated dg category and $\A\subseteq \C$ an exact dg subcategory. By \cref{EmbeddingTheorem} there exists a pretriangulated dg category $\D^b_{dg}(\A)$ and an exact dg functor 
	\[\real:\D^b_{dg}(\A)\longrightarrow \C.\]
	We call this functor the \definef{dg realization functor}. The goal of this section is to give characterizations for $\real$ to be quasi fully-faithful under the additional assumption that $\A$ is connective.
	
	\subsection{Dg realization functors for exact dg subcategories}
	
	\begin{defi}
		Let $\T$ be a triangulated category and $\Sc\subseteq \T$ a subcategory. We say that $\Sc$ is \definef{$1$-generated} if for all $X,Y\in\Sc$ and morphisms $f:X\to Y[n]$ in $\T$ for an integer $n\geq 1$, there exist
		\begin{itemize}
			\item $X=X_0,X_1,\ldots,X_{n-1},X_n=Y\in\Sc$,
			\item and morphisms $f_i:X_i[i]\to X_{i+1}[i+1]$ in $\T$,
		\end{itemize}
		such that $f=f_{m-1}\cdots f_1f_0$.
	\end{defi}
	
	The following is a generalization of \cite[Thm.~3.11]{PV18} for exact dg categories. We remark that characterizations \ref{itemc} and \ref{itemf} do not appear in their theorem. However, their proof of the implication $(c)\Rightarrow (b)$ (using the numbering in \cite{PV18}) implies the equivalence of \ref{itemb} and \ref{itemc}. The characterization \ref{itemf} is well-known in their setting as well.
	
	\begin{theorem}\label{GeneralisedRealFunctor}
		Let $\C$ be a pretriangulated dg category and $\Cat{B}\subseteq \C$ a full extension-closed dg subcategory. Set $\A\coloneqq\tau^{\leq 0}\Cat{B}$. The following statements hold:
		\begin{enumerate}
			\item\label{item1} The realization functor $\real$ induces isomorphisms $\Hom_{D^b(\A)}(X,Y[n])\overset{\cong}{\to} \Hom_{H^0(C)}(X,Y[n])$ for $X,Y\in H^0(\A)$ and $n\leq 1$.
			\item\label{item2} The following are equivalent:
			\begin{enumerate}
				\item\label{itema} $\real$ is quasi fully-faithful.
				\item\label{itemb} $\real$ induces isomorphisms $\Hom_{D^b(\A)}(X,Y[n])\overset{\cong}{\to} \Hom_{H^0(\C)}(X,Y[n])$ for all $n\geq 2$ and $X,Y\in H^0(\A)$.
				\item\label{itemc} $\real$ induces epimorphisms $\Hom_{D^b(\A)}(X,Y[n])\twoheadrightarrow \Hom_{H^0(\C)}(X,Y[n])$ for all $n\geq 2$ and $X,Y\in H^0(\A)$.
				\item\label{itemd} (Ef) $\Hom_{H^0(\C)}(-,Y[n])$ is weakly effaceable for every $Y \in H^0(\A)$ and $n\geq 2$, that is for all $X \in H^0(\A)$ and $f:X\to Y[n]$ in $H^0(\C)$ there exists a deflation $g:Z\to X$ in $H^0(\A)$ such that $fg=0$.
				\item\label{iteme} (CoEf) $\Hom_{H^0(\C)}(X,-[n])$ is weakly effaceable for every $X \in H^0(\A)$ and $n\geq 2$, that is for all $Y \in H^0(\A)$ and $f:X\to Y[n]$ in $H^0(\C)$ there exists an inflation $g:Y\to Z$ in $H^0(\A)$ such that $g[n]\circ f=0$.
				\item\label{itemf} $H^0(\A)$ is 1-generated in $H^0(\C)$.
			\end{enumerate}
			If the above equivalent conditions are satisfied, the essential image of $\real$ is $\mathrm{tria}(H^0(\A))$, the smallest triangulated subcategory of $H^0(\C)$ containing the objects of $H^0(\A)=H^0(\B)$.
		\end{enumerate}
	\end{theorem}
	
	\begin{proof}
		In order to prove \ref{item1}, we first observe that for $n\leq 0$ this follows immediately from the assumption that $\Cat{B}\subseteq\C$ is full and the first part of \cref{EmbeddingTheorem}. For $n=1$, the induced morphism can be identified with 
		\[\Hom_{\D^b(\A)}(X,Y[1]) \cong \EE_{H^0(\A)}(X,Y)= \EE_{H^0(\C)}(X,Y)\cong \Hom_{H^0(\C)}(X,Y[1]).\]
		
		For \ref{item2}, we observe that the equivalence of \ref{itema} and \ref{itemb} follows by a devissage argument using that $\mathrm{tria}(H^0(\A))=\D^b(\A)$, which holds for example by \cref{theorem:Co-T-Structure}.
		
		\ref{itemb}$\Leftrightarrow$\ref{itemc}: It is enough to prove that \ref{itemc} implies \ref{itemb}. The proof we present here is an adaption of the proof of injectivity in the implication $(c)\Rightarrow (b)$ in \cite[Thm.~3.11]{PV18}. Inductively, we prove the injectivity of the map $\Hom_{D^b(\A)}(X,Y[n])\to \Hom_{H^0(\C)}(X,Y[n])$. For $n=1$ this is \ref{item1}. Consider $n\geq 2$ and $\alpha\in\Hom_{D^b(\A)}(X,Y[n])$ such that $\real(\alpha)=0$. Now, $H^0(\A)\subseteq \D^b(\A)$ is 1-generated by the construction of $\EE^n$ and the fact that $\Hom_{\D^b(\A)}(-,?[n])\cong \EE^n(-,?)$ by \cref{HigherExt}. Thus, there exist an object $X_1\in H^0(\A)$ as well as a morphism $f_0\in\Hom_{D^b(\A)}(X,X_1[1])$ and $f_1\in\Hom_{D^b(\A)}(X_1[1],Y[n])$ such that $\alpha=f_1f_0$. Choose a triangle 
		\[\begin{tikzcd}
			{X_1} & C & X & {X_1[1]}
			\arrow["{g_0}", from=1-1, to=1-2]
			\arrow[from=1-2, to=1-3]
			\arrow["{f_0}", from=1-3, to=1-4]
		\end{tikzcd}\]
		in $\D^b(\A)$. As $H^0(\A)$ is extension-closed, it follows that $C\in H^0(\A)$. Thus, we obtain a commutative diagram whose rows are triangles in $H^0(\C)$. 
		\[\begin{tikzcd}
			X & {X_1[1]} & C[1] & {X[1]} \\
			0 & {Y[n]} & {Y[n]} & 0
			\arrow["{\real(f_0)}", from=1-1, to=1-2]
			\arrow[from=1-1, to=2-1]
			\arrow["{\real(-g_0[1])}", from=1-2, to=1-3]
			\arrow["{\real(f_1)}", from=1-2, to=2-2]
			\arrow[from=1-3, to=1-4]
			\arrow["-\epsilon", dashed, from=1-3, to=2-3]
			\arrow[from=1-4, to=2-4]
			\arrow[from=2-1, to=2-2]
			\arrow[Rightarrow, no head, from=2-2, to=2-3]
			\arrow[from=2-3, to=2-4]
		\end{tikzcd}\]
		By surjectivity of the map
		\[\real:\Hom_{D^b(\A)}(C[1],Y[n])\defl\Hom_{H^0(\C)}(C[1],Y[n]),\] 
		there exists $\epsilon'\in\Hom_{D^b(\A)}(C[1],Y[n])$ such that $\real(\epsilon')=\epsilon$. Hence $\real(\epsilon'\circ g_0[1])=\real(f_1)$ and thus, by induction, $\epsilon'\circ g_0[1]=f_1$. But then $\alpha=f_1f_0=\epsilon'\circ g_0[1]\circ f_0=0$, which is what we needed to show.
		
		\ref{itemb}$\Leftrightarrow$\ref{itemd}: Observe that $\real$ induces a morphism of $\delta$-functors 
		\[\real^n:\left(\Hom_{\D^b(\A)}(-,?[n])\right)_{n\geq 0}\longrightarrow \left(\Hom_{H^0(\C)}(-,?[n])\right)_{n\geq 0}.\]
		Hence, if \ref{itemb} holds, then \ref{itemd} follows as $\Hom_{\D^b(\A)}(-,?[n])$ is a weakly effaceable $\delta$-functor by the proof of \cite[Prop.~6.24]{Chen23}. This proves \ref{itemb} implies \ref{itemd}. For the converse, as 
		\[\real^0=\id_{\Hom(-,?)}:\Hom_{\D^b(\A)}(-,?)\longrightarrow \Hom_{H^0(\C)}(-,?)\] 
		we obtain that $\real$ must be an isomorphism if $\Hom_{H^0(\C)}(-,Y[n])$ is weakly effaceable by \cref{DeltaCorollary}.
		
		\ref{itemb}$\Leftrightarrow$\ref{iteme}: This can be proven dually using the dual of \cref{DeltaCorollary}, as in the proof of \ref{itemb}$\Leftrightarrow$\ref{itemd}. 
		
		\ref{itemb}$\Rightarrow$\ref{itemf}: Follows from the fact that $H^0(\A)$ is 1-generated in $\D^b(\A)$ as we have seen in the proof of \ref{itemb}$\Leftrightarrow$\ref{itemc}.
		
		\ref{itemf}$\Rightarrow$\ref{itemc}: By assumption we can factor every degree $n$ morphism as degree $1$ morphisms. But by \ref{item1}, $\real$ induces an epimorphism in degree $1$.
		
		The last statement in \ref{item2} follows from the fact that in this case $\real$ identifies $\D^b(\A)$ with a triangulated subcategory of $H^0(\C)$ and $\D^b(\A)$ is the smallest triangulated subcategory of $\D^b(\A)$ containing $H^0(\A)$, since $\D^b(\A)$ is constructed as a Verdier quotient of $\tr(\A)$ in \cite[Lem.~6.5]{Chen23}.
	\end{proof}

	\begin{rema}
		In \cite[Thm.~1.2]{SW25} a similar statement is proven for exact $\infty$-categories. The proof given in \cite[\S~4]{SW25} is very similar to the proof above.  
	\end{rema}
	
	We get the following characterization under the additional assumption that $\A$ has enough projectives, which is inspired by \cite[Prop.~1.3.3.7]{Lurie17} but seems to be new in our setting. We briefly recall the necessary definitions. We call an object $P\in\A$ \definef{projective} if $P$ is projective in the extriangulated category $H^0(\A)$, that is $\EE(P,-)=0$. We denote by $\define{\Proj(\A)}\subseteq \A$ the subcategory of projective objects in in $\A$. We say that \definef{$\A$ has enough projectives}, if the extriangulated category $H^0(\A)$ has enough projectives, that is for all $X\in H^0(\A)$ there exists $P\in\Proj(\A)$ and a deflation 
	\[P\defl X.\]
	
	\begin{coro}\label{coro:RealisationEnoughProj}
		Let $\C$ be a pretriangulated dg category and $\Cat{B}\subseteq \C$ a full extension-closed dg subcategory. Set $\A\coloneqq\tau^{\leq 0}\Cat{B}$ and assume $\A$ has enough projectives. The following statements are equivalent:
		\begin{enumerate}
			\item\label{item:RealisationEnoughProj1} $\real:\D^b_{dg}(\A)\to \C$ is fully-faithful.
			\item\label{item:RealisationEnoughProj2} For all $P\in\Proj(\A)$ and $X\in\A$ it holds $\Hom_{\C}(P,X[n])=0$ for all $n\geq 2$.
		\end{enumerate}
		Moreover, if the above equivalent conditions are satisfied, the essential image of $\real$ is $\mathrm{tria}(H^0(\A))$.
	\end{coro}
	
	\begin{proof}
		The last part of the statement is part of \cref{GeneralisedRealFunctor}, hence it is enough to prove the equivalence \ref{item:RealisationEnoughProj1}$\Leftrightarrow$\ref{item:RealisationEnoughProj2}.
		
		\ref{item:RealisationEnoughProj1}$\Rightarrow$\ref{item:RealisationEnoughProj2}: If $\real$ is fully-faithful, it holds for all $P\in\Proj(\A)$ and $X\in\A$ 
		\[\Hom_{H^0(\C)}(P,X[n])\cong \Hom_{\D^b(\A)}(P,X[n]) \cong \EE^n(P,X)=0 \quad \forall n\geq 1\]
		by definition of $\EE^n$ and the fact that $\EE(P,-)=0$.
		
		\ref{item:RealisationEnoughProj1}$\Leftarrow$\ref{item:RealisationEnoughProj2}: By \cref{GeneralisedRealFunctor}\ref{itemf} it is enough to show that $H^0(\A)\subseteq H^0(\C)$ is $1$-generated. Let $X,Y\in H^0(\A)$ and $f\in\Hom_{H^0(\C)}(X,Y[n])$ with $n\geq 2$. Since $\A$ has enough projectives, there exists a triangle in $H^0(\C)$ of the form 
		\[X_1\longrightarrow P \stackrel{g}{\longrightarrow} X \stackrel{f_1}{\longrightarrow} X_1[1]\]
		such that $X_1\in H^0(\C)$ and $P\in\Proj(\A)$. By assumption it holds $0=fg\in\Hom(P,Y[n])$. Hence, we can complete the following to a morphism of triangles.
		\[\begin{tikzcd}
			P & X & {X_1[1]} & {P[1]} \\
			0 & {Y[n]} & {Y[n]} & 0
			\arrow["{g}",from=1-1, to=1-2]
			\arrow[from=1-1, to=2-1]
			\arrow["{f_1}", from=1-2, to=1-3]
			\arrow["f", from=1-2, to=2-2]
			\arrow[from=1-3, to=1-4]
			\arrow["{f'}", dashed, from=1-3, to=2-3]
			\arrow[from=1-4, to=2-4]
			\arrow[from=2-1, to=2-2]
			\arrow["\id", from=2-2, to=2-3]
			\arrow[from=2-3, to=2-4]
		\end{tikzcd}\]
		In particular, $f$ factors as $f=f'f_1$. Since $X_1\in H^0(\A)$ the claim follows by induction.
	\end{proof}

\subsection{Realization functors for extended hearts of $t$-structures}	
	As a corollary of the above we obtain the following analogue of \cite[Thm.~B]{CHZZ19}. The proof we present is an adaption of their proof.
	
	\begin{defi}
		Let $\C$ be a pretriangulated dg category. A \definef{$t$-structure} $t=(\C^{\leq 0},\C^{\geq 0})$ on $\C$ is a $t$-structure on $H^0(\C)$. For $d\geq 1$ we define the \definef{$d$-extended heart} of $t$ as
		\[\define{\Hc{\C}{d}}:=\define{\C^{(-d,0]}}\coloneqq \tau^{\leq 0}\left(\C^{\leq 0}\cap \C^{>-d}\right).\]
		Then, $\Hc{\C}{d}$ canonically inherits the structure of an exact dg category from $\C$. We will use the notation $\define{\He{\C}{d}}\coloneqq H^0(\Hc{\A}{d})$ to refer to the underlying extriangulated category.
	\end{defi}
	
	\begin{rema}
		In fact, $\Hc{\C}{d}$ is an abelian $d$-truncated dg category in the sense of \cite[Thm.~3.44]{Moc25}. This fact plays a role in \cref{sec:AbelianD}.
	\end{rema}
	
	\begin{coro}\label{DenseCoro}
		Let $\C$ be a pretriangulated dg category with a bounded $t$-structure and $d\geq 1$. The realization functor $\real:\D^b_{dg}(\Hc{\C}{d})\to \C$ is a quasi-equivalence if and only if it is dense.
	\end{coro}
	
	\begin{proof}
		It is enough to prove that density implies quasi fully-faithful. Applying \cref{GeneralisedRealFunctor}, it is enough to prove that $\He{\C}{d}$ is $1$-generated in $H^0(\C)$. Hence, consider a morphism $f\in\Hom_{H^0(\C)}(X,Y[n])$ for $X,Y\in\He{\C}{d}$ and $n\geq 2$. By assumption, we obtain $Z\in\D^b(\Hc{\C}{d})$ and a triangle in $H^0(\C)$
		\begin{equation}\label{eq:DenseCoro1}
			X\overset{f}{\longrightarrow}Y[n]\overset{a}{\longrightarrow} \real(Z) \longrightarrow X[1].
		\end{equation}
		Moreover, since $\D^b(\Hc{\C}{d})$ is a Verdier quotient of $H^0(\tw(\Hc{\C}{d}))$, it follows from \cref{theorem:Co-T-Structure} that there exist $Z^{-n}\in \He{\C}{d}$, $Z'\in\He{\C}{d}\ast \He{\C}{d}[1]\ast\ldots \ast \He{\C}{d}[n-1]$ and a triangle in $\D^b(\Hc{\C}{d})$ of the form
		\begin{equation}\label{eq:DenseCoro2}
			Z'{\longrightarrow}Z\overset{p}{\longrightarrow} Z^{-n}[n] \longrightarrow Z'[1].
		\end{equation}
		Define 
		\[a':Y\xrightarrow{a[-n]} \real(Z)[-n] \xrightarrow{\real(p)[-n]} \real(Z^{-n}[n])[-n]=Z^{-n}.\] 
		We claim that there exists a triangle in $H^0(\C)$ with objects in $\He{\C}{d}$ 
		\begin{equation}\label{eq:DenseCoro3}
			Y\overset{a'}{\longrightarrow}Z^{-n}\overset{b}{\longrightarrow} C \overset{\delta}{\longrightarrow}Y[1],
		\end{equation} 
		or equivalently, $C\coloneqq \cone(a')$ lies in $\He{\C}{d}$. By the long exact sequence in cohomology for the triangle \eqref{eq:DenseCoro3}, 
		\[\cdots \to 0=H^{-d}_t(C)\xrightarrow{H^{-d}_t(\delta)}H^{1-d}_t(Y)\xrightarrow{H^{1-d}_t(a')}H^{1-d}_t(Z^{-n})\to \cdots,\]
		it is enough to argue that $H^{1-d}_t(a')$ is a monomorphism. Indeed, it holds
		\[H^{1-d}_t(a')=H^{1-d}_t(\real(p)[-n]\circ a[-n])=H^{1-d}_t(\real(p)[-n])\circ H^{1-d}_t(a[-n]).\] 
		Therefore it is enough to argue that the latter morphisms are monomorphisms. Now $H^{1-d}_t(a[-n])$ is a monomorphism, by analysing the long exact sequences in cohomology for the triangle \eqref{eq:DenseCoro1}
		\[0=H^{1-d}_t(X[-n])\to H^{1-d}_t(Y)\xrightarrow{H^{1-d}_t(a[-n])}H^{1-d}_t(\real(Z)[n])\to \cdots.\]
		Similarly, the long exact sequence in cohomology for the triangle \eqref{eq:DenseCoro2}
		\[0=H^{1-d}_t(Z'[-n])\to H^{1-d}_t(Z[-n])\xrightarrow{H^{1-d}_t(\real(p)[-n])}H^{1-d}_t(Z^{-n})\to \cdots.\]
		yields that $H^{1-d}_t(\real(p)[-n])$ is a monomorphism, which is what we needed to prove.
		 
		Applying $\Hom_{H^0(\C)}(X,-)$ to the triangle \eqref{eq:DenseCoro3} yields the exact sequence
		\[\Hom_{H^0(\C)}(X,C[n-1])\overset{\delta\circ-}{\longrightarrow} \Hom_{H^0(\C)}(X,Y[n])\overset{a'[n]\circ -}{\longrightarrow} \Hom_{H^0(\C)}(X,Z^{-n}[n]).\]
		As $a'[n]f=\real(p)af=0$, there exists $f'\in\Hom_{H^0(\C)}(X,C[n-1])$ such that $f=\delta f'$. By induction, we obtain that $\He{\C}{d}$ is 1-generated in $H^0(\C)$. 
	\end{proof}
	
	As a consequence of the above we obtain the following relation to the notion of compliciality of a $t$-structure, which clarifies that a $(d-1)$-complicial $t$-structure on a pretriangulated dg category $\C$ allows to reconstruct $\C$ up to quasi-equivalence from the extended heart $\Hc{\C}{d}$. 
	
	\begin{defi}[{\cite[Def.~C.5.3.1.]{Lurie18}}]\label{def:complicial}
		We call a $t$-structure $t$ on a triangulated category $\T$ \definef{$n$-complicial} if for every object $X\in \T^{\leq 0}$ there exists $X'\in \T^{[-n,0]}$ and a morphism $f\in \Hom_{\T}(X',X)$, such that $H^0(f):H^0(X')\to H^0(X)$ is an epimorphism in $\heart_t$. We call $t$ \definef{strictly $n$-complicial} if $n$ is minimal with this property. We call a simple-minded collection $\DL$ \definef{(strictly) $n$-complicial} if the corresponding $t$-structure $t_{\DL}$ is (strictly) $n$-complicial.
	\end{defi}
	
	The following theorem is a variant of \cite[Thm.~6.3.2]{Ste23} for dg categories.
	
	\begin{theorem}\label{compliciality}
		Let $\C$ be a pretriangulated dg category with a bounded $t$-structure. Let $d\geq 1$. The following are equivalent:
		\begin{enumerate}
			\item The $t$-structure on $\C$ is $(d-1)$-complicial.
			\item The realization functor $\real:\D^b_{dg}(\Hc{\C}{d})\to \C$ is a quasi-equivalence.
		\end{enumerate}
		Similarly, the following are equivalent:
		\begin{enumerate}
			\item The $t$-structure on $\C$ is strictly $(d-1)$-complicial.
			\item The realization functor $\D^b_{dg}(\Hc{\C}{d})\to \C$ is a quasi-equivalence, but the realization functor $\D^b(\Hc{\C}{d-1})\to \C$ is not.
		\end{enumerate}
	\end{theorem}
	
	\begin{proof}
		It is enough to prove the first part of the statement. We first assume that $\real$ is a quasi-equivalence. Consider $X\in \C^{\leq 0}$. Density of $\real$ implies that there exists $T\in\D^b(\Hc{\C}{d})$, such that $X\cong \real(T)$. Moreover, \cref{theorem:Co-T-Structure} and the construction of $\D^b(\Hc{\C}{d})$ as a Verdier quotient of $\tr(\Hc{\C}{d})$ yield $T'\in \He{\C}{d}$ as well as $T''\in {\He{\C}{d}}[1]\ast{\He{\C}{d}}[2]\ast\ldots\ast {\He{\C}{d}}[n]$ for some $n\geq 1$ and a triangle 
		\[T'\overset{f'}{\longrightarrow}T{\longrightarrow} T'' {\longrightarrow}T'[1]\]
		in $\D^b(\Hc{\C}{d})$. By exactness of $\real$, we obtain a triangle in $H^0(\C)$ 
		\[\real(T')\overset{f}{\longrightarrow}X{\longrightarrow} \real(T'') {\longrightarrow}\real(T')[1].\]
		Then $H^0(f)$ is surjective since $\real(T'')\in \C^{\leq -1}$. This proves that the $t$-structure is $(d-1)$-complicial.
		
		Conversely, by \cref{DenseCoro}, it is enough to prove that $\real$ is dense. Hence, let $Z\in\C$. Since the $t$-structure on $\C$ is bounded, there exist $n_1,n_2\in\mathbb{Z}$ such that $Z\in\C^{[-n_1,n_2]}$. Then $X:=Z[n_2]\in\C^{[-n,0]}$ for $n=n_1+n_2$. We prove, inductively, that there exist $T\in\D^b(\Hc{\C}{d})$ and an isomorphism $g:X\stackrel{\cong}{\longrightarrow}\real(T)$ such that $\real$ induces the following isomorphism for all $W\in\He{\C}{d}$: 
		\[\Hom_{H^0(\C)}(g,W)\circ\real_{T,W}: \Hom_{\D^b(\Hc{\C}{d})}(T,W)\stackrel{\cong}{\longrightarrow} \Hom_{H^0(\C)}(X,W).\]
		If $n< d$, then $X=\real(X)$ and there is nothing to prove. Thus, assume $n\geq d$. By assumption, there exist $P\in \He{\C}{d}$ and a morphism $f\in \Hom_{H^0(\C)}(P,X)$ such that $H^0(f):H^0(P)\to H^0(X)$ is an epimorphism in $\heart_t$. Taking a cocone of $f$ we obtain a triangle 
		\[X'\overset{h}{\longrightarrow} P\overset{f}{\longrightarrow}X{\longrightarrow} X'[1]\]
		in $H^0(\C)$. Then $X'\in\C^{[-n+1,0]}$, as $H^0(f)$ is an epimorphism and $H^1(P)=0$. By induction, there exist $T'\in\D^b(\Hc{\C}{d})$, an isomorphism $g':X'\stackrel{\cong}{\longrightarrow}\real(T')$ and a morphism $h'\in\Hom_{\D^b(\Hc{\C}{d})}(T',P)$ such that $\real(h')g'=h$. Defining $T\coloneqq \cone(h')\in\D^b(\Hc{\C}{d})$ and applying the exact functor $\real$ yields the following commutative diagram of triangles in $H^0(\C)$.
		\[\begin{tikzcd}
			{X'} & {P} & X & {X'[1]} \\
			{\real(T')} & {P} & {\real(T)} & {\real(T')[1]}
			\arrow["h", from=1-1, to=1-2]
			\arrow["{g'}"', draw=none, from=1-1, to=2-1]
			\arrow["\cong", from=1-1, to=2-1]
			\arrow["f", from=1-2, to=1-3]
			\arrow[equals, from=1-2, to=2-2]
			\arrow[from=1-3, to=1-4]
			\arrow["g"', dashed, from=1-3, to=2-3]
			\arrow["{g'[1]}"', draw=none, from=1-4, to=2-4]
			\arrow["\cong", from=1-4, to=2-4]
			\arrow["{\real(h')}", from=2-1, to=2-2]
			\arrow[from=2-2, to=2-3]
			\arrow[from=2-3, to=2-4]
		\end{tikzcd}\]
		It follows that $g$ is an isomorphism. Finally, let $W\in\He{\C}{d}$ and apply $\Hom(-,W)$ to the above diagram to obtain the following commutative diagram of long exact sequences.
		\[\adjustbox{scale=0.67}{\begin{tikzcd}
			{\Hom_{H^0(C)}(X',W)} & {\Hom_{H^0(C)}(P,W)} & {\Hom_{H^0(C)}(X,W)} & {\Hom_{H^0(C)}(X'[1],W)} & {\Hom_{H^0(C)}(P[1],W)} \\
			{\Hom_{H^0(C)}(R(T'),W)} & {\Hom_{H^0(C)}(P,W)} & {\Hom_{H^0(C)}(R(T),W)} & {\Hom_{H^0(C)}(R(T')[1],W)} & {\Hom_{H^0(C)}(P[1],W)} & {} \\
			{\Hom_{\D^b(\Hc{\C}{d})}(T',W)} & {\Hom_{\D^b(\Hc{\C}{d})}(P,W)} & {\Hom_{\D^b(\Hc{\C}{d})}(T,W)} & {\Hom_{\D^b(\Hc{\C}{d})}(T'[1],W)} & {\Hom_{\D^b(\Hc{\C}{d})}(P[1],W)}
			\arrow[from=1-2, to=1-1]
			\arrow[from=1-3, to=1-2]
			\arrow[from=1-4, to=1-3]
			\arrow[from=1-5, to=1-4]
			\arrow["{\Hom_{H^0(C)}(g',W)}", from=2-1, to=1-1]
			\arrow[equals, from=2-2, to=1-2]
			\arrow[from=2-2, to=2-1]
			\arrow["{\Hom_{H^0(C)}(g,W)}", from=2-3, to=1-3]
			\arrow[from=2-3, to=2-2]
			\arrow["{\Hom_{H^0(C)}(g'[1],W)}", from=2-4, to=1-4]
			\arrow[from=2-4, to=2-3]
			\arrow[equals, from=2-5, to=1-5]
			\arrow[from=2-5, to=2-4]
			\arrow["\real_{T',W}", from=3-1, to=2-1]
			\arrow["\real_{P,W}", from=3-2, to=2-2]
			\arrow[from=3-2, to=3-1]
			\arrow["\real_{T,W}", from=3-3, to=2-3]
			\arrow[from=3-3, to=3-2]
			\arrow["\real_{T'[1],W}", from=3-4, to=2-4]
			\arrow[from=3-4, to=3-3]
			\arrow["\real_{P[1],W}", from=3-5, to=2-5]
			\arrow[from=3-5, to=3-4]
		\end{tikzcd}}\]
		The claim now follows by induction and the $5$-lemma.
	\end{proof}
	
	We now consider the special case of a $t$-structure defined by a silting collection. Let $\C$ be a pretriangulated dg category and assume $\KST\coloneqq H^0(\C)$ is Krull--Schmidt and $\Hom$-finite. Assume $t=t_{\DP}=(\KST^{\leq 0}, \KST^{\geq 0})$ is a bounded silting $t$-structure with length heart for a silting collection $\DP$. As explained in \cref{t-str-ch} this means that, for $\DL$ the collection of simples in $\heart_{t}$, there is an equality $t=t_{\DL}$. 
	
	\begin{coro}\label{CorollarySilting}
		The following are equivalent: 
		\begin{enumerate}
			\item\label{item:Silting} $\DP\subseteq \KST^{(-d,0]}$.
			\item\label{item:t_complicial} $t$ is $(d-1)$-complicial.
			\item The realization functor 
			\[\real:\D^b_{dg}(\Hc{\C}{d})\longrightarrow \C\]
			is a quasi-equivalence.
			\item $\KST^{(-d,0]}\subseteq \KST$ is $1$-generated.
		\end{enumerate}
	\end{coro}
	
	\begin{proof}
		By \cref{compliciality} and \cref{GeneralisedRealFunctor} it is enough to prove $\ref{item:Silting}\Leftrightarrow\ref{item:t_complicial}$.
		
		$\ref{item:Silting}\Rightarrow \ref{item:t_complicial}$: It follows from \cite[Thm.~3.16(iv)]{Bon24} that $\heart_t$ has enough projectives and, for all projective objects $\Tilde{P}\in\heart_t$, there exists $P\in\DP$ with $H^0(P)=\Tilde{P}$. By \cite[Lem~3.2(2)]{Bon24} for $X\in\D^{\leq 0}$ and $P\in\DP$ the functor $H^0$ induces isomorphisms
		\[\Hom_{\D}(P,X)\cong \Hom_{\heart_t}(H^0(P),H^0(X)).\]
		It follows that $t$ is $(d-1)$-complicial.
		
		$\ref{item:t_complicial}\Rightarrow \ref{item:Silting}$: Let $P\in\DP$. By assumption there exists $X'\in\D^{(-d,0]}$ and a morphism $f:X'\to P$ such that $H^0(f):H^0(X')\to H^0(P)$ is an epimorphism in $\heart_t$. Since $H^0(P)\in\heart_t$ is projective by \cite[Lem.~3.3]{Bon24} there exists a section $g':H^0(P)\to H^0(X')$ of $H^0(f)$. By \cite[Lem.~3.2(2)]{Bon24}, there exists $g:P\to X'$ such that $H^0(g)=g'$ and thus $H^0(fg)=\id_{H^0(P)}$. Thus by the same lemma it holds $fg=\id_P$ and hence $P$ is a direct summand of $X'$ and consequently in $\KST^{(-d,0]}$.
	\end{proof}

	\section{Koszul duality for dg algebras}
	
	In the following section $A$ is a proper connective dg $k$-algebra. The main goal of this section is to explain how the property that the cohomology of $A$ is concentrated in an interval $(-d,0]$ can be characterized in terms of the minimal $A_{\infty}$-structure on its Koszul dual. We will first recall Koszul duality between connective and coconnective dg algebras following \cite{F25}, and analyze how it restricts to our setting. Afterwards, we explain how our findings can be translated to a minimal model on the Koszul dual side. 
	
	\subsection{Koszul duality for locally-finite connective dg algebras}
	In this subsection we briefly summarize the results from \cite{F25}. We use the following notation from \cite{F25}. For a locally-finite, connective dg algebra $A$ we denote by \define{$S_A$} the dg $A$-module
	\[S_{A}\coloneqq\topf(H^0(A))\in\fd(A).\] 
	For a locally-finite, coconnective dg algebra $E$ we define
	\[S_{E}\coloneqq H^0(E)\in\fd(E).\]
	Following \cite[\S10.2]{Kel94} we define the \definef{Koszul dual} of a locally-finite connective dg algebra $A$ by 
	\[\define{\KoszulDual{A}}\coloneqq \End_{\D_{dg}(A)}(S_{A}).\]
	Similarly, we define the \definef{Koszul dual} of a locally-finite coconnective dg algebra $E$ such that $S_E=H^0(E)$ is semisimple by
	\[\define{\KoszulDual{E}}\coloneqq \End_{\D_{dg}(E)}(S_{E}).\]
	For $A$ and $E$ as above, we consider the \definef{Koszul duality dg functors}
	\begin{align*}
		&\define{\Phi_A}\coloneqq \Hom_{\D_{dg}(A)}(-,S_A)\colon \D_{dg}(A)\longrightarrow \D_{dg}((\KoszulDual{A})^{op})^{op}, \\ 
		&\define{\Phi_E}\coloneqq \Hom_{\D_{dg}(E)}(-,S_E)\colon \D_{dg}(E)\longrightarrow \D_{dg}((\KoszulDual{E})^{op})^{op}.
	\end{align*}

	\begin{theorem}[{\cite[Thm.~4.3]{F25}}]\label{Thm:Fushimi_4.3}
		Koszul duality yields a bijective correspondence between quasi-isomorphism classes of the following two classes of dg algebras.
		\begin{enumerate}
			\item Connective locally-finite dg algebras $A$.
			\item Coconnective locally-finite dg algebras $E$ such that $H^0(E)$ is semisimple and $\fd(E)$ is $\Hom$-finite.
		\end{enumerate}
	\end{theorem}

	\begin{theorem}[{\cite[Thm.~4.4]{F25}}]\label{Thm:Fushimi_4.4}
		Let $A$ be a locally-finite connective dg algebra. Then the dg functor $\Phi_A$ induces equivalences of pretriangulated dg categories as follows
		\begin{align*}
			\Phi_A\colon \fd_{dg}(A)\stackrel{\simeq}{\longrightarrow}\per_{dg}((\KoszulDual{A})^{op})^{op}, \quad \Phi_A\colon\per_{dg}(A)\stackrel{\simeq}{\longrightarrow}\fd_{dg}((\KoszulDual{A})^{op})^{op}.
		\end{align*}
		Moreover, $\Phi_A(S_A)\cong(\KoszulDual{A})^{op}$ and $\Phi_A(A)\cong S_{(\KoszulDual{A})^{op}}$ and hence the composition of $\Phi_A$ with $\Hom_{\D_{dg}((\KoszulDual{A})^{op})}(-,(\KoszulDual{A})^{op})$ yields a $t$-exact equivalence of pretriangulated dg categories
		\[\overline{\Phi}_A\colon \fd(A)\xrightarrow{\simeq}\per_{dg}(\KoszulDual{A}),\]
		such that $\overline{\Phi}_A(S_A)\cong \KoszulDual{A}$.
	\end{theorem}

	\subsection{Koszul duality for proper connective dg algebras}\label{section:KoszulDualityProper}
	
	In the following we explain how the above bijection restricts to connective dg algebras $A$ that are proper and coconnetive locally-finite dg algebras $E$ with semisimple $H^0(E)$, satisfying suitable homological finiteness conditions, that we introduce here. 
	
	\begin{defi}
		Let $\A$ be a dg category. We call $\A$ \definef{(homologically) smooth} if the diagonal bimodule $_a\A_b={\A}(b,a)$ is perfect, that is $\A\in\per(\A\otimes_{k}\A^{op})$. We call $\A$ \definef{Hfd-closed} if $\fd(\A)\subseteq \per(\A)$. Finally, we call $\A$ \definef{regular} if $\per(\A)=\thick^{n}(\{\repr{A}:A\in\A\})$ for some $n\in\NN$. Here we define, for a triangulated category $\T$ and a collection of objects $\C\subseteq \T$, 
		\[\define{\thick^1(\C)}\coloneqq \add(\{C[i]:C\in\C,i\in\ZZ\}), \quad \define{\thick^n(\C)}\coloneqq \add(\thick^{n-1}(\C)\ast\thick^1(\C)) \text{ for }n\geq 2.\]
	\end{defi}
	
	In general, we have the following relation between the above notions. 
	\begin{equation}\label{eq:IncompleteDiagram}
		\begin{tikzcd}
			{\A \text{ is smooth}} && {\A \text{ is Hfd-closed}} \\
			& {\A \text{ is regular}}
			\arrow["{\text{\cite[Lem.~3.8]{KS25}}}", Rightarrow, from=1-1, to=1-3]
			\arrow["{\text{\cite[Lem.~3.6]{Lu10}}}"{description}, Rightarrow, from=1-1, to=2-2]
		\end{tikzcd}
	\end{equation}
	
	Using the results in this section we complete the above diagram below in the case where $k$ is perfect and $\A=E$ is a coconnective locally-finite dg algebra such that $H^0(E)$ is semisimple.

	\begin{theorem}\label{KoszulDualityProperHfdClosed}
		Koszul duality yields a bijective correspondence between quasi-isomorphism classes of the following two classes of dg algebras.
		\begin{enumerate}
			\item Connective proper dg algebras $A$.
			\item\label{item:HfdClosed} Coconnective locally-finite Hfd-closed dg algebras $E$, such that $H^0(E)$ is semisimple.
		\end{enumerate}
	\end{theorem}
	
	\begin{proof}
		We first observe that for a locally finite dg algebra $E$, $\per(E)$ is $\Hom$-finite. Hence, for $E$ as in \ref{item:HfdClosed}, $\fd(E)\subseteq \per(E)$ is $\Hom$-finite. Thus, in view of \cref{Thm:Fushimi_4.3} it is enough to show that a connective locally-finite dg algebra $A$ is proper if and only if $\KoszulDual{A}$ is Hfd-closed. To this end, we observe that the equivalence $\Phi_A$ yields that 
		\[\per(A)\subseteq\fd(A) \text{ if and only if } \fd((\KoszulDual{A})^{op})^{op}\subseteq\fd((\KoszulDual{A})^{op})^{op}.\]
		Thus $A$ is proper if and only if $(\KoszulDual{A})^{op}$ is Hfd-closed. By \cite[Lem.~A.3]{Goo25} every proper connective dg algebra is quasi-isomorphic to a finite-dimensional connective dg algebra $B$. Finally, by \cite[Lem.~6.17]{Goo24}, $\KoszulDual{B}$ is Hfd-closed if and only if $(\KoszulDual{B})^{op}$ is Hfd-closed.
	\end{proof}
	
	Under the assumption, that $k$ is a perfect field we obtain the following variant of \cref{KoszulDualityProperHfdClosed}, which we were not able to find in the literature in this form but should be known to experts. Observe that the implication ``\ref{item:A-side}$\Rightarrow$\ref{item:E-side}'' is proven in \cite[Prop.~6.12]{Goo24}. The same proof applies in our setting. We include it for the convenience of the reader. Observe that this does not give a new proof that $\fd(A)$ is smooth for a proper connective dg algebra over a perfect field (see \cite[Prop.~6.9(ii)]{KS25}), as the proof is based on this result.
	
	\begin{theorem}\label{KoszulDualityProperSmooth}
		Let $k$ be a perfect field. Koszul duality yields a bijective correspondence between quasi-isomorphism classes of the following two classes of dg algebras.
		\begin{enumerate}
			\item\label{item:A-side} Connective proper dg algebras $A$.
			\item\label{item:E-side} Coconnective locally-finite, smooth dg algebras $E$, such that $H^0(E)$ is semisimple.
		\end{enumerate}
	\end{theorem}
	
	\begin{proof}
		Clearly all connective proper dg algebras are in the first class of \cref{Thm:Fushimi_4.3}. Moreover, by \cite[Lem.~4.1]{KW23} if $E$ is smooth, then $\fd(E)$ is $\Hom$-finite. Thus, using the bijection from \cref{Thm:Fushimi_4.3}, it is enough to show that proper and smooth correspond to each other under this correspondence. 
		
		First assume that $A$ is proper. In \cite[Prop.~6.9(ii)]{KS25} (here we need the perfect field assumption) it is proven that $\per_{dg}(E)\cong\fd_{dg}(A)$ is smooth. Since $\per_{dg}(E)\subseteq \D_{dg}(E)$ consists of a set of generators of $\D_{dg}(E)$ by \cite[Lem.~2.1]{KS25} it holds that the functor
		\[\D_{dg}(E)\longrightarrow \D_{dg}\left(\per_{dg}(E)\right), \quad P\mapsto \Hom_{\D_{dg}(E)}(-,P)\] 
		defines an equivalence of pretriangulated dg categories. By \cite[Thm.~3.17]{LS14} smoothness is preserved under dg Morita equivalence and hence $E$ is smooth. Thus we have shown that properness of $A$ implies smoothness of $E$. On the other hand, if $E$ is smooth, then $A$ is proper by \cite[Cor.~4.2.3.,Rem.~4.2.4]{KL24}.
	\end{proof}
	
	As applications we record the following two corollaries. \cref{coro:LoeweyDfd} is proven more directly in \cite[Lem.~5.9]{RS22}.
	
	\begin{coro}\label{coro:LoeweyDfd}
		Let $k$ be a perfect field. Let $A$ be a proper connective dg algebra. Then there exists $n\in\mathbb{N}$ such that 
		\[\fd(A)=\thick^n(S_A).\]
	\end{coro}
	
	\begin{proof}
		By \cref{KoszulDualityProperSmooth} we know that $E\coloneqq \KoszulDual{A}$ is smooth. Hence, by \cite[Lem.~3.6]{Lu10} it follows that $\per(E)=\thick^n(E)$ for some $n\in\mathbb{N}$. But the equivalence $\overline{\Phi}_A:\fd(A)\cong \per(E)$ sends $S_A$ to $E$, hence it follows that $\fd(A)=\thick^n(S)$
	\end{proof}
	
	\begin{coro}\label{coro:Charsmooth}
		Let $k$ be a perfect field. Let $E$ be a coconnective, locally-finite dg algebra such that $H^0(E)$ is semisimple. Then $E$ is smooth if and only if it is Hfd-closed.
	\end{coro}
	
	\begin{proof}
		This follows from combining \cref{KoszulDualityProperSmooth} and \cref{KoszulDualityProperHfdClosed}. 
	\end{proof}
 	
 	The following example shows that regularity of $E$ does not imply that $E$ is Hfd-closed or even smooth. The example is inspired by the example given in \cite[Paragraph after Lem.~3.8]{KS25}.
 	
 	\begin{ex}
 		Consider the Kronecker quiver with arrows $a_i$ of degree $|a_i|=-i$ for $i\in\NN_{>0}$, that is a graded quiver $Q\coloneqq(Q_0,Q_1)$ defined by
 		\[Q_0=\{1,2\}, \quad Q_{1}\coloneqq \{a_i:1\to 2 : i\in\NN\}, \quad |a_i|\coloneqq -i.\]
 		We define the graded path-algebra $A\coloneqq kQ$ and $E\coloneqq \KoszulDual{A}$. We observe that $A$ is a formal, connective, locally-finite dg algebra that is not proper. Hence, by \cref{Thm:Fushimi_4.3}, $E$ is a coconnective, locally finite dg algebra such that $H^0(E)=k^2$ and $\fd(E)$ is $\Hom$-finite. However, $E$ is not Hfd-closed by \cref{KoszulDualityProperHfdClosed} and therefore can not be smooth.
 		
 		We claim, however, that $E$ is regular. As in the proof of \cref{coro:LoeweyDfd}, one sees that regularity of $E$ is equivalent to $\fd(A)=\thick^n(S_A)$ for some $n\in\NN$. We claim that $\fd(A)=\thick^4(S_A)$. To this end let $X\in\fd(A)$. Then there exist $n\in\ZZ$ and $d\in\NN_{>0}$ such that $X[n]\in\fd(A)^{(-d,0]}$. We need to prove that $X[n]\in\thick^4(S_A)$. Using the notation $\trunc^{(-d,0]}\coloneqq\trunc^{>-d}\circ\trunc^{\leq 0}$ it is proven in \cref{coro:truncatedHearts} that $\fd(A)^{(-d,0]}\cong\fd(\trunc^{(-d,0]}A)^{(-d,0]}$. Now, $A_d\coloneqq\trunc^{(-d,0]}A$ is the graded path algebra $KQ_d$ of the quiver 
 		\[Q_d\coloneqq \begin{tikzcd}
 			1
 			\arrow[r, draw=none, "\raisebox{+1.5ex}{\vdots}" description]
 			\arrow[r, bend left,        "a_1"]
 			\arrow[r, bend right, swap, "a_{d-1}"]
 			&
 			2
 		\end{tikzcd}, \quad |a_i|=-i.\]
 		Moreover, using the standard graded projective bimodule resolution of $A_d$
 		\[0\longrightarrow \bigoplus_{i=1}^{d-1} \left(A_de_2 \otimes_k e_1A_d\right)(i) \longrightarrow \bigoplus_{i=1}^2 A_de_i\otimes_k e_iA_d \longrightarrow A_d\longrightarrow 0\]
 		we see that for $X_d\in\fd(A_d)$ it holds 
 		\[X_d\cong X_d\otimes^{\DL}_{A_d}A_d\in \thick^2(A_d)\]
 		and hence $\fd(A_d)=\thick^2(A_d)$. Finally, denoting by $S_i=\topf(H^0(e_iA))$ the simple object corresponding to the vertex $i$, the triangle 
 		\[\coprod_{i=1}^{d-1} S_1[i] \longrightarrow A_d \longrightarrow S_1 \oplus S_2 \longrightarrow \coprod_{i=1}^{d-1} S_1[i+1]\]
 		in $\fd(A_d)$ yields that $A_d\in\thick^2(S_{A_d})$ and hence $\fd(A_d)=\thick^4(S_{A_d})$.
 	\end{ex}
 	
 	Hence over a perfect field $k$ the diagram \eqref{eq:IncompleteDiagram} can be completed as follows
 	\[\begin{tikzcd}
 		{\A \text{ is smooth}} && {\A \text{ is Hfd-closed}} \\
 		& {\A \text{ is regular}}
 		\arrow["{\text{\cref{coro:Charsmooth}}}", Rightarrow, 2tail reversed, from=1-1, to=1-3]
 		\arrow["{\text{\cite[Lem.~3.6]{Lu10}}}"{description}, shift left=2, between={0.15}{1}, Rightarrow, from=1-1, to=2-2]
 		\arrow[Rightarrow, from=1-3, to=2-2]
 		\arrow["{\scaleto{\times}{12pt}}"{marking, yshift=-2.22pt}, shift left=3, curve={height=-8pt}, Rightarrow, from=2-2, to=1-1]
 		\arrow["{\scaleto{\times}{12pt}}"{marking, yshift=-2.22pt}, shift right=3, curve={height=8pt}, Rightarrow, from=2-2, to=1-3]
 	\end{tikzcd}\]

	We now analyse more carefully in which way this bijection restricts, if we want to determine the degrees in which $H^*(A)$ is concentrated.
	
	\begin{theorem}\label{theorem:KoszulD}
		Let $d\in\NN_{>0}$. Koszul duality yields a bijective correspondence between quasi-isomorphism classes of the following two classes of dg algebras:
		\begin{enumerate}
			\item Connective proper dg algebras $A$ such that $H^*(A)$ is concentrated in degrees $(-d,0]$.
			\item Coconnective locally-finite Hfd-closed dg algebras $E$ such that $H^0(E)$ is semisimple and $\per(E)^{(-d,0]}\subseteq\per(E)$ is $1$-generated.
		\end{enumerate}
		Moreover, the standard $t$-structure on $\fd_{dg}(A) \simeq \per_{dg}(E)$ is $(d-1)$-complicial or equivalently, the realization functor
		\[\real:\D_{dg}^b\left( \fd_{dg}(A)^{(-d,0]} \right)\longrightarrow \fd_{dg}(A)\]
		is a quasi-equivalence.
	\end{theorem}

	\begin{proof}
		Given the first part of the theorem, the equivalence of dg categories is in \cref{Thm:Fushimi_4.4}. The compliciality property follows from \cref{GeneralisedRealFunctor}.
		
		By \cref{KoszulDualityProperHfdClosed}, it is enough to prove that $A\in\fd(A)^{(-d,0]}$ if and only if $\per(E)^{(-d,0]}\subseteq\per(E)$ is $1$-generated. By \cref{CorollarySilting}, $A\in\fd(A)^{(-d,0]}$ if and only if $\fd(A)^{(-d,0]}\subseteq \fd(A)$ is $1$-generated. By the $t$-exact equivalence $\fd(A)\simeq\per(E)$ $\fd(A)^{(-d,0]}\subseteq \fd(A)$ is $1$-generated if and only if $\per(E)^{(-d,0]}\subseteq\per(E)$ is $1$-generated. 
	\end{proof}
	
	For a perfect field, using \cref{coro:Charsmooth}, \cref{theorem:KoszulD} can be stated as follows
	
	\begin{theorem}\label{theorem:KoszulDperfect}
		Let $k$ be a perfect field and $d\in\NN_{>0}$. Koszul duality yields a bijective correspondence between quasi-isomorphism classes of the following two classes of dg algebras:
		\begin{enumerate}
			\item Connective proper dg algebras $A$ such that $H^*(A)$ is concentrated in degrees $(-d,0]$.
			\item Coconnective locally-finite smooth dg algebras $E$ such that $H^0(E)$ is semisimple and $\per(E)^{(-d,0]}\subseteq\per(E)$ is $1$-generated.
		\end{enumerate}
		Moreover, the standard $t$-structure on $\fd_{dg}(A) \simeq \per_{dg}(E)$ is $(d-1)$-complicial or equivalently, the realization functor
		\[\real:\D_{dg}^b\left( \fd_{dg}(A)^{(-d,0]} \right)\longrightarrow \fd_{dg}(A)\]
		is a quasi-equivalence.
	\end{theorem}

	\begin{rema}
		The combination of the above theorems implies in particular, that, for a perfect field $k$, for a coconnective locally-finite smooth dg algebra $E$ such that $H^0(E)$ is semisimple there exists $d\geq 1$ such that $\per(E)^{(-d,0]}\subseteq \per(E)$ is $1$-generated.
	\end{rema}

	\subsection{Coherently-generated $A_{\infty}$-algebras}\label{subsection:CohGen}
	
	In the following section we explain how to interpret the condition that $\per(E)^{(-d,0]}\subseteq \per(E)$ is $1$-generated as a condition on the minimal $A_{\infty}$-structure on $H^{\ast}(E)$ which allows us to obtain a characterization of $A_{\infty}$-algebras arising as Koszul-duals of proper connective dg-algebras $A$ as well as determining in which interval the cohomology of $A$ is concentrated. The motivation comes from the following theorem of Keller (for a proof, see \cite[Thm.~3.4.2]{KM25} or \cite{J24}).
	
	\begin{theorem}[{\cite[Sec.~2.2,Prop.~1(b)]{Kel02}}]
		Let $\Lambda$ be a finite-dimensional algebra. Let $S_1,\ldots,S_n$ be a complete set of isomorphism classes of simple $\Lambda$-modules and $S\coloneqq S_1\oplus\ldots\oplus S_n$. Then, the Yoneda algebra $\Ext^{\ast}_{\Lambda}(S,S)$ is generated in its homogeneous components of cohomological degrees $0$ and $1$ as an $A_{\infty}$-algebra.
	\end{theorem}
	
	A natural question to ask is whether, or under which additional assumptions, it is possible to reverse the above implication. The following example illustrates that the naive notion of being generated in degrees $0,\ldots,d$ is not enough to ensure that the cohomology of the Koszul dual is concentrated in degrees $(-d,0]$
	
	\begin{ex}\label{ex:Generation}
		Consider the graded algebra $E=kQ/(a_{n-1}a_{n-2}\cdots a_1)$ for $n\geq 4$, where $Q$ denotes the linearly oriented quiver of Dynkin type $A_n$,
		\[Q\coloneqq
		\begin{tikzcd}
			1 & 2 & 3 & \cdots & {n-1} & n
			\arrow["{a_1}", from=1-1, to=1-2]
			\arrow["{a_2}", from=1-2, to=1-3]
			\arrow["{a_3}", from=1-3, to=1-4]
			\arrow["{a_{n-2}}", from=1-4, to=1-5]
			\arrow["{a_{n-1}}", from=1-5, to=1-6]
		\end{tikzcd}\]
		with $|a_i|=1$. Then $E$ defines an $A_{\infty}$-algebra (with $m_n=0$ for $n\neq 2$), which is non-negative, locally-finite, strictly unital and generated in degrees $0$ and $1$. Moreover, it is homologically smooth, because there is a finite projective resolution of the $E\otimes_{k} E^{op}$-module $E$ of the following form (see \cite[Thm.~4.1]{Bar97})
		\[0\longrightarrow e_nE \otimes Ee_1 \longrightarrow \bigoplus_{a\in Q_1} e_{t(a)}E \otimes Ee_{s(a)}\longrightarrow \bigoplus_{i\in Q_0} e_iE\otimes Ee_i\longrightarrow E.\]
		However, this can not be the Koszul dual of a finite-dimensional algebra $\tilde{\Lambda}$. To see this one can either use \cite[Sec.~2.2,Prop.~1(a)]{Kel02} or consider degrees $0,1,2$ of $E$ to see it would have to hold $\tilde{\Lambda}\cong k\tilde{Q}/(\{b_ib_{i+1}:1\leq i\leq n-2\})$ where
		\[\tilde{Q}\coloneqq 
		\begin{tikzcd}
			1 & 2 & 3 & \cdots & {n-1} & n
			\arrow["{b_1}"', from=1-2, to=1-1]
			\arrow["{b_2}"', from=1-3, to=1-2]
			\arrow["{b_3}"', from=1-4, to=1-3]
			\arrow["{b_{n-2}}"', from=1-5, to=1-4]
			\arrow["{b_{n-1}}"', from=1-6, to=1-5]
		\end{tikzcd}.\] 
		Here we use $n\geq 4$. However, the Koszul dual of this algebra is $kQ\ncong E$. 
		
		In fact, one computes 
		\[\Hom_{\D(E)}(S_n,S_1[3-n])\neq 0,\] 
		and therefore $H^{*}(\KoszulDual{E})$ is not concentrated in degrees $(3-n,0]$. By \cref{KoszulDualityProperHfdClosed} $\per(E)^{(3-n,0]}\subseteq\per(E)$ is not $1$-generated.
		
		This example shows that, although the present graded algebra is generated by its degree $0$ and $1$ part, it is not coherently generated in degrees $0,1,\ldots,n-3$, a term that will be defined below. For example, for $n=4$, the following morphism of degree $2$ between objects in $H^0(\tw(E)^{(-1,0]})$ can not be factored into degree $1$ morphisms between objects in $H^0(\tw(E)^{(-1,0]})$
		
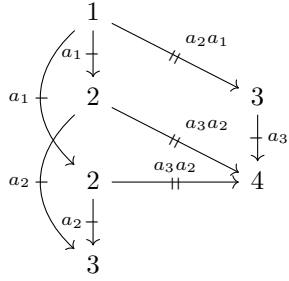
\[\begin{tikzcd}
			1 \\
			2 && 3 \\
			2 && 4 \\
			3
			\arrow["{a_1}"'{inner sep=.8ex}, "\shortmid"{marking}, from=1-1, to=2-1]
			\arrow["{a_2a_1}"{inner sep=.8ex}, "{\shortmid\shortmid}"{marking}, from=1-1, to=2-3]
			\arrow["{a_1}"'{inner sep=.8ex}, "\shortmid"{marking}, curve={height=24pt}, from=1-1, to=3-1]
			\arrow["{a_3a_2}"{inner sep=.8ex}, "{\shortmid\shortmid}"{marking}, from=2-1, to=3-3]
			\arrow["{a_2}"'{inner sep=.8ex}, "\shortmid"{marking}, curve={height=24pt}, from=2-1, to=4-1]
			\arrow["{a_3}"{inner sep=.8ex}, "\shortmid"{marking}, from=2-3, to=3-3]
			\arrow["{a_3a_2}"{inner sep=.8ex}, "{\shortmid\shortmid}"{marking}, from=3-1, to=3-3]
			\arrow["{a_2}"'{inner sep=.8ex}, "\shortmid"{marking}, from=3-1, to=4-1]
		\end{tikzcd}\]
	\end{ex}

	In the rest of this section we explain in which sense \cref{KoszulDualityProperSmooth} answers the question we posed above. In fact, Keller's original (unpublished) proof, which is explained in \cite[Thm.~3.4.2]{KM25}, shows that $\tw(\Ext^{\ast}_{\Lambda}(S,S))^{(-1,0]}\subseteq\tw(\Ext^{\ast}_{\Lambda}(S,S))$ is $1$-generated and reformulates this is in terms of the $A_{\infty}$-structure on $\Ext^{\ast}_{\Lambda}(S,S)$.
	
	In the following we consider a basic proper connective dg algebra $A$ and $E\coloneqq\KoszulDual{A}$. Using \cref{theorem:HTT}, we obtain a strictly unital coconnective minimal $A_{\infty}$-algebra $\E\coloneqq (H^{\ast}(E),m^{\E}_2,m^{\E}_3,\ldots)$. In the following we use the notation 
	\[\E^{[0,n]}\coloneqq \bigoplus_{i=0}^n\E^i.\] 
	Recall that, by the universal property of $\per_{dg}(E)$ and $\tw(\E)$, it holds $\per_{dg}(E)\simeq \tw(\E)$. 
	
	By definition, a morphism $f:T_0\to T_n[n]$ in $H^0(\tw(\E))$ with $T_0,T_n\in\tw(\E)^{(-d,0]}$ corresponds to a matrix with entries in $\E^{[0,n+d-1]}$. Recall that
	\[m^{\tw(\E)}_2=\sum_{t=0}^{\infty}\sum \pm m^{\E}_{2+t}(\id^{\otimes i_1} \otimes \delta^{\otimes j_1}\otimes \cdots \otimes \id^{\otimes i_r}\otimes \delta^{\otimes j_r}).\]
	The condition that $\tw(\E)^{(-d,0]}$ is $1$-generated means that the morphism $f$ can be written as
	\begin{equation}\label{eq:CohGen}
		f=m^{\tw(\E)}_2(f_n,\ldots,m^{\tw(\E)}_2(f_3,m^{\tw(\E)}_2(f_2,f_1)))
	\end{equation} 
	for some $f_i:T_{i-1}[i-1]\to T_i[i]$ in $H^0(\tw(\E))$ with $1\leq i\leq n$, and $T_i\in\tw(\E)^{(-d,0]}$.
	Each of the $f_i$ is by definition a matrix with entries in $\E^{[0,d]}$, hence \eqref{eq:CohGen} may be interpreted as a condition that asserts that $\E$ is \definef{coherently generated in degrees $0,1,2,\ldots,d$}.

	\leaveout{\begin{lemma}[{\cite[Lem.~4.1]{SHYD19}}]
		Let $E$ be a strictly unital mininmal coconective $A_{\infty}$-algebra. Then $\{e_1E,\ldots,e_nE\}$ is a simple-minded collection in $\per(E)$
	\end{lemma}
	
	It holds $\per_{dg}(E)\simeq \per_{A_{\infty}}(\Einfty)\simeq \tw(\Einfty)$ by the universal property of $\tw(\Einfty)$.
	
	For $p\in\ZZ$, we use the notation 
	\[\define{\Einfty^p(i,j)}\coloneqq e_j\Einfty^p e_i\coloneqq \left \{m_2(e_j,m_2(e,e_i))\stackrel{m_1=0}{=}m_2(m_2(e_j,e),e_i):e\in\Einfty^p \right\}.\]
	
	Recalling \cref{Def:Tw}, we first observe that we may identify objects $\tw(\Einfty)$ with triples
	\begin{gather*}
		(x,v,r)=\left((x_{ij})_{1\leq i,j\leq n},(v_i)_{1\leq i \leq n},(r_i)_{1\leq i\leq n}\right),\\ 
		n\in\NN, r_i\in\ZZ, v_i\in\{1,\ldots r\}, x_{ij}\in\Einfty^{r_i-r_j+1}(v_j,v_i),x_{ij}=0 \text{ for }i \geq j
	\end{gather*}
	such that 
	\begin{equation}\label{eq:TwObject}
		\sum_{t=1}^{\infty}(-1)^{\frac{t(t-1)}{2}} m^{\Mat{\Einfty}}_t(x^{\otimes t})=0.
	\end{equation}
	Under this identification, morphisms of degree $m$ between two objects $(x,v,r),(y,v,r')$ in $\tw(\Einfty)$ can be identified with \definef{matrices of degree $m$} 
	\[f=(f_{ij}), \quad f_{ij}\in\Einfty^{m+r'_i-r_j}(v_j,v'_i)\]
	and $m^{\tw(\Einfty)}_n$ is defined as in \cref{Def:Tw}.
	
	Using the description of the $t$-structure of the simple-minded collection $\{\Einfty e_1,\ldots, \Einfty e_r\}$ from \cref{t-str-ch}, the objects of $\tw(\Einfty)^{(-d,0]}$ correspond to triples $(x,v,r)$ that satisfy \ref{eq:TwObject} such that
	\[0\leq r_n \leq r_{n-1} \leq \ldots \leq r_1 \leq d-1.\]
	In particular $x_{ij}\in\Einfty^{[0,d]}(v_j,v_i)$ for all $1\leq i,j\leq n$
	
	Now, the condition that $\per(E)^{(-d,0]}\subseteq \per(E)$ is $1$-generated can be written out as follows. Given objects $(x,v,r)$ and $(y,v',r')$ in $\tw(\Einfty)^{(-d,0]}$ and a matrix of degree $n$, $f\in\tw(\Einfty)^n(x,v)$ such that $m^{\tw(\Einfty)}_1(f)=0$, there exist objects 
	\[(x,v,r)=(x^0,v^0,r^0),(x^1,v^1,r^1), (x^2,v^2,r^2),\ldots, (x^{n},v^n,r^n)=(y,v',r')\]
	in $\tw(\Einfty)^{(-d,0]}$ and matrices $f^k:(x^k,v^k,r^k)\to (x^{k+1},v^{k+1},r^{k+1})$ of degree $1$ such that
	\[m_1^{\tw(\Einfty)}(f^k)=0, \quad f=f^{n-1}\overset{\tw}{\circ}\ldots\overset{\tw}{\circ}f^0 \quad \text{in }H^0(\Einfty).\] 
	Here we define for $h,g\in H^0(\tw(\Einfty))$ 
	\[h\overset{\tw}{\circ}g=m^{\tw}_2(h,g).\] 
	One can easily check that $\overset{\tw}{\circ}$ is assocative, which means that the above is well defined.
	
	We can simplify the above condition by observing that $\per(E)^{(-d,0]}$ is $1$-generated in $\per(E)$ if and only if for all $X,Z\in\per(E)^{(-d,0]}$, and $f:X\to Z[2]$ in $\per(E)$, there exists $Y\in\per(E)^{(-d,0]}$ and morphisms $f^0:X\to Y[1]$ and $f^1:Y[1]\to Z[2]$ such that $f=f^1\circ f^0$. That is, given a matrix $f$ of degree $2$ with $m^{\tw}_1(f)=0$, there exist matrices $f^0,f^1,h$ of degree $1$ such that it holds
	\[f=m^{\tw(\Einfty)}_1(h)+m^{\tw(\Einfty)}_2(f^1,f^0).\]
	
	Pictorially, the situation looks as follows. For simplicity, we depict the case where $X$ is a shifted simple and $Y$ and $Z$ can be build as the extension of two shifted simples. 
	\[\begin{tikzcd}
		X && Y && Z \\
		&& {v^1_2[r^1_2]} && {v^2_2[r^2_2]} \\
		{v^0_1[r^0_1]} && {v^1_1[r^1_1]} && {v^2_1[r^2_1]}
		\arrow["{f^0}"{inner sep=.8ex}, "\shortmid"{marking}, from=1-1, to=1-3]
		\arrow["f"{inner sep=.8ex}, "{\shortmid\shortmid}"{marking}, curve={height=-24pt}, from=1-1, to=1-5]
		\arrow["h"{inner sep=.8ex}, "\shortmid"{marking}, curve={height=18pt}, from=1-1, to=1-5]
		\arrow["{f^1}"{inner sep=.8ex}, "\shortmid"{marking}, from=1-3, to=1-5]
		\arrow["{\scalemath{0.7}{f^1_{22}}}"{inner sep=.8ex}, "\shortmid"{marking}, from=2-3, to=2-5]
		\arrow["{x^1_{12}}"{inner sep=.8ex}, "\shortmid"{marking}, from=2-3, to=3-3]
		\arrow["{\scalemath{0.7}{f^1_{12}}}"{pos=0.6, inner sep=.8ex}, "\shortmid"{marking}, curve={height=-6pt}, from=2-3, to=3-5]
		\arrow["{x^2_{12}}"{inner sep=.8ex}, "\shortmid"{marking}, from=2-5, to=3-5]
		\arrow["{\scalemath{0.7}{f^0_{21}}}"{pos=0.6, inner sep=.8ex}, "\shortmid"{marking}, from=3-1, to=2-3]
		\arrow["{\scalemath{0.7}{h_{21}}}"{pos=0.3, inner sep=.8ex}, "\shortmid"{marking}, curve={height=-40pt}, from=3-1, to=2-5]
		\arrow["{\scalemath{0.7}{f^0_{11}}}"{inner sep=.8ex}, "\shortmid"{marking}, from=3-1, to=3-3]
		\arrow["{\scalemath{0.7}{h_{11}}}"{inner sep=.8ex}, "\shortmid"{marking}, curve={height=18pt}, from=3-1, to=3-5]
		\arrow["{\scalemath{0.7}{f^1_{21}}}"{pos=0.3, inner sep=.8ex}, "\shortmid"{marking}, curve={height=6pt}, from=3-3, to=2-5]
		\arrow["{\scalemath{0.7}{f^1_{11}}}"{inner sep=.8ex}, "\shortmid"{marking}, from=3-3, to=3-5]
	\end{tikzcd}\]   
	So for example, it holds 
	\begin{align*}
		f_{11}= & \pm m_2(f^1_{11},f^0_{11})\pm m_2(f^1_{12},f^0_{21})\pm m_2(x^2_{12},h_{21})\\
		& \pm m_3(x^2_{12},f^{1}_{21},f^{0}_{11}) \pm m_3(,f^1_{11},x^1_{12},f^{0}_{21}) \pm  m_3(x^2_{12},f^{1}_{22},f^{0}_{21}) \pm m_4(x^2_{12},f^1_{21},x^1_{12},f^{0}_{21})
	\end{align*}
	and the signs are determined by \cref{Def:Tw}.
	
	In short, this can be interpreted as requiring that $\Einfty$ is generated in the degrees $0,1,\ldots,d$ in a coherent way. Hence, we call a strictly unital, locally-finite, minimal $A_{\infty}$-algebra $\Einfty$ \definef{coherently generated in degrees $[0,d]$} if $\tw(\Einfty)^{(-d,0]}\subseteq\tw(\Einfty)$ is $1$-generated. One might wonder if it is possible to simplify the definition and make it more explicit. Unfortunately, at the present moment we were not able to do so. The difficulty is illustrated in \cref{ex:Generation}.
	 
	Before we explain this example we remark that \cref{theorem:KoszulD} yields the following corollary.
	
	\begin{coro}\label{coro:KoszulD}
		Let $k$ be a perfect field. Koszul duality yields a bijective correspondence between quasi-isomorphism classes of the following two classes.
		\begin{enumerate}
			\item Connective, proper, basic dg algebras $A$, such that $H^*(A)$ is concentrated in degrees $(-d,0]$.
			\item Coconnective, locally-finite, smooth, strictly unital, minimal $A_{\infty}$-algebras $\Einfty$ that are coherently generated in degrees $[0,d]$.
		\end{enumerate}
	\end{coro}}

	\section{Abelian $d$-truncated dg categories}\label{sec:AbelianD}
	
	Before we generalize the above findings to certain (length) exact dg categories we investigate the notion of abelian $d$-truncated dg categories. In the following, we will use the following notation.
	
	\begin{nota}
		Let $\A$ be an additive dg category. For an object $A\in\A$, the homotopy cokernel of the morphism $A\to 0$, if it exists, is denoted by $\define{\Sigma A}$. Dually, the homotopy kernel of the morphism $0\to A$, if it exists, is denoted by $\define{\Omega A}$.
	\end{nota}
	
	Recall the following notion from \cite{Moc25}, which is a natural generalization of abelian categories to the dg setting.
	
	\begin{defi}[{\cite[Def.~3.12]{Moc25}}]
		Let $\A$ be an additive dg category. We call $\A$ \definef{$d$-truncated}, if $H^i(\A(x,y))=0$ for all $i\notin(-d,0]$. A $d$-truncated dg category $\A$ is an \definef{abelian $d$-truncated dg category} if it has the following properties:
		\begin{enumerate}
			\item $\A$ admits homotopy kernels and homtopy cokernels.
			\item Suppose $f\in Z^0(\A)(x,y)$ satisfies $\Omega^{d-1}(\hker(f))=0$. Then $f=\hker(\hcoker(f))$, i.e. there exists a homotopy bicartesian square of the following form.
			\[\begin{tikzcd}
				x & y \\
				0 & {\hcoker(f)}
				\arrow["f",from=1-1, to=1-2]
				\arrow[from=1-1, to=2-1]
				\arrow["h", from=1-1, to=2-2]
				\arrow[dashed, from=1-2, to=2-2]
				\arrow[from=2-1, to=2-2]
			\end{tikzcd}\]
			\item Suppose $f\in Z^0(\A)(x,y)$ satisfies $\Sigma^{d-1}(\hcoker(f))=0$. Then $f=\hcoker(\hker(f))$, i.e. there exists a homotopy bicartesian square of the following form.
			\[\begin{tikzcd}
				\hker(f) & x \\
				0 & y
				\arrow[dashed, from=1-1, to=1-2]
				\arrow[from=1-1, to=2-1]
				\arrow["h", from=1-1, to=2-2]
				\arrow["f",from=1-2, to=2-2]
				\arrow[from=2-1, to=2-2]
			\end{tikzcd}\]
		\end{enumerate}
	\end{defi}
	
	The corresponding $\infty$-categorical notion was introduced in \cite[Def.~6.2.4]{Ste23}.
	
	\begin{defi}[{\cite[Def.~6.2.4]{Ste23}}]
		Let $\A$ be a \definef{$(d,1)$-category}, that is an $\infty$-category such that $\pi_{n}(\A(x,y),f)$ is trivial for all $n\geq d$ and $x,y\in\A$, $f\in\A(x,y)$. We say that $\A$ is an \definef{abelian $(d,1)$-category} if the following conditions are satisfied:
		\begin{enumerate}
			\item $\A$ is additive and admits finite limits and colimits,
			\item\label{item:abelian(d,1)2} Let $f:x\to y$ be a \definef{$(d-2)$-truncated morphism} in $\A$, that is for all $a\in\A$ and $g\in\A(a,x)$, denoting by $[g]\in\pi_0(\A(a,x))$ the homotopy class of a morphism, the induced morphism
			\[\pi_n([f]_{*})\colon \pi_n({\A}(a,x),[g])\to \pi_n({\A}(a,y),[f]\circ[g])\] 
			is injective for $n=d-1$ and bijective for $n\geq d$. Then $f$ is the fiber of its cofiber, that is the cofiber sequence $x\to y\to \cofib(f)$ is also a fiber sequence.
			\item Dual of \ref{item:abelian(d,1)2}.
		\end{enumerate}
		We will denote by \define{$(d,1)\mathrm{Cat}_{\mathrm{ab}}$} the $\infty$-category of abelian $(d,1)$-categories and exact functors. 
	\end{defi}
	
	In \cite[Thm~3.44]{Moc25} it is proven that for a pretriangulated dg category $\C$ with a $t$-structure $t$ on $\C$, the $d$-extended heart $\C^{(-d,0]}$ is an abelian $d$-truncated dg category. The following proposition yields the converse, which is a generalisation of the well-known fact that an abelian category embeds as the heart of the natural $t$-structure on $\D^b(\A)$. The result is proven in \cite[Thm.~6.3.2]{Ste23} in the $\infty$-categorical setting. The proof we present relies heavily on the results in \cite{Ste23} and essentially consists of making the relevant comparisons between the dg and $\infty$-categorical notions. 
	
	\begin{theorem}\label{prop:CharAb}
		Let $\A$ be an abelian $d$-truncated dg category. Then, there exists a $t$-structure $t$ on $\D^b_{dg}(\A)$, such that the universal embedding $\A\hookrightarrow \D^b_{dg}(\A)$ induces a quasi-equivalence 
		\[\A\stackrel{\simeq}{\longrightarrow}\D^b_{dg}(\A)^{(-d,0]}.\]
		Thus abelian $d$-truncated dg categories are precisely $d$-extended hearts of $t$-structures on pretriangulated dg categories.
	\end{theorem}

	We first recall the necessary results from \cite{Ste23}, please see \cite{Ste23} for the definitions.
	
	\begin{theorem}[{\cite[Thm.~6.3.2.(1)]{Ste23}}]\label{theorem:Stefanich}
		Let $\mathrm{Cat}^b_{\mathrm{pst}}$ be the $\infty$-category of finitely-complete bounded prestable $\infty$-categories and exact functors and fix $d\geq 1$. The functor 
		\[(-)_{\leq d-1}:\mathrm{Cat}^b_{\mathrm{pst}}\longrightarrow (d,1)\mathrm{Cat}_{\mathrm{ab}}\] 
		admits a fully faithful left adjoint 
		\[\D^b_{\infty}(-)_{\geq 0}:(d,1)\mathrm{Cat}_{\mathrm{ab}}\longrightarrow \mathrm{Cat}^b_{\mathrm{pst}}.\]
		Furthermore an object belongs to the image of $\D^b_{\infty}(-)_{\geq 0}$ if and only if it is $(d-1)$-complicial. 
	\end{theorem}
	
	\begin{rema}
		Observe that \cite[Prop.~2.2.15.(3)]{Ste23} shows that in the setting of the above theorem our notion of $(d-1)$-compliciality from \cref{def:complicial} agrees with the one in \cite[Def.~2.3.8]{Ste23}.
	\end{rema}
	
	In \cite[Rem.~6.3.3]{Ste23} Stefanich defines for an abelian $(d,1)$-category $\A$ the stable category
	\[\D^b_{\infty}(\A)=\SW(\D^b_{\infty}(\A)_{\geq 0})\coloneqq \colim (\D^b_{\infty}(\A)_{\geq 0}\stackrel{\Sigma}{\to}\D^b_{\infty}(\A)_{\geq 0}\stackrel{\Sigma}{\to}\ldots )\]
	which is the Spanier--Whitehead $\infty$-category of $\D^b_{\infty}(\A)_{\geq 0}$ (see \cite[\S C.1.1]{Lurie18}). It follows from \cite[Prop.~C.1.2.9]{Lurie18} that $\D^b_{\infty}(\A)$ admits a (homologically indexed) bounded $t$-structure $t$ whose aisle is given by $\D^b_{\infty}(\A)_{\geq 0}$. Since $\D^b_{\infty}(-)_{\geq 0}$ is fully faithful, the unit 
	\[\eta\colon \A\stackrel{\simeq}{\longrightarrow} \D^b_{\infty}(\A)_{[0,d)}\coloneqq \D^b_{\infty}(\A)_{\geq 0} \cap \D^b_{\infty}(\A)_{\leq d-1}\]
	is an equivalence. In particular, $\A$ can be identified with an extension closed subcategory of an exact $\infty$-category, and is thus canonically an exact $\infty$-category in the sense of \cite{Bar15}. Alternatively, this can be shown directly from the definition of an abelian $(d,1)$-category. We denote by $\St(\A)$ the stable hull of $\A$ in the sense of \cite[Def.~3.1]{Kle22} and compare the two constructions $\D^b_{\infty}(\A)$ and $\St(\A)$. The universal property of $\St(\A)$ yields a canonical morphism 
	\[\St(\A)\longrightarrow \D^b_{\infty}(\A).\]
	
	\begin{prop}\label{Prop:St=Db}
		For an abelian $(d,1)$-category $\A$, the canonical morphism
		\[\St(\A)\stackrel{\simeq}{\longrightarrow}\D^b_{\infty}(\A)\]
		is an equivalence of stable $\infty$-categories.
	\end{prop}
	
	\begin{proof}
		Since $\D^b_{\infty}(-)_{\geq 0}$ is fully faithful, the unit 
		\[\eta\colon \A\stackrel{\simeq}{\longrightarrow} \D^b_{\infty}(\A)_{[0,d)}\coloneqq \D^b_{\infty}(\A)_{\geq 0} \cap \D^b_{\infty}(\A)_{\leq d-1}\]
		is an equivalence. Moreover, by \cref{theorem:Stefanich} the $t$-structure $t$ is $(d-1)$-complicial. We claim that $\A\subseteq \D^b_{\infty}(\A)_{\geq 0}$ is left special in the sense of \cite[Def.~1.1]{SW25}: For every deflation $P:x\defl a$ in $\D^b_{\infty}(\A)_{\geq 0}$ with $x\in\D^b_{\infty}(\A)_{\geq 0}$ and $a\in\A$ there exists a morphism $f:b\to x$ with $b\in\A$ such that $p\circ f:b\defl a$ is a deflation in $\A$. In that case, \cite[Thm.~1.2]{SW25} yields that the canonical functor
		\[\St(\A) \simeq \St(\D^b_{\infty}(\A)_{[0,d)}) {\longrightarrow} \St(\D^b_{\infty}(\A)_{\geq 0})\simeq\D^b_{\infty}(\A) \] 
		is fully faithful. Here the last equivalence holds since $\D^b_{\infty}(\A)_{\geq 0}$ is prestable and hence, by \cite[Cor.~3.8]{Kle22},
		\[\St(\D^b_{\infty}(\A)_{\geq 0})\simeq \SW(\D^b_{\infty}(\A)_{\geq 0})=\D^b_{\infty}(\A).\]
		As the $t$-structure on $\D^b_{\infty}(\A)$ is bounded and $\D^b_{\infty}(\A)_{[0,d)}$ is the essential image of $\A$ under the above functor it is dense as well.
		
		Hence, it remains to prove that $\A\subseteq \D^b_{\infty}(\A)_{\geq 0}$ is left special. To this end, we consider a deflation $p:x\defl a$ with $x\in\D^b_{\infty}(\A)$ and $a\in\A$. By $(d-1)$-compliciality, there exists a morphism $f:b\to x$ in $\D^b_{\infty}(\A)_{\geq 0}$ with $b\in\D^b_{\infty}(\A)_{[0,d)}=\A$ such that $H_0(f):H_{0}(b)\to H_{0}(x)$ is an epimorphism in $\heart_t$. The long exact sequence in homology for the triangle
		\[\fib(f)\longrightarrow b\stackrel{f}{\longrightarrow} x \longrightarrow \fib(f)[1]\]
		in $\D^b(\A)$ yields that the fibre $\fib(f)$ of $f$ in $\D^b_{\infty}(\A)$ is in $\D^b_{\infty}(\A)_{\geq 0}$, which implies that $f$ is a deflation. Hence $p\circ f:b\defl a$ is a deflation as well.
	\end{proof}
	
	\begin{proof}[Proof of \cref{prop:CharAb}.]
		Denote by $N_{dg}$ the dg nerve functor in the sense of \cite[Constr.~3.1.1.6]{Lurie17}. By \cite[Prop.~4.19]{Moc25}, a connective dg category $\A$ is abelian $d$-truncated if and only if $N_{dg}(\A)$ is an abelian $(d,1)$-category. The universal embedding $F_{dg}:\A\hookrightarrow \D^b_{dg}(\A)$ yields an exact morphism 
		\[N_{dg}(F_{dg}):N_{dg}(\A)\longrightarrow N_{dg}(\D^b_{dg}(\A)).\] 
		By the universal property of $\St(N_{dg}(\A))$ there is a morphism $G_{\infty}$ such that the following diagram commutes
		\[\begin{tikzcd}
			{N_{dg}(\A)} && {\St(N_{dg}(\A))} \\
			& {N_{dg}(\D^b_{dg}(\A))}
			\arrow["{F_{\infty}}", hook, from=1-1, to=1-3]
			\arrow["{N_{dg}(F_{dg})}"', hook, from=1-1, to=2-2]
			\arrow["{G_{\infty}}", dashed, from=1-3, to=2-2]
		\end{tikzcd}\]
		By \cite[Thm.~6.25]{Chen23} the morphism $G_{\infty}$ is an equivalence of stable $\infty$-categories. By the above discussion and \cref{Prop:St=Db}, there exists a $t$-structure $t$ on $\St(N_{dg}(\A))\simeq N_{dg}(\D^b_{dg}(\A))$ such that $F_{\infty}$ induces an equivalence of $\infty$-categories 
		\[N_{dg}(\A)\stackrel{\simeq}{\longrightarrow} \St(N_{dg}(\A))_{[0,d)_t}.\] 
		Since the objects of $N_{dg}(\A)$ are the objects of $\A$, it follows that $F_{dg}$ induces a quasi-equivalence 
		\[\A\stackrel{\simeq}{\longrightarrow}\D^b_{dg}(\A)^{(-d,0]_t}. \qedhere\]
	\end{proof}
	
	\section{Length exact $d$-categories}\label{sec:LengthExact}
	
	In this section, we introduce the notion of length exact $d$-categories. Roughly speaking, we want to capture the exact $d$-truncated dg categories that are built from the simples in the exact category $\dis{\E}\subseteq \E$ of discrete objects defined below. One goal of this chapter to characterize when such exact dg categories are abelian $d$-truncated.
	 
	In this section, let $\E$ be an exact $d$-truncated dg category. 
	
	\begin{defi}
		We define the \definef{subcategory of discrete objects} $\dis{\E}\subseteq \E$ to be the full dg subcategory consisting of objects
		\[\dis{\E}\coloneqq\left\{X\in\E:H^n\E(-,X)=0 \text{ for all } n\neq 0\right\}.\]
	\end{defi}
	
	Since $\E$ is connective, by \cite[Thm.~6.1]{Chen23}, we can and we will identify $\tau^{\leq 0}\E$ with $\tau^{\leq 0}\D'$ for an extension closed dg subcategory $\D'$ of $\D^b_{dg}(\E)$. Under this identification, using a long exact sequence it is easy to see that $\dis{\E}$ is extension-closed, hence an exact dg category. By the following lemma, $\dis{\E}$ is quasi-equivalent to an exact ($1$-)category.
	
	\begin{lemma}\label{lemma:Exact1} Let $\E$ be an exact $1$-category. The following statements hold.
		\begin{enumerate}
			\item\label{item:Exact1} The canonical extriangulated structure on $H^0(\E)$ (see \cref{theorem:CanonExt}) is an exact structure. 
			\item\label{item:Abelian1} If $\E$ is an abelian $1$-category, then $H^0(\E)$ is abelian.
		\end{enumerate}
	\end{lemma}
	
	\begin{proof}
		We first prove \ref{item:Exact1}. With the canonical extriangulated structure $H^0(\E)$ is an extriangulated category with first negative extension 
		\[\EE^{-1}(X,Y):=\Hom_{\D^b(\E)}(X,Y[-1])=H^{-1}(\E)(X,Y)=0,\]  
		as $\E$ is a $1$-category. The assertion now follows from \cite[Prop.~2.6]{AET23}.
		
		In order to prove \ref{item:Abelian1}, we observe that by \cite[Thm.~3.21]{Moc25} every morphism $f$ in the exact category $H^0(\E)$ has a factorization $f=m\circ e$ with $e$ a deflation and $m$ an inflation and hence by \cite[Exercise~8.6]{Bue10} $H^0(\E)$ is abelian.
	\end{proof}
	
	Let $\C$ be a class of objects in $\E$. Following \cite{Z20} we denote by $\define{\Filt(\C)}$ the full subcategory of objects $X$ in $\E$ that are \definef{filtered by $\C$}, that is there exist objects $X_1,\ldots, X_n=X$ in $\E$, objects $C_1,\ldots,C_n$ in $\C$ and conflations 
	\[0\infl X_1\defl C_1\ext, \quad X_i\infl X_{i+1}\defl C_{i+1}\ext \text{ for } i>1.\]
	By \cite[Lem.~3.2]{Z20} $\Filt(\C)$ is the smallest extension-closed subcategory of $\E$ containing $\C$.
	
	We say that $\E$ is \definef{$d$-filtered} if $\dis{\E}[0,\ldots ,d-1]\subseteq\E$ (in $\D^b_{dg}(\E)$) and 
	\[\E=\Filt(\{X[i]:X\in\dis{\E},0\leq i\leq d-1\}).\]
	
	\begin{ex}\label{example:D-ComplicialExact}
		\begin{enumerate}
			\item An exact category $\E$ is $1$-filtered, since $\dis{\E}=\E$.
			\item Every abelian $d$-truncated dg category is $d$-filtered. Indeed, by \cref{prop:CharAb} every abelian $d$-truncated dg category can be identified with $\D^b_{dg}(\A)^{(-d,0]}$ for some $t$-structure $t$ on $\D^b_{dg}(\A)$. Then $\dis{\A}=(\D^b_{dg}(\A)^{<0})^{\perp}\cap\D^b_{dg}(\A)^{\leq 0}$ can be identified with $\heart_t$ and 
			\[\D^b(\A)^{(-d,0]}=\heart_t[d-1]\ast \ldots \ast \heart_t[1] \ast \heart_t.\]
		\end{enumerate}
	\end{ex}
	
	We assume $\E$ to be $d$-filtered in the following. For a class of objects $\C$ in $\dis{\E}$ and $1\leq n\leq d$, we define
	\[\define{\Filt_n(\C)}\coloneqq \Filt\left(\left\{ C[i]:C\in \C,0\leq i\leq n-1 \right\}\right)\subseteq \E.\]
	
	\begin{rema}\label{rema:Filt_n}
		In the above setting, there is an equality
		\[\Filt_n(\C)=\Filt_n(\Filt(\C)).\]
		Indeed, the inclusion ``$\subseteq$'' holds by definition, while the reverse inclusion holds using that we have the inclusion
		\[\{F[i]:F\in\Filt(C),1\leq i\leq n-1\}\subseteq \Filt_n(C),\]
		by shifting the triangles in a filtration of $F$.
	\end{rema}

	\begin{defi}
	Let $\E$ be $d$-filtered and $n\in\mathbb{N}$ such that $1\leq n\leq d$. We call a set of objects $\DL\subseteq\E$ an \definef{$n$-semibrick} if
	 \begin{enumerate}
	 	\item $\{L,L[1],\ldots,L[n-1]:L\in\DL\}\subseteq \E$,
	 	\item $H^{i}(\E)(L,L')=0$ for all $L,L'\in\DL$ and $i<0$, and
		\item for $L,L'\in\DL$
		\[H^0(\E(L,L'))=\begin{cases}0 & \text{ if }L\neq L', \\ \text{division algebra} &\text{ else.} \end{cases}\]
	\end{enumerate}
	We fix a complete collection of representatives \definef{$\Simples(\dis{\E})$} of the isomorphism classes of simple objects of the exact category $\dis{\E}$. We call $\E$ a \definef{length exact $d$-category} if $\Simples(\dis{\E})$ is a set and 
	\[\E=\Filt_d(\Simples(\dis{\E})).\]
	\end{defi}
	
	\begin{rema}
		Observe that the above definition of simples does not agree with either definition of simples in extriangulated categories given in \cite[Sec.~3]{WWZZ22} or \cite[Def.~3.1]{BHST24}. The reason for this is that we want to capture the case in which objects in $\E$ can be built from shifts of simple objcts in $\dis{\E}$. As illustrated in \cref{example:Simples}, simple objects in $\dis{\E}$ are not necessarily simple in the extriangulated category $H^0(\E)$ in the sense of the above references, hence the above notions are not suited for the properties we wish to study in this article.
	\end{rema}
	
	Let $1\leq n\leq d$. We call a full, extension-closed, exact dg subcategory $\W\subseteq \E$ \definef{wide $n$-subcategory} if $\W$ is abelian $n$-truncated with respect to the induced exact structure. For an wide $n$-subcategory $\W\subseteq\E$ the universal property of $\D^b_{dg}(\W)$ yields a canonical inclusion functor $\D^b_{dg}(\W)\hookrightarrow \D^b_{dg}(\E)$. Hence, we can and we will identify $\D^b_{dg}(\W)$ with a full exact dg subcategory of $\D^b_{dg}(\E)$.
	
	We obtain the following generalization of \cite[Thm.~2.5]{Eno21}. A variant of the proof we present here was explained in the case $d=1$ by Bodzenta at the Hausdorff Trimester Program: Symplectic Geometry and Representation Theory at the Hausdorff Center for Mathematics in Bonn in 2017; we were, however, not able to find it in the literature.
	
	\begin{theorem}\label{theorem:Enomoto}
		Let $\E$ be a length exact $d$-category and $1\leq n\leq d$. Then the assignments $\DL\mapsto \Filt_n(\DL)$ and $\W\mapsto \Simples(\dis{\W})$ define mutually inverse bijections between the following two classes.
		\begin{enumerate}
			\item The class of $n$-semibricks in $\E$.
			\item The class of length wide $n$-subcategories of $\E$.
		\end{enumerate}
	\end{theorem}
	
	\begin{proof}
		Let $\DL$ be a $n$-semibrick in $\E$. We prove that $\W\coloneqq \Filt_n(\DL)$ is a wide $n$-subcategory and $\Simples(\dis{\W})=\DL$. First observe that $\W= \Filt_n(\DL)\subseteq \E$ is clearly an extension closed exact dg subcategory of $\E$.
		
		We claim that $\DL$ defines a simple-minded collection in $\D^b_{dg}(\W)$ and for the corresponding $t$-structure $t$ on $\D^b_{dg}(\W)$, the extended heart $\D^b_{dg}(\W)^{(-n,0]}$ can be identified with $\W$ via the universal embedding. From this claim we get that $\W$ is an abelian $n$-category by \cite[Thm.~3.44]{Moc25}, since it is an extended heart of a $t$-structure. Moreover, under this identification the simple objects in $\dis{\W}=\heart_t$ are the objects of $\DL$ by \cite[Thm.~4.4]{Sch20}. Hence, it is enough to prove the above claim.
		
		We first observe that since $\DL$ is a $n$-semibrick, for $L,L'\in\DL$,
		\[\Hom_{\D^b(\W)}(L,L')\cong H^0(\W)(L,L')= H^0(\E)(L,L') =\begin{cases}0 & \text{ if }L\neq L',\\ \text{ division algebra} & \text{ if } L=L'.  \end{cases}\]
		Moreover, for all $L,L'\in\DL$ and $k>0$,
		\[\Hom_{\D^b(\W)}(L,L'[-k])=H^{-k}(\E)(L,L')=0.\]
		At last, $\thick(\DL)=\D^b(\W)$ since $\W=\Filt_n(\DL)\subseteq\thick(\DL)$ and thus
		\[\D^b(\W)=\thick(\W)\subseteq\thick(\DL).\]
		
		Finally, we need to show that $\W$ can be identified with $\D^b_{dg}(\W)^{(-n,0]}$. Recall that for the $t$-structure corresponding to $\DL$, the heart is given by $\heart_t=\Filt(\DL)$, the simple objects in $\heart_t$ are precisely the objects of $\DL$ and 
		\[\D^b(\W)^{(-n,0]}=\Filt(\DL)[n-1]\ast\Filt(\DL)[n-2]\ast\cdots \ast \Filt(\DL)\subseteq\Filt_n(\DL)=\W.\] 
		Since $\D^b(\W)^{(-n,0]}$ is extension-closed, and contains $\DL[i]$ for $0\leq i\leq n-1$, the reverse inclusion holds as well. This finishes the proof of the statement that $\DL\mapsto \Filt_n(\DL)$ is well-defined and a right inverse of $\W\mapsto \Simples(\dis{\W})$.
		
		Assume now $\W\subseteq \E$ is a length wide $n$-subcategory. Then $\W$ is a length abelian $n$-truncated dg category and $\DL\coloneqq\Simples(\dis{\W})$ is a $1$-semibrick by applying Schur's Lemma to $H^0(\dis{\W})$, which is an abelian category by \cref{lemma:Exact1}. Since $\DL\subseteq \dis{\W}$, for $L,L'\in\DL$ and $i<0$,
		\[H^{i}(\E)(L,L')=H^{i}(\W)(L,L')=0.\]
		Finally, since $\W$ is $n$-filtered, for $L\in\DL$ and $0\leq i<n$ we have $L[i]\in\W$.
		This yields that the map $\W\mapsto \DL$ is well defined. As $\W$ is length, $\Filt_n(\DL)=\W$. Thus, $\DL\mapsto \Filt_n(\DL)$ is also a left inverse of $\W\mapsto \Simples(\dis{\W})$.
	\end{proof}
	
	\begin{coro}\label{coro:AbelianSchur}
		Let $\mathcal{E}$ be a length exact $d$-category. Then the following are equivalent:
		\begin{enumerate}
			\item\label{AbelianCond} $\mathcal{E}$ is an abelian $d$-truncated dg category,
			\item\label{SchurCond2} $\Simples(\dis{\E})$ is a $d$-semibrick,
			\item\label{SchurCond1} $\Simples(\dis{\E})$ is a semibrick in the exact category $H^0(\dis{\E})$.
		\end{enumerate}
	\end{coro}
	
	\begin{proof}
		Let $\DL\coloneqq \Simples(\dis{\E})$. By definition $\Filt_d(\DL)=\E$. Thus by \cref{theorem:Enomoto}, $\E$ is a length wide $d$-subcategory of $\E$ if and only if $\DL$ is a $d$-semibrick. However, the first two conditions of a $d$-semibrick are automatic by the assumption that $\E$ is a length exact $d$-category and $\DL\subseteq\dis{\E}$. Hence, in \ref{SchurCond1}, it is enough to check that $\DL$ is a semibrick in $\dis{\E}$. Moreover, by definition, a length exact $d$-category is a wide $d$-subcategory of itself precisely when it is abelian $d$-truncated.
	\end{proof}
	
	As a corollary of the proof above we obtain an explicit description of the $t$-structure on $\D^b_{dg}(\A)$ that lets us identify $\A$ and $\D^b_{dg}(\A)^{(-d,0]}$ in the case of length abelian $d$-truncated dg categories.
	
	\begin{coro}\label{AbelianEmbedding}
		Let $\A$ be a length abelian $d$-truncated dg category. Then $\Simples(\dis{\A})\subseteq \D^b_{dg}(\A)$ is a simple-minded collection and the universal embedding $\A\to \D^b_{dg}(\A)$ induces a quasi-equivalence
		\[\A\stackrel{\simeq}{\longrightarrow}\D^b_{dg}(\A)^{(-d,0]}\]
		for the corresponding $t$-structure, which we denote by \definef{$t_{\Simples}$}.
	\end{coro}
	
	\begin{proof}
		Since the universal embedding induces a quasi-equivalence 
		\[\tau^{\leq 0}\A(X,Y)\stackrel{\simeq}{\to} \tau^{\leq 0}\D^b_{dg}(\A)(X,Y) \text{ for } X,Y\in\A,\] \cref{coro:AbelianSchur} implies that $\Simples(\dis{\A})\subseteq\D^b_{dg}(\A)$ is a simple-minded collection and thus the universal embedding induces a quasi-equivalence
		\[\A\stackrel{\simeq}{\longrightarrow} \Filt_d(\Simples(\dis{\A}))=\heart_{t_{\Simples}}[d-1]\ast\heart_{t_{\Simples}}[d-2]\ast \cdots \ast \heart_{t_{\Simples}}=\D^b_{dg}(\A)^{(-d,0]}. \qedhere\] 
	\end{proof}

	\begin{ex}
		We have seen in \cref{example:D-ComplicialExact} that every abelian $d$-truncated dg category is $d$-filtered. This is not the case for exact $d$-truncated dg categories. One possible explanation for this defect is that exact $m$-truncated dg categories are exact $n$-truncated dg categories for every $m\geq n$, while this is never the case for an abelian $m$-truncated dg category that is not the $0$ category.
		
		Consider for example any exact category $\E$. We can view $\E$ as an exact $1$-category. Then 
		\[\dis{\E}[0,\ldots,1-1]=\dis{\E}\subseteq \E\]
		and hence $\E$ is $1$-filtered. However, $\E$ is an exact $d$-truncated dg category for any $d\geq 1$, but for $d>1$ we have
		\[\left\{L[n]\in\D^b(\E):L\in\Simples(\dis{\E}), 0\leq n<d\right\}\nsubseteq \E,\]
		since $\Hom_{\D^b(\E)}(L[n],\E)\cong\Hom_{\D^b(\E)}(L,\E[-n])=0$. Thus $\E$ is not $d$-filtered for $d>1$.
	\end{ex}
	
	\begin{ex}\label{example:Simples}
		\leaveout{The following two examples illustrate that, even for abelian $d$-categories, suspensions of objects in $\DL$ might be simple in the extriangulated sense, but do not have to.
		
		$\D^b_{dg}(k)^{(-d,0]}$: all shifts of simples are simple, since the only triangles involving $k$ are up to rotation given by
		\[k\to k\oplus k \to k \to k[1], \quad k\to 0 \to k[1] \to k[1]\]
		
		$\D^b_{dg}(k\Tilde{A}_3/\rad^2)^{(-d,0]}$ where $\Tilde{A}_2$ is the cyclic Nakayama algebra with $2$ simples: In this second example, no positive shift of a simple is simple in the extriangulated sense. The reason are the following conflations in $\D^b_{dg}(k\Tilde{A}_3/\rad^2)^{(-d,0]}$ for every $0<i<d$:
		\[2[i-1]\infl 1[i] \defl \begin{smallmatrix} 2\\1 \end{smallmatrix}[i] \ext, \quad 
		1[i-1] \infl 2[i] \defl \begin{smallmatrix} 1\\2 \end{smallmatrix} \ext.\]}
		
		The following illustrates why the notion of a (sub-object) simple object in an extriangulated category introduced in \cite[Def.~3.1]{BHST24} is not suited for our purposes.
		
		Let $\C$ be any pretriangulated dg category with a $t$-structure $t$ and $d>1$. Let $S\in\heart_t$, possibly a simple object. Then 
		\[S\oplus S \stackrel{\pi_2}{\infl} S \defl S[1] \ext\]
		is a conflation in $\E\coloneqq\C^{(-d,0]}$ and the first morphism (the projection to the second summand) is neither $0$ nor an isomorphism. This means that the simple objects of $\dis{\E}$ are not simple in the sense of \cite[Def.~3.1]{BHST24}.
	\end{ex}

	\section{Length abelian $d$-truncated dg categories with enough projectives}
	
	In what follows, $\A$ denotes a length abelian $d$-truncated dg category with finitely many simples $\Simples\coloneqq\Simples(\dis{\A})$, and $H^0(\A)$ its homotopy category with its canonical extriangulated structure. By \cref{AbelianEmbedding}, we can and we will identify $\A$ with $\D^b_{dg}(\A)^{(-d,0]}\eqqcolon \Hc{\A}{d}$. 
	
	Recall from section \ref{sec:Realisation} that we call an object $P\in\A$ \definef{projective} if $P$ is projective in the extriangulated category $H^0(\A)$, that is $\EE(P,-)=0$. We denote by $\define{\Proj(\A)}\subseteq \A$ the subcategory of projective objects in in $\A$. We say that \definef{$\A$ has enough projectives}, if the extriangulated category $H^0(\A)$ has enough projectives, that is for all $X\in H^0(\A)$ there exists $P\in\Proj(\A)$ and a deflation 
	\[P\defl X.\]
	We say that $P\in\Proj(\A)$ is a projective generator if for all $X\in\A$ there exists a deflation of the form $P^n\defl X$ for some $n\in\NN$.
	
	The goal of this chapter is to prove an analogue of the classical statement that a length abelian category with finitely many simples is equivalent to the module category of a finite-dimensional algebra if and only if it has enough projectives if and only if it has globally bounded Loewy-length (see \cite{Ga62,Ga73}). 
	
	The following theorem from \cite{F25} characterizes pretriangulated dg categories that are quasi-equivalent to $\fd_{dg}(A)$ for a connective, locally-finite dg algebra $A$. We explain how to restrict to proper connective dg algebras in order to characterize length abelian $d$-truncated dg categories with enough projectives.
	
	\begin{theorem}[{\cite[Thm.~F]{F25}}]
		A locally-finite pretriangulated dg category $\C$ is equivalent to $\fd_{dg}(A)$ for some locally-finite connective dg algebra $A$ if and only if there is a bounded $t$-structure on $\C$ with length heart which has a projective generator.
	\end{theorem}
	
	\begin{coro}\label{coro:d-simple-minded}
		Let $\C$ be a locally-finite pretriangulated dg category. Assume $\DL\subseteq \C$ is a finite $d$-complicial simple-minded collection and define
		\[L=\bigoplus_{L_i\in\DL}L_i, \quad E\coloneqq \C\left(L,L \right).\] 
		The following statements are equivalent:
		\begin{enumerate}
			\item\label{item:d-simple-minded1} $A\coloneqq \KoszulDual{E}$ is a proper connective dg algebra such that $H^*(A)$ is concentrated in degrees $(-d,0]$ and there is a $t$-exact quasi-equivalence
			\[(\Phi_A)^{-1}\circ \C(-,L)\colon\C \stackrel{\simeq}{\longrightarrow}\per_{dg}(E^{op})^{op}\stackrel{\simeq}{\longrightarrow}\fd_{dg}(A)\]
			with respect to $t_{\DL}$ and the standard $t$-structure respectively. 
			\item\label{item:d-simple-minded2} The heart $\heart_{t_{\DL}}$ has a projective generator.
		\end{enumerate}
	\end{coro}

	\begin{proof}
		The implication \ref{item:d-simple-minded1}$\Rightarrow$\ref{item:d-simple-minded2} holds since the heart of the standard $t$-structure is equivalent to $\Mod(H^0(A))$. For the reverse implication we recall that it is shown in the proof of \cite[Thm.~5.2]{F25} that there is an equivalence of pretriangulated dg categories
		\[\C(-,\DL):\C\stackrel{\simeq}{\longrightarrow}\per_{dg}(E^{op})^{op} \]
		The above exact functor sends the simple-minded collection $\DL$ to the simple-minded collection given by the direct summands of $E^{op}$, and thus is a $t$-exact equivalence. Hence, $E$ satisfies the assumptions of \cref{theorem:KoszulD} and, for $A=\KoszulDual{E}$, we obtain a $t$-exact quasi-equivalence
		\[(\Phi_A)^{-1}:\per_{dg}(E^{op})^{op}\stackrel{\simeq}{\longrightarrow}\fd_{dg}(A). \qedhere\]
	\end{proof}
	
	Following \cite[\S7.2.2]{Lurie17}, for a triangulated category $\KST$ and a $t$-structure $t$ on $\KST$, we call an object $P\in\KST$ \definef{derived projective} (with respect to $t$), if $P\in\KST^{\leq 0}$ and $\Hom_{\KST}(P,X[1])=0$ for all $X\in\KST^{\leq 0}$. We denote by \definef{$\Dproj_t(\KST)$} the full subcategory of derived projective objects with respect to $t$.
	
	\begin{lemma}\label{lemma:DerivedProj}
		Let $P\in\Proj(\A)$. Then $P\in\Dproj_{t_{\Simples}}(\D^b(\A))$.
	\end{lemma}
	
	\begin{proof}
		Consider $L\in\Simples$ and $n\in\mathbb{N}_{>0}$. Then, since $P$ is projective, it holds by \cref{HigherExt} and the definition of $\EE^n$ that
		\[\Hom_{\D^b(\A)}(P,L[n])\cong \EE^n_{\A}(P,L)=0.\]
		As $\D^b(A)^{<0}=\Filt(\{L[n]:L\in\Simples,n\in\mathbb{N}_{>0}\})$ the claim follows.
	\end{proof}
	
	\begin{coro}\label{prop:EnoughProjHeart}
		Let $\A$ be a length abelian $d$-truncated dg category with finitely many simples that is locally-finite and $\Ext$-finite. Then, $\A$ has a projective generator if and only if $\dis{\A}$ has a projective generator.
	\end{coro}
	
	\begin{proof}
		Assume first that $P\in\A$ is a projective generator. Let $t\coloneqq t_{\Simples}$ and recall from \cref{example:D-ComplicialExact} that we can identify $\A\simeq \Hc{\A}{d}$ and $\dis{\A}\simeq \heart_{t}$. By \cref{lemma:DerivedProj} it follows that $P\in\Dproj_t(\D^b(\A))$ and hence, by \cite[Lem.~3.3]{Bon24}, $H^0_t(P)\in\heart_{t}$ is projective. Let now $X\in\heart_t$. By assumption, there exists a deflation $P^n\defl X$ in $\Hc{\A}{d}$. Thus $H^0_t(f):H^0_t(P)^n\to H^0_t(X)\cong X$ is an epimorphism. This shows that $H^0_t(P)$ is a projective generator of $\heart_t$. 
		
		Assume now $\heart_{t}$ has a projective generator. By \cref{AbelianEmbedding}, $\Simples$ is a $(d-1)$-complicial simple-minded collection. Thus \cref{coro:d-simple-minded} implies that there exists a $t$-exact quasi-equivalence $\D^b_{dg}(\A)\simeq\fd_{dg}(A)$. In particular, it holds $\A\simeq\fd_{dg}(A)^{(-d,0]}$. Moreover, $A\in\fd(A)^{(-d,0]}$ is projective and, by \cite[Lem.~3.2]{Bon24}, for $X\in\fd(A)^{(-d,0]}$,
		\begin{equation}\label{eq:H0iso1}
			H^0:\Hom_{\fd(A)}(A,X)\stackrel{\cong}{\longrightarrow} \Hom_{\fd(A)}(H^0(A),H^0(X)).
		\end{equation}
		Since in $\Mod(H^0(A))\simeq\heart\subseteq\fd(A)$, there exists an epimorphism $H^0(A)^n\defl H^0(X)$ for some $n\in\mathbb{N}$, the isomorphism \eqref{eq:H0iso1} yields that $A\in\fd(A)^{(-d,0]}$ is a projective generator. Hence the preimage of $A$ under the above equivalence is a projective generator of $\A$. 
	\end{proof}
	
	\begin{coro}\label{prop:EnoughProjModA}
		Let $\A$ be a length abelian $d$-truncated dg category with finitely many simples that is locally-finite and $\Ext$-finite. Let $P\in\A$ be an object and $A\coloneqq \A(P,P)$. Then, the following statements are equivalent:
		\begin{enumerate}
			\item\label{item:EnoughProj1} $P$ is a projective generator of $\A$.
			\item\label{item:EnoughProj3} There exists a $t$-exact quasi-equivalence
			\[\D^b_{dg}(\A)\stackrel{\simeq}{\longrightarrow}\fd_{dg}(A),\]
			with respect to $t_{\Simples}$ and the standard $t$-structure respectively. Moreover, this equivalence sends $P$ to $A$ and restricts to a quasi-equivalence $\A \stackrel{\simeq}{\to}\fd_{dg}(A)^{(-d,0]}$
			\item\label{item:EnoughProj4} There exists a quasi-equivalence
			\[\A\stackrel{\simeq}{\longrightarrow}\fd_{dg}(A)^{(-d,0]}\]
			that sends $P$ to $A$.
		\end{enumerate}
	\end{coro}
	
	\begin{proof}
		\ref{item:EnoughProj1}$\Rightarrow$\ref{item:EnoughProj3}: As before, we define $t\coloneqq t_{\Simples}$ and identify $\A\simeq \Hc{\A}{d}$ and $\dis{\A}\simeq \heart_{t}$. By (the proof of) \cref{prop:EnoughProjHeart} $H^0_t(P)$ is a projective generator of $\heart_t$. Define $L\coloneqq \oplus_{S\in\Simples}S$. We might assume that $\topf(H^0_t(P))=L$, otherwise we can replace $\Simples$ by the direct summands of $\topf(H^0_t(P))$, which does not change the $t$-structure $t_{\Simples}$ since $H^0_t(P)$ is a projective generator of $\heart_t$. By \cref{AbelianEmbedding}, $\Simples$ is a $(d-1)$-complicial simple-minded collection. Thus, since $\heart_t$ has a projective generator, applying \cref{coro:d-simple-minded} to $\C\coloneqq\D^b_{dg}(\A)$ yields the $t$-exact quasi-equivalence 
		\[\varphi\coloneqq (\Phi_{\KoszulDual{E}})^{-1}\circ \C(-,L)\colon\D^b_{dg}(\A) \stackrel{\simeq}{\longrightarrow}\per_{dg}(E^{op})^{op}\stackrel{\simeq}{\longrightarrow}\fd_{dg}(\KoszulDual{E}), \quad E\coloneqq \C(L,L).\]
		It holds $\Hom_{\D^b(\A)}(P,L[n])=0$ if $n\neq 0$. For $n<0$ this holds because $P\in\D^b(\A)^{\leq 0}$ and $L\in\heart_t$, for $n>0$ this follows by \cref{lemma:DerivedProj}. Hence it holds 
		\begin{align*}
			\C(P,L) & \simeq H^0(\C)(P,L)=\Hom_{\D^b(\A)}(P,L)\cong \Hom_{\heart_t}(H^0_t(P),L) \\ 
			& \cong \Hom_{\heart_t}(\topf(H^0_t(P)),L)\cong \Hom_{\heart_t}(L,L)= H^0(\C)(L,L)\cong S_{E^{op}}.
		\end{align*}
		Thus $\C(P,L)\simeq S_{E^{op}}$ and since $(\Phi_{\KoszulDual{E}})^{-1}(S_{E^{op}})\simeq \KoszulDual{E}$ by \cref{Thm:Fushimi_4.4}, it follows that $\varphi(P)\simeq \KoszulDual{E}$. Hence, we compute 
		\[A=\A(P,P)\simeq\D^b_{dg}(\A)(P,P)\simeq \fd_{dg}(\KoszulDual{E})(\KoszulDual{E},\KoszulDual{E}) \simeq \KoszulDual{E}.\]
		Hence $A\simeq \KoszulDual{E}$ and \ref{item:EnoughProj3} holds.
		 
		\ref{item:EnoughProj3}$\Rightarrow$\ref{item:EnoughProj4}: This is clear.
		 
		\ref{item:EnoughProj4}$\Rightarrow$\ref{item:EnoughProj1}: As we have seen in the proof of \cref{prop:EnoughProjHeart}, $A\in\fd(A)^{(-d,0]}$ is a projective generator, hence the same holds for $P$.
	\end{proof}
	
	The following recognition theorem gives a characterization of algebraic triangulated categories of the form $\fd(A)$, for $A$ a proper connective dg algebra with cohomology concentrated in degrees $1-d$ to $0$. In particular, for $d=1$, this gives a characterization of algebraic triangulated categories of the form $\D^b(\Lambda)$ for $\Lambda$ a finite-dimensional algebra.  
	
	\begin{theorem}\label{theorem:RecognitionTheorem}
		Let $\C$ be a locally-finite pretriangulated dg category such that the following conditions are satisfied:
		\begin{enumerate}
			\item There exists a collection $\DL\coloneqq \{L_1,\ldots,L_n\}\subseteq \C$ which is a simple-minded collection in $\thick(\DL)$. In particular, $\DL$ yields a bounded $t$-structure $t_{\DL}$ on $\thick(\DL)$ which we denote by $(\thick(\DL)^{\leq 0},\thick(\DL)^{\geq 0})$.
			\item\label{item:assumption2} There exists $d\in\NN_{\geq 1}$ and indecomposable objects $\{P_1,\ldots,P_n\}\subseteq H^0(\C)$
			such that for all $1\leq i\leq n$ there exists a triangle
			\[Y_i\longrightarrow P_i \longrightarrow L_i \longrightarrow Y_i[1]\]
			in $H^0(\C)$ such that 
			\[Y_i\in\thick(\DL)^{(-d,0]}=\Filt(\DL)[d-1]\ast\ldots\ast\Filt(\DL)[1]\ast\Filt(\DL).\]
			\item\label{item:assumption3} For $P\coloneqq \oplus_{i=1}^n P_i$ and $L\coloneqq\oplus_{i=1}^n L_i$ it holds $\Hom_{H^0(\C)}(P,L[k])=0$ for all $k\geq 1$.
		\end{enumerate} 
		Then, the realization functor 
		\[\real:\D^b_{dg}(\thick(\DL)^{(-d,0]})\stackrel{\simeq}{\longrightarrow} \thick(\DL)\subseteq \C\]
		is a $t$-exact quasi-equivalence. Moreover, composing this equivalence with the equivalence from \cref{prop:EnoughProjModA}, for $A\coloneqq\C(P,P)$ we obtain a $t$-exact quasi-equivalence
		\[\begin{tikzcd}
			& {\D^b_{dg}(\thick(\DL)^{(-d,0]})} & \\
			{\fd_{dg}(A)} && {\thick(\DL)\subseteq\C} \\
			{\per_{dg}(A)} && {\thick(P)}
			\arrow["\simeq"', from=1-2, to=2-1]
			\arrow["\real", from=1-2, to=2-3]
			\arrow["\simeq", dashed, from=2-1, to=2-3]
			\arrow[hook, from=3-1, to=2-1]
			\arrow["\simeq", dashed, from=3-1, to=3-3]
			\arrow[hook,from=3-3, to=2-3]
		\end{tikzcd}\]
		which maps $A$ to $P$. In particular, $A\in\fd(A)^{(-d,0]}$ is proper connective and $t_{\DL}$ agrees with the $t$-structure defined by the silting collection $\add(P)$. Moreover, if $\DL\subseteq \C$ is full, that is $\thick(\DL)=\C$, $\fd_{dg}(A)$ and $\C$ are equivalent as pretriangulated dg categories equipped with $t$-structures.
	\end{theorem}
	
	\begin{proof}
		Since $\DL$ is a simple-minded collection in $\thick(\DL)$, it holds  
		\[\thick(\DL)^{(-d,0]}=\Filt(\DL)[d-1]\ast\ldots\ast\Filt(\DL)[1]\ast\Filt(\DL),\]
		which is an extension-closed. Thus, assumption \ref{item:assumption2} implies that $P_i\in\thick(\DL)^{(-d,0]}$ and assumption \ref{item:assumption3} that it is projective for all $1\leq i\leq n$.
		By the Horseshoe Lemma for extriangulated categories \cite[Lem.~3.5]{H25}, $\thick(\DL)^{(-d,0]}$ has enough projectives if and only if for all $L_i$ and $0\leq j\leq d-1$ there exists $P'\in\Proj(\thick(\DL)^{(-d,0]})$ and a deflation $P'\defl L[j]$. If $j>0$ then $0\defl L[j]$ is such a deflation, for $j=0$ this is assumption \ref{item:assumption2}. Moreover, this shows that for every object $X$ there exists a triangle of the form
		\[K\longrightarrow \bigoplus_{i=1}^nP_i^{m_i}\stackrel{g}{\longrightarrow} X \longrightarrow K[1]\]
		for some $m_i\in\NN$ and $K\in \thick(\DL)^{(-d,0]}$. In particular, for $X\in\Proj(\thick(\DL)^{(-d,0]})$, the long exact sequence obtained from the above triangle by applying $\Hom_{H^0(\C)}(X,-)$ 
		\[\ldots \to \Hom_{H^0(\C)}(X,\bigoplus_{i=1}^nP_i^{m_i})\stackrel{g\circ-}{\longrightarrow} \Hom_{H^0(\C)}(X,X) \to \Hom_{H^0(\C)}(X,K[1])\cong \EE(X,K)=0 \]
		implies that there exists $f:X\to \bigoplus_{i=1}^nP_i^{m_i}$ such that $gf=\id_X$. That is, $X$ is a summand of $\bigoplus_{i=1}^nP_i^{m_i}$ and thus $\Proj(\thick(\DL)^{(-d,0]})=\add(P)$.
		
		Thus, by \cref{coro:RealisationEnoughProj}, it is enough to show that $\Hom_{H^0(\C)}(P,X[k])=0$ for all $k\geq 2$ and $X\in \thick(\DL)^{(-d,0]}$. Again, since $\thick(\DL)^{(-d,0]}=\Filt(\DL)[d-1]\ast\ldots\ast\Filt(\DL)$, it is enough to show $\Hom_{H^0(\C)}(P,L[k])=0$ for $k\geq 2$, which is assumption \ref{item:assumption3}.
		
		That the equivalence is $t$-exact follows from the fact that the $t$-structure on $\D^b_{dg}(\thick(\DL)^{(-d,0]})$ is defined by the simple-minded collection $\DL$. Finally, applying \cref{prop:EnoughProjModA} to $\A=\thick(\DL)^{(-d,0]}$ yields a $t$-exact equivalence $\D^b_{dg}(\A)\to\fd_{dg}(A)$ which sends $P$ to $A$. Hence, the final claims follow immediately from the fact that $P\in\thick(\DL)^{(-d,0]}$ and the standard $t$-structure on $\fd(A)$ can be described as the $t$-structure for the silting collection $\add(A)$.
	\end{proof}
	
	\begin{rema}\label{rema:ExplanationAMY}
		Since $(\per(A),\fd(A),A)$ is an ST-triple by \cite[Lem.~4.13]{AMY19} (see \cite[Def.~4.3]{AMY19} for the definition), \cref{theorem:RecognitionTheorem} proves that $(\thick(P),\thick(\DL),P)$ defines an ST-triple. 
		
		We also remark that, for $d=1$, it is enough to know that $(\thick(P),\thick(\DL),P)$ defines an ST-triple to deduce the above theorem from \cite[Prop.~6.3]{AMY19}. Indeed, in that case $P$ is a tilting object in $\thick(P)$, and hence one can apply \cite[Prop.~6.3]{AMY19} to deduce that the ST-triple $(\thick(P),\thick(\DL),P)$ is equivalent to $(\mathrm{K}^b(\mathrm{proj}(\Lambda)),\D^b(\Lambda),\Lambda)$ for $\Lambda=\End_{H^0(\C)}(P)$.
	\end{rema}
	
	\begin{coro}\label{coro:truncatedHearts}
		Let $A$ be a locally-finite connective dg algebra. Define the dg algebra 
		\[A_d=\End_{\fd_{dg}(A)}\left(\trunc^{>-d}\circ\trunc^{\leq 0}A\right).\] 
		Then it holds
		\[\fd_{dg}(A)^{(-d,0]}\xrightarrow{\simeq}\fd_{dg}(A_d)^{(-d,0]}.\]
	\end{coro}
	
	\begin{proof}
		This follows by applying \cref{prop:EnoughProjModA} to $\A=\fd_{dg}(A)^{(-d,0]}$ and $P=\trunc^{>-d}\circ\trunc^{\leq 0}A$.
	\end{proof}
	
	We now define the Loewy length of an object in $\A$ in analogy with the definition of levels in triangulated categories introduced in \cite[\S2.3]{ABIM10}. We denote by $\define{\add^{\Sigma}_d(\Simples)}\subseteq \A$ the subcategory containing all direct sums of objects of the form $L[i]$ with $L\in\Simples$ and $0\leq i\leq d-1$. Moreover for a subcategory $\C\subset H^0(\A)$ we denote by $\define{\smd(C)}$ the intersection of all full subcategories of $H^0(\A)$ which contain $C$ and are closed under isomorphisms and direct summands. We define
	\[\define{\thick^n_d(\Simples)}\coloneqq \smd(\add^{\Sigma}_d(\Simples)^{\ast n}).\]
	
	We define the \definef{Loewy length} of $M\in\A$ as
	\[\Loewy_{\A}(M)\coloneqq \inf\{n\in\mathbb{N}:M\in\thick^n_{d}(\Simples)\}.\]
	Moreover, we define the \definef{height} of $\A$ as
	\[\define{\height(\A)}\coloneqq \sup\{\Loewy_{\A}(M):M\in\A\}.\]
	It is clear from the definition that the above is the correct notion to measure the height of $\A$ in the sense that it is the minimal amount of steps it takes to build all objects by shifts of objects in $\Simples$. There is, however, an alternative, more refined definition of Loewy length that we primarily work with. We shall see that the two notion agree in the case $d=1$ and, moreover for arbitrary $d$, they are equivalent in the sense of \cref{lemma:BasicPropLoewy} \ref{item:EquivalentLoewy}.
	
	We define the \definef{top} of $X\in\A$ by
	\[\define{\topd(X)}\coloneqq \max\{-d+1\leq i\leq 0:H^{i}(X)\neq 0\}, \quad \define{\topf(X)}\coloneqq \topf(H^{\topd(X)}(X)),\]
	where $\topf(H^{\topd(X)}(X))$ is the top of $H^{\topd(X)}(X)$ in the abelian category $\dis{\A}$. It follows that there exists a conflation
	\[\rad(X)\stackrel{\iota_X}{\infl} X \stackrel{\pi_X}{\defl}\topf(X)\ext,\]
	which determines $\rad(X)$ uniquely up to isomorphism. We call $\rad(X)$ the \definef{radical} of $X$. Moreover, we define 
	\[\define{\rad^0(X)}\coloneqq X \text{ and } \define{\rad^n(X)}\coloneqq \rad(\rad^{n-1}(X)) \text{ for }n\geq 1.\] 
	
	Finally, we define the \definef{big Loewy length} of an object $X\in\A$ by
	\[\define{\BigLoewy_{\A}(X)}\coloneqq\inf\left\{n\in\mathbb{N}:\rad^n(X)=0\right\}.\]

	The following lemma computes the cohomology of $\rad(X)$.
	
	\begin{lemma}\label{lemma:RadComputation}
		Let $X\in\A$. Then the morphism $\iota_X$ induces the following isomorphisms 
		\[ 
		H^{i}(\rad(X))\cong 
		\begin{cases} 0 & \text{for } i>\topd(X), \\ 
			\rad(H^{i}(X)) & \text{for } i=\topd(X), \\ 
			H^{i}(X) & \text{for } i<\topd(X) 
		\end{cases}
		\]
	\end{lemma}
	
	\begin{proof}
		Let $c\coloneqq\topd(X)$ and consider the triangle
		\[\rad(X)\stackrel{\iota_X}{\longrightarrow} X \stackrel{\pi_X}{\longrightarrow} \topf(X)\to \rad(X)[1]\]
		to obtain the following long exact sequence in cohomology:
		\[\begin{tikzcd}[column sep=small,row sep=small]
 			\cdots\rar&0\rar& H^{i}(\rad(X))\rar[hook, two heads] & H^{i}(X)\rar & 0 \rar & \cdots \ar[out=-10,in=170,overlay]{dlllll}\\
			 0\rar & H^{c}(\rad(X))\rar[hook] & H^{c}(X)\rar[two heads] &\topf(H^{c}(X))\rar{0} \rar & H^{c+1}(\rad(X)) \rar & 0 \ar[out=-10,in=170,overlay]{dlllll} \\
			 \cdots\rar & 0\rar & H^{j}(\rad(X))\rar & 0 \rar & \cdots
		\end{tikzcd}\]
		The claim follows.
	\end{proof}
	
	The following lemma explains the connection between $\BigLoewy_{\A}$ and $\Loewy_A$ as well as how to compute $\BigLoewy_{\A}$ in terms of the Loewy length in $\dis{\A}$.
	
	\begin{lemma}\label{lemma:BasicPropLoewy}
		The following statements hold.
		\begin{enumerate}
			\item\label{item:DiscreteLoewy} Let $0\leq i\leq d-1$ and $M\in\dis{\A}$. Then $\BigLoewy_{\A}(M[i])=\Loewy_{\A}(M[i])=\Loewy_{\heart}(M)$.
			\item\label{item:LoewyIneq} $\Loewy_{\A}(X)\geq \max\{\Loewy_{\heart}(H^{i}(X)):-d+1\leq i\leq 0\}$.
			\item\label{item:BigLoewyComp} $\BigLoewy_{\A}(X)=\sum_{i=-d+1}^0\Loewy_{\heart}(H^{i}(X))$.
			\item\label{item:EquivalentLoewy} For all $X\in\A$, it holds $\Loewy_{\A}(X)\leq \BigLoewy_{\A}(X)\leq d\cdot \Loewy_{\A}(X)$.
		\end{enumerate}
	\end{lemma}
	
	\begin{proof}
		\ref{item:DiscreteLoewy} follows since $\dis{\A}\simeq \dis{\A}[1]$ and $\topf(M)$ is the largest semisimple quotient of $M$.
		
		\ref{item:LoewyIneq} can be proven exactly as \cite[Thm.~6.2.(2)]{ABIM10}. The idea is to perform an induction over $\Loewy_{\A}(X)$, write $X$ as an extension of objects of smaller Loewy length, pass to cohomology and use induction as well as subadditivity for $\Loewy_{\heart}$.
		
		\ref{item:BigLoewyComp} follows from \cref{lemma:RadComputation} and \ref{item:DiscreteLoewy}.
		
		The first inequality in \ref{item:EquivalentLoewy} follows from the fact that $\topf(X)\in\add(\Simples[i])$ for some $0\leq i< d$. The second inequality holds by combining \ref{item:BigLoewyComp} and \ref{item:LoewyIneq}.
	\end{proof}
	
	\begin{rema}
		The bounds in \ref{item:EquivalentLoewy} are sharp and $\Loewy_{\A}(X)=\BigLoewy_{\A}(X)$ holds if and only if $X\cong H^{i}(X)$ for some $-d+1\leq i\leq 0$.
	\end{rema}

	\begin{coro}\label{coro:EnoughProjLoewy}
		Let $\A$ be a length abelian $d$-truncated dg category with finitely many simples, and assume that $\A$ is locally-finite and $\Ext$-finite. Then, the following statements are equivalent: 
		\begin{enumerate}
			\item\label{item:EnoughProjLoewy1} $H^0(\A)$ has enough projectives,
			\item\label{item:EnoughProjLoewy2} $\height(\A)<\infty$,
			\item\label{item:EnoughProjLoewy3} $\height(\dis{\A})<\infty$.
		\end{enumerate}
	\end{coro}
	
	\begin{proof}
		\ref{item:EnoughProjLoewy1}$\Rightarrow$\ref{item:EnoughProjLoewy2}: Let $M\in\A$ arbitrary. If $\A=\Hc{\A}{d}$ has enough projectives, then $\A\cong\Hc{A}{d}$ for the proper connective dg algebra $A$ constructed in \cref{prop:EnoughProjModA}. Hence $\heart=\Mod(H^0(A))$ has enough projectives as well and thus there exists $b\in\mathbb{N}$ such that $\height(\heart)\leq b$. We obtain 
		\[\Loewy_{\A}(M)\leq \sum_{i=-d+1}^0 \Loewy_{\heart}(H^{i}(M))\leq d\cdot b.\]
		
		\ref{item:EnoughProjLoewy2}$\Rightarrow$\ref{item:EnoughProjLoewy3}: This follows immediately from statement \ref{item:DiscreteLoewy} in \cref{lemma:BasicPropLoewy}.
		
		\ref{item:EnoughProjLoewy3}$\Rightarrow$\ref{item:EnoughProjLoewy1}: $\dis{\A}\simeq \heart_{t_{\Simples}}$ is an abelian category. Hence $\height(\dis{\A})<\infty$ implies that $\dis{\A}$ has enough projectives by \cite{Ga62,Ga73} (see also \cite[Thm.~4.6]{BB25}). Now the claim follows from \cref{prop:EnoughProjHeart}.
	\end{proof}
	
	The following theorem summarises our findings from this section. 
	\begin{theorem}\label{prop:EnoughProj}
		Let $\A$ be a length abelian $d$-truncated dg category with finitely many simples, and assume that $\A$ is locally-finite and $\Ext$-finite. Then, the following are equivalent:
		\begin{enumerate}
			\item\label{item:EnoughProjectives} $\A$ has enough projectives. 
			\item\label{item:ModA} $\A$ is quasi-equivalent to $\fd_{dg}(A)^{(-d,0]}$ for $A:=\End_{\A}(\bigoplus_{P_i\in\Proj(\A) \text{ indec.}}P_i)$.
			\item\label{item:FiniteLoewy} $\height(\A)$ is finite.
		\end{enumerate}
	\end{theorem}
	
	\begin{proof}
		The equivalence \ref{item:EnoughProjectives}$\Leftrightarrow$\ref{item:ModA} is \cref{prop:EnoughProjModA}. The equivalence \ref{item:EnoughProjectives}$\Leftrightarrow$\ref{item:FiniteLoewy} is \cref{coro:EnoughProjLoewy}. 
	\end{proof}
	
	Together with \cref{theorem:KoszulD} this yields the following theorem.
	
	\begin{theorem}\label{theorem:SummaryAbelian}
		Let $k$ be a perfect field. The following are in bijection:
		\begin{enumerate}
			\item\label{item:SummaryAbelian1} Smooth locally-finite pretriangulated dg categories $\C$ equipped with a strictly $(d-1)$-complicial bounded $t$-structure $t$ such that $\heart_t$ has enough projectives, up to $t$-exact quasi-equivalence.  
			\item\label{item:SummaryAbelian2} Length abelian $d$-truncated dg categories $\A$ with finitely many simples which are locally finite and $\Ext$-finite such that ${\height(\A)<\infty}$, up to exact quasi-equivalence.  
			\item\label{item:SummaryAbelian3} Proper basic dg algebras, that is $H^0(A)$ is basic, with cohomology strictly concentrated in degrees $(-d,0]$, up to quasi-isomorphism. 
			\item\label{item:SummaryAbelian4} Smooth locally-finite coconnective basic dg algebras, that is $H^0(E)$ is basic, such that $H^0(E)$ is semisimple and $\per(E)^{(-d,0]}\subseteq \per(E)$ is $1$-generated, up to quasi-isomorphism.
		\end{enumerate}
		Denoting by $L$ the direct sum of simple objects in $\C^{(-d,0]}$ or $\A$ and by $P$ the minimal right approximation of $L$ by projective objects in $\C^{(-d,0]}$ or $\A$, respectively, the bijections are given as follows (notice that the diagram commutes):
		
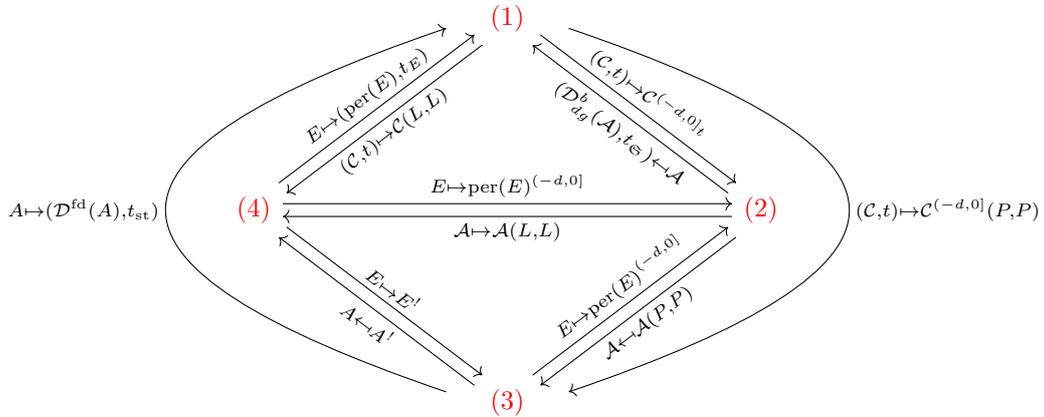
\[\begin{tikzcd}
			&&& {\ref{item:SummaryAbelian1}} \\
			\\
			\\
			{\ref{item:SummaryAbelian4}} &&&&&& {\ref{item:SummaryAbelian2}} \\
			\\
			\\
			&&& {\ref{item:SummaryAbelian3}}
			\arrow["{(\C,t)\mapsto \C(L,L)}"', sloped, shift left, from=1-4, to=4-1]
			\arrow["{(\C,t)\mapsto \C^{(-d,0]_t}}", sloped, shift left, from=1-4, to=4-7]
			\arrow["{(\C,t)\mapsto\C^{(-d,0]}(P,P)}", shift left=5, curve={height=-140pt}, from=1-4, to=7-4]
			\arrow["{E\mapsto(\per(E),t_E)}", sloped, shift left, from=4-1, to=1-4]
			\arrow["{E \mapsto \per(E)^{(-d,0]}}", shift left, from=4-1, to=4-7]
			\arrow["{E \mapsto \KoszulDual{E}}", sloped, shift left=1, from=4-1, to=7-4]
			\arrow["{\small{(\D^b_{dg}(\A),t_{\Simples})\mapsfrom \A}}"', sloped, shift left=1, from=4-7, to=1-4]
			\arrow["{\A\mapsto \A(L,L)}", shift left, from=4-7, to=4-1]
			\arrow["{\A\mapsfrom\A(P,P)}"', sloped, shift left, from=4-7, to=7-4]
			\arrow["{A\mapsto (\fd(A),t_{\mathrm{st}})}", shift left=5, curve={height=-140pt}, from=7-4, to=1-4]
			\arrow["{A\mapsfrom \KoszulDual{A}}"', sloped, shift left, from=7-4, to=4-1]
			\arrow["{E\mapsto\per(E)^{(-d,0]}}", sloped, shift left=1, from=7-4, to=4-7]
		\end{tikzcd}\]
	\end{theorem}

	We obtain the following corollaries for $\C=\D^b_{dg}(\Lambda)$.
	
	\begin{coro}
		Let $\Lambda$ a finite-dimensional algebra. Let $t_{\DP}$ be a $t$-structure associated to a finite silting collection $\DP$ and denote by $t_{\mathrm{st}}$ the standard $t$-structure on $\D^b(\Lambda)$. We denote by $\DL$ the simple objects in $\heart_t$ and define
		\[A\coloneqq\End_{\D^b_{dg}(\Lambda)}(\bigoplus_{P\in\DP}P), \quad E\coloneqq  \End_{\D^b_{dg}(\Lambda)}(\bigoplus_{L\in\DL}L)\]
		The following are equivalent: 
		\begin{enumerate}
			\item $\DP$ is a $d$-term silting collection, that is $\DP\subseteq\D^b(\Lambda)^{(-d,0]_{t_{\mathrm{st}}}}$,
			\item $t_{\DP}$ is  $(d-1)$-complicial.
			\item $\DL$ is a $d$-term simple-minded collection, that is $\DL\subseteq \D^b(\Lambda)^{(-d,0]_{t_{\mathrm{st}}}}$
		\end{enumerate} 
		Moreover, we have a $t$-exact equivalence $\D^b_{dg}(\Lambda)\simeq\fd_{dg}(A)$ with respect to the silting $t$-structure $t_{\DP}$ on the left and the standard $t$-structure on the right. 
	\end{coro}
	
	\begin{coro}
		Let $\Lambda$ be a finite-dimensional algebra and $\DL$ be a $d$-term simple-minded collection. Then $\per(\REnd_{\Lambda}(\DL))^{(-d,0]}\subseteq \per(\REnd_{\Lambda}(\DL))$ is $1$-generated.
	\end{coro}

	\section{Koszul Duality for length exact $d$-categories}\label{sec:KDExact}
	
	In this section we want to leverage our findings to prove a Koszul duality statement between certain length exact $d$-categories and certain coconnective dg categories. We fix $1\leq d\in\mathbb{N}$ and define for a coconnective dg category $E$
	\[\define{\tw(E)^{(-d,0]}}:=\Filt\left(\left\{L[k]:L\in E,0\leq k\leq d-1\right\}\right)\subseteq \tw(E).\] 
	We say that $E$ is \definef{admissible Schur} if the following condition holds: Let $x,y\in E$ and $f\in H^0(E)(x,y)$ an \definef{admissible morphism}, that is $f$ can be factored as $f=\iota\circ\pi$ with $\pi$ a deflation and $\iota$ an inflation in $\dis{\tw(E)^{(-d,0]}}$, then $f$ is an isomorphism or $0$. \leaveout{Finally, we define
	\begin{align*}
		\define{\mathrm{Ex}_d(\text{length})}&\coloneqq \{\E:\text{length exact } d\text{-category}\},\\
		\define{\mathrm{Ab}_d(\text{length})}&\coloneqq \{\A:\text{length abelian } d\text{-truncated dg category}\},\\
		\define{\mathrm{Ab}_d(\text{length},\text{proj.},\text{fin.})}&\coloneqq \{\A\in\mathrm{Ab}_d(\text{length}):\A \text{ has enough projectives, locally-finite}\},\\
		\define{\mathrm{Cat}^{cc}_{dg}}&\coloneqq \{E:\text{coconnective dg category}\},\\
		\define{\mathrm{Cat}^{cc}_{dg}(d,\text{ad.Schur})}&\coloneqq \{E\in \mathrm{Cat}^{cc}_{dg}: \text{generated in degrees } [0,d], \text{ admissible Schur}\},\\
		\define{\mathrm{Cat}^{cc}_{dg}(d,\text{Schur})}&\coloneqq \{E \in \mathrm{Cat}^{cc}_{dg}(d): H^0(E)(x,y) \text{ is a division-ring } \forall x,y\in E\},\\
		\define{\mathrm{Cat}^{cc}_{dg}(d,\text{Schur}, \text{smooth}, \text{fin.})}&\coloneqq \{E \in \mathrm{Cat}^{cc}_{dg}(d,\text{Schur}): \text{homologically smooth, locally-finite}\}.
	\end{align*}}
	
	Recall that, for every length exact $d$-category $\E$, the set $\define{\DL_{\E}}\coloneqq\Simples(\dis{\E})\subseteq \dis{\E}$ has the property that for all $L,L'\in\DL_{\E}$ and $n<0$ it holds 
	\[H^{n}(\D^b_{dg}(\E))(L,L')\cong H^n(\E)(L,L')=0.\]
	Hence, we can view \definef{$E\coloneqq \D^b_{dg}(\E)(\DL_{\E},\DL_{\E})$} as a coconnective dg category with one object for each object of $\DL_{\E}$.
	
	Conversely, for a coconnective dg category $E$, $\tw(E)^{(-d,0]}$ is an exact $d$-truncated dg category. Observe that $\tw(E)^{(-d,0]}$ will not be an interval in a $t$-structure in general (for example, if $d=1$ then $\tw(E)^{(-d,0]}$ is exact but not abelian in general). However, under additional assumptions, this is the prototypical example of a length exact $d$-category. The following lemma gives such conditions.

	\begin{lemma}\label{lemma:WellDef}
		Let $E$ be a coconnective dg category such that $\E\coloneqq\tw(E)^{(-d,0]}\subseteq \tw(E)$ is $1$-generated. Moreover, assume that $\dis{\E}=\Filt(\ob(E))$ and $E$ is admissible Schur. Then $\E$ is a length exact $d$-category.
	\end{lemma}
	
	\begin{proof}
		Firstly, $\E\coloneqq\tw(E)^{(-d,0]}$  is an extension closed subcategory of a pretriangulated dg category and thus an exact dg category. Moreover, it is $d$-truncated by a devissage argument that we include for the convenience of the reader. Consider  
		\[S_l\coloneqq \left\{M:H^n(\tw(E))(M,L[i])=0 \text{ for all } L\in E, 0\leq i\leq d-1, n\notin(-d,0] \right\}\subseteq \tw(E).\]
		Then $S_l$ is extension closed by the long exact sequence in cohomology we obtain when applying $\tw(E)(-,L[i])$ to a triangle. Moreover it holds $L[i]\in S_l$ for $0\leq i\leq d-1$ since $E$ is coconnective. It follows that $\E\subseteq S_l$. 
		
		Consider now 
		\[S_r\coloneqq \left\{N:H^n(\tw(E))(M,N)=0 \text{ for all } M\in S_l, n\notin(-d,0] \right\}\subseteq \tw(E).\]
		Analogously to the above one argues that $S_r$ is extension-closed. Moreover, by definition of $S_l$, $L[i]\in S_r$ for all $L\in E$ and $0\leq i\leq d-1$. Thus $\E\subseteq S_r$ and since $\E\subseteq S_l$, it follows from the definition of $S_r$ that $\E$ is $d$-truncated.
		
		By assumption it holds
		\[\dis{\E}=\Filt(\{L:L\in\ob(E)\}),\] 
		and hence it follows that $\E$ is $d$-filtered.
		\leaveout{The inclusion ``$\supseteq$'' clearly holds since $\ob(E)\subseteq\dis{\E}$ by coconnectivity. For the reverse inclusion consider $X\in\dis{\E}$. By definition of $\E$ there exists $X_1,\ldots,X_n=X\in\E$ and $L_1,\ldots,L_n\in\ob(E)$ as well as $0\leq r_1,\ldots,r_n <d$ and conflations
		\[0\infl X_1\defl L_1[r_1]\ext, \quad X_i\infl X_{i+1}\defl L_{i+1}[r_{i+1}]\ext \text{ for } i>1.\]
		In particular $X_1=L_1[r_1]$ and we obtain an inflation $L_1[r_1]\infl X$ by composition. By the Yoneda Lemma this inflation is non-zero (since $L_1[r_1]$ is non-zero). Since $X\in\dis{\E}$ it has to hold $r_1=0$. Assume now we have shown that $r_i=0$ for all $i<k$ for some fixed $k\leq n$. Then it holds $X_i\in\dis{\E}$ for $i<k$ as well. We wish to show that $r_{k}=0$ as well. Assume $r_k>0$. Since the conflations are obtained by triangles in $\tw(E)$ rotation yields a conflation 
		\[L_k[r_k-1]\infl X_k\defl X_{k+1}\ext.\]
		By the same argument as above it has to hold $r_k=1$ and}
		
		Finally, we show that $L\in\ob(E)$ is simple in $\dis{\E}$. Assume there was a conflation in $H^0(\dis{\E})$ of the form
		\[X\stackrel{\iota_1}{\infl} L {\defl} M \quad 0\neq X,M\in H^0(\dis{\E}).\]
		Since it holds $H^0(\dis{\E})=\Filt(\ob(E))$ and $H^0(\dis{\E})$ is exact, there exists a non-zero inflation $\iota_2:L'\infl X$ for some $L'\in\ob(E)$. But then $\iota\coloneqq \iota_1\iota_2:L'\to L$ is an admissible morphism (it is an inflation), hence by assumption $\iota=0$ or $\iota$ is an isomorphism. If $\iota$ was $0$, then $\iota_1$ is $0$ since it is a monomorphism (and $\iota_2$ is non-zero), which yields $X=0$, a contradiction to the assumption. If $\iota$ was an isomorphism, $\iota_1$ would be an isomorphism as well and hence $M=0$, again a contradiction to the assumption.
	\end{proof}

	\begin{theorem}\label{KoszulExact}
		There are mutually inverse bijections between the following: 
		\begin{enumerate}
			\item\label{item:KoszulExact1} Coconnective dg categories $E$ such that $\E\coloneqq\tw(E)^{(-d,0]}\subseteq \tw(E)$ is $1$-generated,  $\dis{\E}=\Filt(\ob(E))$ and $E$ is admissible Schur, up to quasi-equivalence.
			\item\label{item:KoszulExact2} Length exact $d$-categories, up to exact quasi-equivalence.
		\end{enumerate}
		They restrict to bijections between the following:
		\begin{enumerate}
			\item[(1')] Coconnective dg categories $E$ such that $\tw(E)^{(-d,0]}\subseteq \tw(E)$ is $1$-generated, and $H^0(E)(x,y)$ is a division algebra for all $x,y\in\ob(E)$, up to quasi-equivalence.
			\item[(2')] Length abelian $d$-truncated dg categories, up to exact quasi-equivalence.
		\end{enumerate}
		If $k$ is a perfect field, they restrict further to bijections between the following:
		\begin{enumerate}
			\item[(1'')] Coconnective locally-finite smooth dg categories $E$ with $|\ob(E)|<\infty$ such that $\tw(E)^{(-d,0]}\subseteq \tw(E)$ is $1$-generated, and $H^0(E)(x,y)$ is a division algebra for all $x,y\in\ob(E)$, up to quasi-equivalence.
			\item[(2'')] Locally-finite, $\Ext$-finite length abelian $d$-truncated dg categories with finitely many simples and enough projectives, up to exact quasi-equivalence.
		\end{enumerate}
		
		The bijections are given as follows:
		\begin{align*}
			F:\ref{item:KoszulExact1} &\longrightarrow \ref{item:KoszulExact2} & G:\ref{item:KoszulExact2} &\longrightarrow \ref{item:KoszulExact1}\\
			E &\longmapsto \tw(E)^{(-d,0]} & \E &\longmapsto \D^b_{dg}(\E)(\DL_{\E},\DL_{\E})
		\end{align*}
		Moreover, if $\E$ and $E$ correspond to each other under the above bijections, there is a quasi-equivalence \[\tw(E)\stackrel{\simeq}{\longrightarrow}\D^b_{dg}(\E).\]
	\end{theorem}
	
	\begin{proof}
		We first observe that once we show that we get a well-defined bijection of \ref{item:KoszulExact1} and \ref{item:KoszulExact2}, the other bijections follow from \cref{coro:AbelianSchur} and \cref{prop:EnoughProj} and \cref{theorem:KoszulD} respectively. 
		
		\cref{lemma:WellDef} proves that $F$ is well defined. In order to prove that $G$ is well-defined we consider $\E$ as in \ref{item:KoszulExact2} and $E=\D^b_{dg}(\E)(\DL_{\E},\DL_{\E})$. Since $\D^b_{dg}(\E)$ is a pretriangulated dg category and generated by $\DL_{\E}$, the universal property of $\tw(E)$ yields the following commutative diagram.
		\[\begin{tikzcd}
			\DL_{\E} \\
			\tw(E) && {\D^b_{dg}(\E)} \\
			{\tw(E)^{(-d,0]}} && {\Filt_d(\DL_{\E})=\E}
			\arrow[hook, from=1-1, to=2-1]
			\arrow[hook, from=1-1, to=2-3]
			\arrow["{\simeq}", dashed, from=2-1, to=2-3]
			\arrow[hook, from=3-1, to=2-1]
			\arrow["{\simeq}", dashed, from=3-1, to=3-3]
			\arrow[hook, from=3-3, to=2-3]
		\end{tikzcd}\]
		It follows that $H^0\tw(E)^{(-d,0]}\subseteq H^0\tw(E)$ is $1$-generated, since $H^0(\E)\subseteq \D^b(\E)$ is $1$-generated. Moreover, by applying \cite[Lem.~3.5]{HR20} to the exact category $H^0(\dis{\E})$ we see that $E$ is admissible Schur. This proves that $G$ is well defined.  The diagram above also proves that it holds $F\circ G=\id$ (up to exact quasi-equivalence).
		
		Consider now $E$ as in \ref{item:KoszulExact1}. Then, by \cref{GeneralisedRealFunctor}, the realization functor 
		\[\real:\D^b_{dg}(\tw(E)^{(-d,0]})\longrightarrow \tw(E)\] 
		is a quasi-equivalence that can be chosen to be the identity on $\mathrm{ob}(E)$ and therefore induces for all $L,L'\in\mathrm{ob}(E)$ a quasi-isomorphism
		\[\D^b_{dg}\left(\tw(E)^{(-d,0]}\right)(L,L')\stackrel{\real}{\simeq} \tw(E)(L,L')=E(L,L')\]
		which proves that $G\circ F=\id$ (up to quasi-equivalence).
	\end{proof}
	
	\leaveout{\begin{rema}
		It follows from the above that a $d$-filtered exact dg category $\E$ is a length exact $d$-category if and only if 
		\[\E=\Filt_d\left(\Simples(\dis{\E})\right).\]
	\end{rema}}
	
	\begin{rema}
		In the context of \cref{KoszulExact}, it would be interesting to replace the ``admissible Schur'' condition by a condition that is intrinsic to the dg category $E$. One difficulty arises as follows: for any finite-dimensional algebra $\Lambda$, $\add(\Lambda)\subseteq\Mod(\Lambda)$ is a (split-)exact category with simple objects the indecomposable direct summands of $\Lambda$ and it holds $\D^b(\add(\Lambda))(\DL,\DL)\cong \Lambda$. This means there are no restrictions on $H^0(E)$ (in this case no morphism is admissible). Thus, any such condition on $E$ could not only refer to $H^0(E)$. 
	\end{rema}
	
	\subsection*{Acknowledgments}
	The author thanks Prof. Gustavo Jasso for his guidance and many helpful discussions during the process of writing this article and moreover, for making him aware of the talk given at MPIM in Bonn where the proof of \cref{theorem:Enomoto} for the case $d=1$ was explained. The author would like to thank Isambard Goodbody for very helpful discussions about smoothness of dg algebras which led in particular to the comparison of the different homological finiteness conditions in \cref{section:KoszulDualityProper}.
	
	\subsection*{Financial support}
	The author's research was partly supported by the Swedish Research Council (Vetenskapsrådet) Research Project Grant 2022-03748 `Higher structures in higher-dimensional homological algebra.'
	
	\printbibliography
	
\end{document}